\documentclass[a4paper,11pt,reqno]{amsart}

\usepackage{amssymb}
\usepackage[all]{xy}

\usepackage{amsmath}
\usepackage{amsthm}    
\usepackage{mathrsfs}

\usepackage{hyperref}
\usepackage{xcolor}
\definecolor{paperblue}{cmyk}{1,0,0,0}
\hypersetup{
  colorlinks=true,
  linkcolor=paperblue,
  citecolor=magenta!80!black,
  urlcolor=paperblue,
  bookmarksdepth=2
}
\usepackage{enumitem}
\usepackage{subcaption}
\usepackage[final]{graphicx}
\usepackage{float}
\usepackage{epic}
\usepackage{setspace}

\usepackage[all]{xy}
\usepackage{color}
\usepackage{calligra}
\usepackage[T1]{fontenc}
\usepackage{verbatim}
\newcommand{\N}{\mathbb N}
\newcommand{\Z}{\mathbb Z}
\newcommand{\vN}{\mathrm L}
\newcommand{\tr}{\tau}

\newcommand{\normtwo}[1]{\left\lVert #1\right\rVert_2}

\newcount\mycount
\numberwithin{equation}{section}

\numberwithin{equation}{section}
\newtheorem{introtheorem}{Theorem}

\newtheorem{introcorollary}[introtheorem]{Corollary}

\newtheorem{theorem}{Theorem}[section]

\newtheorem{lemma}[theorem]{Lemma}

\newtheorem{proposition}[theorem]{Proposition}

\newtheorem{corollary}[theorem]{Corollary}

\theoremstyle{definition}
\newtheorem{definition}[theorem]{Definition}
\newtheorem*{definition*}{Definition}

\newtheorem{remark}[theorem]{Remark}

\theoremstyle{remark}

\theoremstyle{remark}

\newcommand{\Cs}{\ensuremath{\mathrm{C}^\ast}}
\newcommand{\Csr}{\mathrm{C}^\ast_{\rm r}}

\newcommand{\id}{\ensuremath{\mathrm{id}}}
\newcommand{\Prob}{\ensuremath{\mathrm{Prob}}}

\begin{document}

\title[Inner \Cs-selfless groups]{Inner \Cs-selfless groups}

\author{Tattwamasi Amrutam}
\address{Institute of Mathematics of the Polish Academy of Sciences,  ul. Sniadeckich 8,\newline 00-656, Warszawa, Poland}
\email{tattwamasiamrutam@gmail.com}

\author{Mario Klisse} \address{Department of Mathematics, Christian-Albrechts-Universität Kiel, Heinrich-Hecht-Platz  6, \newline 24118 Kiel, Germany} \curraddr{} \email{klisse@math.uni-kiel.de} \thanks{}

\author{Matteo Pagliero} \address{Department of Mathematics and Computer Science, University of Southern Denmark, Campusvej 55, 5230 Odense, Denmark} \curraddr{} \email{pagliero@imada.sdu.dk} \thanks{}



\date{\today}

\dedicatory{}

\begin{abstract}
We introduce the notion of inner \Cs-selfless groups and study their structural properties. We prove a Fourier coefficient extraction theorem for invariant subalgebras of crossed products by such groups under a compatibility assumption on the action. Consequently, every inner \Cs-selfless group has the relative \Cs-invariant subalgebra rigidity property, and every finite product of such groups has the \Cs-invariant subalgebra rigidity property.
We also introduce another group-theoretic notion, which we call cyclic Powers property. This implies inner \Cs-selflessness, complete \Cs-selflessness, and the Powers property, among others. We establish it for groups admitting a topologically free extreme boundary action on a compact Hausdorff space with more than two points, and for Zariski-dense discrete subgroups of $\mathrm{PSL}(d,\mathbb K)$ containing a biproximal element, where $\mathbb K$ is a local field and $d\geq2$. In particular, we obtain the Powers property for $\mathrm{PSL}(d,\mathbb Z)$, $d\geq3$, answering a question posed more than twenty years ago by Bridson--de la Harpe. Finally, we show that every non-trivial mixed identity-free discrete group has the relative invariant subalgebra rigidity property for its group von Neumann algebra. This shows that the Baumslag--Solitar groups $\mathrm{BS}(m,n)$ with $|m|,|n|\geq2$ and $|m|\ne|n|$ have the relative ISR property, thereby answering a question of Dudko--Jiang.
\end{abstract}

\maketitle

\setcounter{tocdepth}{1}
\begingroup
\makeatletter
\let\l@subsection\@gobbletwo
\let\l@subsubsection\@gobbletwo
\makeatother
\tableofcontents
\endgroup
\section{Introduction}
Largely motivated by the work of Dykema, R{\o}rdam, and Popa \cite{Pop95}, Robert introduced the notion of selfless \Cs-probability spaces in~\cite{Robert25}. Here, a \Cs-probability space is a unital \Cs-algebra equipped with a distinguished state whose associated GNS-representation is faithful. It follows from \cite{Robert25}, or from an analogous idea developed in \cite{BS16}, that if the associated state is tracial, selflessness implies simplicity, uniqueness of the tracial state, stable rank one, and strict comparison of positive elements, and, in the infinite-dimensional case, these conditions are
also sufficient \cite{Rob26}. Moreover, any two unital embeddings of the Jiang--Su algebra $\mathcal Z$ into such a \Cs-probability space are strongly asymptotically unitarily equivalent by \cite{Pag26}. Apart from providing a new way to view a number of regularity properties of \Cs-algebras, it turns out that selfless \Cs-probability spaces are rather common, and establishing new examples, especially non-nuclear ones, has been a successful collective endeavour. Apart from free product constructions \cite{HKEPR25,FFKP26a,GKPT26,GJKPR26}, one of the most prolific sources of selfless examples are reduced group \Cs-algebras, which motivated the following definition. A non-trivial discrete group $\Gamma$ is said to be \emph{\Cs-selfless} if the
\Cs-probability space $(\Csr(\Gamma),\tau_\Gamma)$ consisting of the reduced group \Cs-algebra of $\Gamma$ and its canonical tracial state is selfless. Amrutam--Gao--Kunnawalkam Elayavalli--Patchell~\cite{AGKP25} introduced a quantitative notion of selflessness for groups and used it, together with the rapid decay property, to establish \Cs-selflessness. Later, Ozawa~\cite{Oza} proved \Cs-selflessness for groups admitting a topologically free extreme boundary action, and more generally for groups satisfying a combinatorial condition that includes all acylindrically hyperbolic groups with trivial finite radical and all countable Zariski-dense subgroups of $\mathrm{PSL}(d,\mathbb R)$. The list of \Cs-selfless groups grew from there on the impetus of the aforementioned works \cite{AGKP25,Vigdorovich25,RTV25,MWY25a,MWY25b,FFKP26b,Vig26,GKPT26,Oml26a,FG26,LBLV26,Ari26, Oml26b,AF26}. 

One of the main aims of this paper is to  establish that several of these constructions yield more than \Cs-selflessness. To isolate this additional property, we call a non-trivial discrete group $\Gamma$ \emph{inner \Cs-selfless} if there exist an index set $I$, an ultrafilter $\mathcal U$ on $I$, a family of group elements $(g_i)_{i\in I}\subseteq\Gamma$, and a unital $\ast$-embedding
\[
\theta\colon C_{\mathrm r}^{\ast}(\Gamma\star\mathbb Z)\hookrightarrow \Cs_r(\Gamma)^{\mathcal U}
\]
whose restriction to $\Cs_r(\Gamma)$ is the canonical diagonal embedding into its ultrapower, and that sends the generator of $\mathbb Z$ to $(\lambda_{g_i})_{i\to\mathcal U}$. For countable groups, this is equivalent to requiring that $\Gamma\star\mathbb Z$ be existentially \Cs-residually-$\Gamma$ in the sense of \cite{AGKP25}. Moreover, existing results in the literature already show that certain groups are inner \Cs-selfless. For instance,  it follows from \cite[Theorem~3.5]{AGKP25} that a finitely generated acylindrically hyperbolic group with trivial finite radical and the rapid decay property is inner \Cs-selfless; while Ozawa's \cite[Theorem 1]{Oza} shows this for countable groups admitting a topologically free \emph{extreme boundary action}. However, we believe that this notion deserves attention, as it carries interesting consequences that are worth exploring in their own right. 

First of all, let us observe that any inner \Cs-selfless group is clearly \Cs-selfless. Indeed, inner \Cs-selflessness additionally requires that the free unitary is represented by group unitaries. Second, all inner \Cs-selfless groups are \emph{mixed identity-free} (abbreviated MIF); see Proposition~\ref{InnerPropertiesProposition}. This terminology was first introduced in \cite{HO}, where these groups were systematically studied for the first time. Very roughly, a group is MIF if every non-trivial equation involving variables and fixed group elements can be made false by a suitable choice of the variables. For instance, requiring that, for every $a\in\Gamma\setminus\{e\}$, there exists $x\in\Gamma$ with $xax^{-1}a^{-1}\ne e$ means precisely that the centre of $\Gamma$ is trivial. In this work we clarify the relation between MIF and inner \Cs-selfless groups. In particular, one can think of inner \Cs-selflessness as being MIF with an additional control on the \Cs-norm. As a consequence of this, we immediately notice that inner \Cs-selfless groups are not closed under taking products, and thus not all \Cs-selfless groups are inner \Cs-selfless; take, for example, $\mathbb F_2 \times\mathbb F_2$, which is \Cs-selfless by \cite[Theorem~2]{Oza}, but not MIF.
In Proposition~\ref{InnerPropertiesProposition} we also provide a number of other consequences of inner \Cs-selflessness, such as that the centralizer of every non-trivial normal subgroup is trivial. We then look at permanence properties. For instance, we can show that every non-trivial normal subgroup and every finite index subgroup of an inner \Cs-selfless group is again inner \Cs-selfless (see Lemmas~\ref{NormalWitnessLemma} and~\ref{FiniteIndexLemma}). In fact, Lemma~\ref{NormalWitnessLemma} shows that a family witnessing inner \Cs-selflessness can be chosen inside any prescribed non-trivial normal subgroup. 

One deep connection established in this work is that inner \Cs-selflessness grants access to the study of \emph{invariant subalgebras} of $\Cs_r(\Gamma)$ under the conjugation-action of $\Gamma$. 
The study of invariant von Neumann subalgebras goes back to work of Chifan--Das \cite{ChifanDas} and \cite{KalantarPanagopoulos}. Kalantar--Panagopoulos~\cite{KalantarPanagopoulos} proved that invariant von Neumann subalgebras come from normal subgroups for irreducible lattices in connected semisimple Lie groups with trivial center, no compact factors, and simple factors of real rank at least two . Motivated by this result, Amrutam--Jiang~\cite{AmrutamJiangISR} introduced the invariant subalgebra rigidity (ISR)-property and established it for several classes of groups, including torsion-free non-amenable hyperbolic groups and their finite products . Subsequent work of Chifan--Das--Sun~\cite{CDS}, Jiang--Zhou~\cite{JZ}, Dudko--Jiang~\cite{DJ}, and Amrutam--Dudko--Jiang--Skalski~\cite{ADJS25} extended this study to further classes of groups, including amenable groups, and to crossed products and tensor products ~\cite{ADJS26}. The corresponding \Cs-algebraic and relative rigidity properties were studied in \cite{AmrutamJiang26,amrutam2026relative, AmrutamPSL26}.

Recall that $\Gamma$ has the \emph{\Cs-invariant subalgebra rigidity property} (abbreviated \Cs-ISR property) if every unital $\Gamma$-invariant \Cs-subalgebra of $\Cs_r(\Gamma)$ is of the form $\Cs_r(\Lambda)$ for a normal subgroup $\Lambda\triangleleft\Gamma$. In the same spirit, the \emph{relative \Cs-ISR property} requires that, for every non-trivial normal subgroup $N\triangleleft\Gamma$, every unital $N$-invariant \Cs-subalgebra be of the form $\Cs_r(\Lambda)$ for a subgroup $\Lambda\leq\Gamma$, necessarily normalized by $N$~\cite{amrutam2026relative}. Our main theorem regarding \Cs-ISR, which can be seen as a Fourier extraction result, reads as follows.

\begin{introtheorem}[{Theorem \ref{InvariantMainTheorem}}]\label{IntroInvariantMainTheorem} Let $\alpha\colon\Gamma\curvearrowright A$ be an action of an inner \Cs-selfless discrete group $\Gamma$ on a unital \Cs-algebra $A$, and assume that the ultrafilter $\mathcal{U}$ on the set $I$ and the family $(g_{i})_{i\in I}\subseteq\Gamma$ witnessing inner \Cs-selflessness satisfy
\begin{eqnarray} \label{IntroCompatibility}
\lim_{i\rightarrow\mathcal{U}}\left\Vert \alpha_{g_{i}}(a)-a\right\Vert =0\quad\text{for every }a\in A.
\end{eqnarray}
Furthermore, let $N\triangleleft\Gamma$ be a non-trivial normal subgroup and assume that $B\subseteq A\rtimes_{\alpha,\mathrm r}\Gamma$ is a (possibly non-unital) \Cs-subalgebra that is invariant under conjugation by $N$. Then, $\mathbb{E}_{A}(b\lambda_{g}^{\ast})\lambda_{g}\in B$ for every $b\in B$ and $g\in\Gamma$, where $\mathbb E_A$ denotes the canonical conditional expectation. \end{introtheorem}

The compatibility condition \eqref{IntroCompatibility} can be arranged, in particular, for trivial actions and for actions with non-trivial kernel. Applying Theorem~\ref{IntroInvariantMainTheorem} to the trivial action on $\mathbb C$ gives the following corollary.

\begin{introcorollary}[{Corollary \ref{rCISR-Corollary}}]
Let $\Gamma$ be an inner \Cs-selfless discrete group. Then $\Gamma$ has the relative \Cs-ISR property.
\end{introcorollary}

Although inner \Cs-selflessness excludes non-trivial direct products, Theorem~\ref{IntroInvariantMainTheorem} also gives rigidity for finite products.

\begin{introcorollary}[{Corollary \ref{C-ISR}}]
Let $\Gamma_1,\ldots,\Gamma_n$ be inner \Cs-selfless discrete groups. Then the product $\Gamma_1\times\cdots\times\Gamma_n$ has the \Cs-ISR property.
\end{introcorollary}

Theorem~\ref{IntroInvariantMainTheorem} also applies to intermediate \Cs-algebras between $\Cs_r(\Gamma)$ and $A\rtimes_{\alpha,\mathrm r}\Gamma$. Under its compatibility hypothesis, the canonical conditional expectation preserves every such intermediate algebra. If $\Gamma$ has the \emph{approximation property}, each intermediate algebra is the crossed product of a $\Gamma$-invariant \Cs-subalgebra of $A$ (see Corollary~\ref{CReflectingCorollary}). 

In Corollary~\ref{TracialReflectingCorollary}, we also establish the corresponding statement in the von Neumann algebra setting, and obtain relative ISR and ISR for finite products of non-trivial mixed identity-free groups in Corollaries~\ref{TracialRelativeISRCorollary} and~\ref{TracialProductISRCorollary}.
In fact, to obtain the result for group von Neumann algebras, mixed identity-freeness alone suffices. We show that every non-trivial mixed identity-free discrete group has the relative ISR property (see Corollary~\ref{TracialRelativeISRCorollary}). The proof uses the ultrapower embedding in Proposition~\ref{prop:characterization} and the free product extraction result in Theorem~\ref{TracialFreeProductCaseTheorem}. In particular, we obtain as a byproduct that Baumslag--Solitar groups $\mathrm{BS}(m,n)$ with $|m|,|n|\geq2$ and $|m|\ne|n|$ have the ISR property (see Corollary~\ref{cor:BS}), answering the corresponding part of \cite[Question~6.1]{DJ}. We also obtain amenable examples in Corollary~\ref{cor:jacobson}, and show that every countable amenable group embeds into a finitely generated amenable group with ISR in Corollary~\ref{cor:embedding}.

Lastly, we look for new classes of examples of inner \Cs-selfless groups. With this primary focus in mind, we introduce a new group-theoretic condition that we call \emph{cyclic Powers property}. This is very close in spirit to the \emph{Powers property} of de la Harpe \cite{de2007simplicity}, which was inspired by Powers' work \cite{Powers75}. The difference lies in the fact that, while the Powers property asks for an arbitrary $n$-tuple of group elements that suitably move a partition of the group around, the cyclic Powers property requires a similar condition, but where the $n$-tuple consists of powers of a single element. In fact, the cyclic Powers property implies the Powers property, and also the $\mathrm{P}_{\mathrm{PHP}}$ property of Ozawa \cite{Oza} (see Lemma~\ref{lem:cyclicPowers-to-Powers}), and thus it implies that the group is \emph{completely \Cs-selfless}. We recall that \emph{completely selfless} \Cs-probability spaces were defined by Ozawa in \cite{Oza} to provide the right strengthened version of selflessness that passes to spatial tensor products. The real added value of our new notion is, however, that it implies inner \Cs-selflessness as well; see Corollary~\ref{cor:cP-to-innerselfless}. In fact, we deduce this fact from a vastly more general result about complete selflessness of crossed products by groups with the cyclic Powers property, which is interesting in its own right.

\begin{introtheorem}[{Theorem~\ref{thm:cP-crossed-product} and Corollary~\ref{cor:uniquely-ergodic-crossed-selfless}}]\label{introthm:cP-crossed-product}
    Let $\Gamma$ be a non-trivial discrete group with the cyclic Powers property.
    Let $\alpha\colon \Gamma\curvearrowright A$ be an action on a unital
    \Cs-algebra, $\rho$ be a state on $A$ such that $\rho \circ \mathbb E_A$ is GNS-faithful, and suppose that there exists a net of unitaries $(v_i)_{i \in I} \subseteq \Csr(\Gamma)$ such that
    \[
        \|\mathbb E_A(v_i^*xv_i)-\rho(\mathbb E_A(x))1_A\|\longrightarrow0
    \]
    for all $x\in A\rtimes_{\alpha, \mathrm r}\Gamma$.
    Then the canonical inclusion $\Csr(\Gamma)\subseteq(A\rtimes_{\alpha,\mathrm r}\Gamma,\rho\circ\mathbb E_A)$ is completely selfless.

    Moreover, the assumptions are satisfied if there exists an 
    element $t\in \Gamma$ of infinite-order such that $\alpha_t$ is uniquely ergodic with unique invariant state $\rho$, and $\rho \circ \mathbb E_A$ is GNS-faithful.
\end{introtheorem}

In particular, this marks a new milestone for selflessness of crossed product \Cs-algebras, which was already studied in \cite{MWY25a,MWY25b,FFKP26b,GJKPR26,Bell26a,Ohs26,Bell26b,KP26}.

We then establish that groups that have the cyclic Powers property, and that are thus inner \Cs-selfless, exist in abundance.
Indeed, we show that the cyclic Powers property holds for every group admitting a topologically free \emph{extreme boundary action} on a compact Hausdorff space with more than two points (see Proposition~\ref{prop:extremelyproximal-cP}), and for Zariski-dense discrete subgroups of $\mathrm{PSL}(d,\mathbb K)$, where $\mathbb K$ is a local field and $d\geq2$, provided that the group contains a biproximal element (see Theorem~\ref{thm:Zariskidense-cyclicPowers}). In particular, this includes Zariski-dense subgroups of $\mathrm{PSL}(d,\mathbb R)$ (see Corollary~\ref{cor:examples}), and establishes that they have the Powers property; see Corollary~\ref{cor:linear-Powers}.
This consequence is particularly remarkable as it answers --- in the positive --- a question of Bridson and de la Harpe from more than $20$ years ago \cite[p.~202]{BridsondelaHarpe2004}, which asks whether the groups $\mathrm{PSL}(d,\mathbb Z)$, $d\geq3$ have the Powers property.

\vspace{3mm}

\subsection*{Organization of the paper}
In addition to this introduction, this paper has five other sections. Section~\ref{Sec:preliminaries} collects the preliminaries and notation.
In Section~\ref{sec:innerselflessness}, we introduce inner $\Cs$-selfless groups and study their structural properties. In particular, we prove Lemmas~\ref{NormalWitnessLemma} and~\ref{FiniteIndexLemma}, and Proposition~\ref{InnerPropertiesProposition} here.
Section~\ref{sec:invC*crossed} is devoted to invariant $\Cs$-subalgebras of crossed products by inner $\Cs$-selfless groups. We prove Theorem~\ref{InvariantMainTheorem} and obtain Corollaries~\ref{rCISR-Corollary}, \ref{C-ISR}, and~\ref{CReflectingCorollary}.
In Section~\ref{sec:cyclicpowers}, we introduce the cyclic Powers property and prove Theorem~\ref{thm:cP-crossed-product}. We obtain inner $\Cs$-selflessness in Corollary~\ref{cor:cP-to-innerselfless}, establish the examples in Corollary~\ref{cor:examples}, and deduce the Powers property for linear groups in Corollary~\ref{cor:linear-Powers}. We also obtain the crossed product application in Corollary~\ref{cor:uniquely-ergodic-crossed-selfless}.
Finally, in Section~\ref{sec:MIFandISR}, we prove Fourier coefficient extraction for $L^2$-compatible actions of non-trivial mixed identity-free discrete groups, and hence the relative ISR property, in Theorem~\ref{TracialInvariantMainTheorem}. We then obtain ISR for finite products of non-trivial mixed identity-free groups in Corollary~\ref{TracialProductISRCorollary}, together with the applications in Corollaries~\ref{cor:BS}, \ref{cor:jacobson}, and~\ref{cor:embedding}.


\subsection*{Statement on the use of AI}
The mathematical ideas, results, and proofs are due to
the authors. ChatGPT (OpenAI) was used for grammatical corrections,
assistance with exposition and references. 

\subsection*{Acknowledgements}
The work began when T.A. was visiting M.P. at IMADA, SDU in June 2026. He thanks the department for their kind invitation.
Part of this project was ultimated during M.P.’s visit to T.A.\ at IMPAN in October 2026, and he is extremely thankful for the warm hospitality received during his stay. T.A.\ is supported by National Science Centre, Poland Sonata, grant number 2025/59/D/ST1/03117. M.P.\ was supported by the Carlsberg Foundation grant CF24-1144.

\vspace{3mm}


\section{Preliminaries and notation}
\label{Sec:preliminaries}
We first fix our notation and recall the free product and ultrapower constructions used in the proofs. We also spend some time to briefly recall the notions of C$^*$-probability spaces and crossed products associated with C$^*$-algebras and tracial von Neumann algebras. 
\subsection{General notation}

We write $\mathbb N:=\{0,1,\ldots\}$, $\mathbb N_{\geq1}:=\{1,2,\ldots\}$, $\mathbb T:=\{z\in\mathbb C\mid |z|=1\}$, and $\ell^2:=\ell^2(\mathbb N)$. For a finite set $F$, both $|F|$ and $\#F$ denote its cardinality. The neutral element of a group is denoted by $e$. For a subset $S$ of a group, $\langle S\rangle$ denotes the subgroup generated by $S$, and $\mathbb F_r$ denotes the free group of rank $r$. We use $\id$ for identity maps.

Inner products on Hilbert spaces are assumed to be linear in the first variable. For a Hilbert space $\mathcal H$, $\mathcal B(\mathcal H)$ denotes the algebra of bounded operators on $\mathcal H$. For a subset $S$ of a unital \Cs-algebra $A$, $\Cs(S)$ denotes the unital \Cs-subalgebra generated by $S$. For a set of operators $S$, $W^*(S)$ denotes the von Neumann algebra generated by $S$. We identify scalars with scalar multiples of the identity. We write $\mathrm{Aut}(A)$ for the group of $\ast$-automorphisms of $A$ and $\operatorname{Ad}(u)(x):=uxu^*$ for conjugation by a unitary $u$. For a projection $p$, we write $p^\perp:=1-p$.

For \Cs-algebras, $\odot$ denotes the algebraic tensor product and $\otimes$ denotes the minimal tensor product, also written $\otimes_{\min}$. For Hilbert spaces, $\otimes$ denotes the Hilbert space tensor product. For a compact Hausdorff space $X$, we write $C(X)$ for the \Cs-algebra of continuous complex-valued functions on $X$, $\mathrm{Homeo}(X)$ for its homeomorphism group, and $\Prob(X)$ for the space of regular Borel probability measures on $X$, equipped with the weak$^\ast$ topology.
\subsection{Free products of groups} \label{subsec:free-products}
Let $\Gamma=\Gamma_{1}\star\Gamma_{2}$ be the free product of two non-trivial discrete groups $\Gamma_{1}$ and $\Gamma_{2}$. By the normal form theorem for free products, every element $g\in\Gamma$ admits a unique expression of the form $g=g_{1}\cdots g_{n}$ with $n\geq0$, where each factor $g_{1},\ldots,g_{n}$ is a non-identity element of one of the respective components of $\Gamma$, and consecutive factors belong to different components. The elements $g_{1},\ldots,g_{n}$ are called the \emph{syllables} of $g$, and $n$ is its \emph{syllable length}. By convention, the empty word represents the neutral element $e\in\Gamma$.

The expression $g_{1}\cdots g_{n}$ is called \emph{cyclically reduced} if $n\leq1$ or if its first and last syllables belong to different components. Every element in $\Gamma$ is conjugate to a cyclically reduced one. Furthermore, if $g\in\Gamma$ is cyclically reduced and $n\geq2$, then $g$ has infinite order and $g^{k}$ has syllable length $|k|n$ for every integer $k\in\mathbb{Z}$.

The conjugacy theorem for free products (see, e.g., \cite[Chapter~IV, Theorem~1.4]{LyndonSchupp01}) states that conjugate cyclically reduced elements have the same syllable length. For $n\geq2$, the cyclically reduced conjugates of $g=g_{1}\cdots g_{n}$ are exactly the cyclic permutations $g_{k+1}\cdots g_{n}g_{1}\cdots g_{k}$, where $0\leq k<n$.

A non-identity element $g\in\Gamma\setminus\{e\}$ is a \emph{proper power} if $g=v^{m}$ for some $v\in\Gamma$ and an integer $m\geq2$. For cyclically reduced $g$ of syllable length at least two, this is equivalent to the syllable sequence of $g$ being a concatenation of at least two identical shorter blocks. Indeed, cyclically reducing a root $v$ gives an element of syllable length at least two whose $m$-th power is conjugate to $g$. That power has a repeated reduced normal form, and the conjugacy theorem implies that the normal form of $g$ is a cyclic permutation of it, hence is also a repetition.\\

The \emph{Bass--Serre tree} $T$ of a free product $\Gamma=\Gamma_{1}\star\Gamma_{2}$ is given by the vertex set $VT:=\Gamma/\Gamma_{1}\sqcup\Gamma/\Gamma_{2}$ and the edge set $ET:=\{(g\Gamma_{1},g\Gamma_{2})\mid g\in\Gamma\}$. The group $\Gamma$ naturally acts on $T$ by left multiplication. This action is without inversions. For $j\in\{1,2\}$, the stabilizer of the vertex $g\Gamma_{j}$ is $g\Gamma_{j}g^{-1}$, and every edge stabilizer is trivial; see \cite[Chapter~I, Section~4.1, Theorem~7]{Serre80}.

An element $g\in\Gamma$ is called \emph{elliptic} if it fixes a vertex of $T$, which is equivalent to $g$ being conjugate into one of the factors $\Gamma_{1}$ or $\Gamma_{2}$. In contrast, $g\in\Gamma$ is called \emph{loxodromic} if it preserves a bi-infinite geodesic $\text{Ax}(g)$ of $T$, called its \emph{axis}, and translates along it by a positive distance. For the Bass--Serre action of $\Gamma$ on $T$, every element is either elliptic or loxodromic; there is no third, parabolic class.

Every loxodromic element has a unique axis. Its \emph{translation length} is the number of edges by which it translates this axis. Conjugation preserves translation length and maps the axis of $g$ to the axis of its conjugate; see \cite[Chapter~I, Section~6.4, Propositions~24--25]{Serre80}. Every cyclically reduced element of syllable length at least two is loxodromic.

The \emph{visual boundary} $\partial T$ of $T$ is the set of ends of $T$, where two geodesic rays represent the same end if they have a common terminal ray. We equip the union $VT\cup\partial T$ with the \emph{observers' topology} and call the resulting compact Hausdorff space the \emph{compactification} of $T$; for details see \cite[Section~5.2]{BrownOzawa08}.

For a loxodromic element $g\in\Gamma$, let $z_{+},z_{-}\in\partial T$ be the ends of its axis with $z_{+}$ being the point in the direction of its translation. Then $g^{n}.x\to z_{+}$ for every $x\in(VT\cup\partial T)\setminus\{z_{-}\}$, and $g^{-n}.x\to z_{-}$ for every $x\in(VT\cup\partial T)\setminus\{z_{+}\}$. Furthermore, since edge stabilizers are trivial, two loxodromic elements with the same translation length whose axes have a common end are equal or inverse.


\subsection{Tracial and Norm ultrapowers}
Let $(M,\tr)$ be a finite von Neumann algebra with faithful normal
tracial state, and let $\mathcal U$ be an ultrafilter on a set $I$.
Its tracial ultrapower is 
\[M_{\mathcal U}:=\ell^\infty(I,M)/\{(x_i)_{i \in I}\mid\lim_{i\to\mathcal U}\normtwo{x_i}=0\}.
\]
We write $(x_i)_{i\rightarrow \mathcal{U}}$ for the class of a bounded family,
$\Delta\colon M\to M_\mathcal{U}$ for the diagonal embedding, and
$\tr_\mathcal{U}((x_i)_{i \rightarrow \mathcal{U}}):=\lim_{i\to\mathcal{U}}\tr(x_i)$ for the limit trace induced by $\tau$.
For a von Neumann subalgebra $P\subseteq M$, we regard $P_\mathcal{U}$
as a von Neumann subalgebra of $M_\mathcal{U}$ via the canonical inclusion.

Let $(A,\rho)$ be a unital \Cs-algebra equipped with a state. For an ultrafilter $\mathcal U$ on a set $I$, the \emph{norm ultrapower} of $A$ with respect to $\mathcal U$ is defined as 
\[
    A^{\mathcal U} := \ell^{\infty}(I,A)/\{(a_i)_{i \in I} \mid \lim_{i\to\mathcal U}\|a_i\|=0\} .
\]
As before, we denote the class of a bounded family $(a_i)_{i \in I}$ simply by $(a_i)_{i \rightarrow \mathcal U}$.
The \Cs-algebra $A^{\mathcal U}$ is unital, and carries a state $\rho^{\mathcal U}$ induced by $\rho$ as follows, $\rho^{\mathcal U}\big( (a_i)_{i \rightarrow \mathcal U}\big) := \lim_{i\to\mathcal U}\rho(a_i)$ for all $(a_i)_{i\rightarrow \mathcal U} \in A^{\mathcal U}$. Again, we denote by $\Delta$ the diagonal embedding of $A$ into $A^{\mathcal U}$ given by constant families.

Note that we have deliberately opted to change the position of $\mathcal U$ depending on the type of ultrapower, in order to avoid confusion between them.

An ultrafilter $\mathcal U$ on a directed set $I$ is called
\emph{cofinal} if
$\{j\in I\mid j\geq i\}\in\mathcal U$ for every $i\in I$.


\subsection{Selfless \texorpdfstring{\Cs}{}-probability spaces}
We briefly recall the terminology from \cite[Subsections~1.2 and~1.6]{KP26}. A state $\rho$ on a unital \Cs-algebra $A$ is called \emph{GNS-faithful} if its GNS-representation is faithful. A \emph{\Cs-probability space} is a pair $(A,\rho)$ with $\rho$ GNS-faithful; this requirement is not imposed on ultrapower states. The space is \emph{non-trivial} if $A\not\cong\mathbb C$. Embeddings of \Cs-probability spaces are unital and state-preserving. The notation $B\subseteq(A,\rho)$ means that $B$ is a unital \Cs-subalgebra of $A$ and $\rho|_B$ is GNS-faithful. A unitary $u\in A$ is called a \emph{Haar unitary} if $\rho(u^n)=0$ for every $n\in\mathbb Z\setminus\{0\}$.

For \Cs-probability spaces $(A,\rho)$ and $(C,\kappa)$, we denote their reduced free product by $(A,\rho)\star(C,\kappa)$ and its free product state by $\rho\star\kappa$; see \cite[Section~4.7]{BrownOzawa08}. All free products of \Cs-probability spaces in this paper are reduced free products with respect to the specified states.

\begin{definition}[{\cite[Definition~2.3]{HKEPR25} and \cite[Subsection~1.6]{KP26}}]
Let $B\subseteq(A,\rho)$ be an inclusion of \Cs-probability spaces. The inclusion is called \emph{selfless} if there exist a non-trivial \Cs-probability space $(C,\kappa)$, an ultrafilter $\mathcal U$ on a set $I$, and a unital embedding $\theta\colon(A,\rho)\star(C,\kappa)\hookrightarrow A^{\mathcal U}$ such that
\begin{enumerate}[label=\textup{(\roman*)},leftmargin=*]
\item $\theta|_A=\Delta$;
\item $\rho^{\mathcal U}\circ\theta=\rho\star\kappa$;
\item $\theta(C)\subseteq B^{\mathcal U}$.
\end{enumerate}
The inclusion is called \emph{completely selfless} if these data can be chosen so that, for every \Cs-algebra $R$, the algebraic $\ast$-homomorphism
\[
\begin{aligned}
\bigl((A,\rho)\star(C,\kappa)\bigr)\odot R
&\longrightarrow(A\otimes R)^{\mathcal U},\\
x\otimes r&\longmapsto(a_i\otimes r)_{i\to\mathcal U},
\end{aligned}
\]
where $\theta(x)=(a_i)_{i\to\mathcal U}$, is continuous for the minimal tensor norm.

The space $(A,\rho)$ is called \emph{selfless}, respectively \emph{completely selfless}, if the inclusion $A\subseteq(A,\rho)$ has the corresponding property.
\end{definition}

\subsection{Group \texorpdfstring{\Cs}{}-algebras, group von Neumann algebras, and crossed products}

Let $\Gamma$ be a discrete group and let $\lambda\colon\Gamma\rightarrow\mathcal{B}(\ell^{2}(\Gamma))$ be the corresponding left-regular representation. We denote the reduced group \Cs-algebra of $\Gamma$ by 
\[
\Csr(\Gamma):=\Cs\left(\{\lambda_{g}\mid g\in\Gamma\}\right)\subseteq\mathcal{B}(\ell^{2}(\Gamma))
\]
and write $\tau_{\Gamma}$ for the canonical faithful tracial state on $\Csr(\Gamma)$. Furthermore, we denote the group algebra of $\Gamma$ by $\mathbb{C}[\Gamma]$ and view it as a dense $\ast$-subalgebra of $\Csr(\Gamma)$. For a subgroup $\Lambda\leq\Gamma$ we denote the canonical conditional expectation from $\Csr(\Gamma)$ onto $\Csr(\Lambda)$ by $\mathbb{E}_{\Lambda}$; for notational convenience we write $\mathbb{E}_{g}:=\mathbb{E}_{\langle g\rangle}$ for $g\in\Gamma$.

We now briefly recall the construction of the reduced crossed product. Let $\alpha\colon\Gamma\curvearrowright A$ be an action of a discrete group on a unital $\Cs$-algebra, and choose a faithful unital representation $\pi\colon A\to\mathcal B(\mathcal H)$. On the space $\ell^2(\Gamma,\mathcal H)$ of square-summable $\mathcal H$-valued functions, define
\[
(\sigma(a)\xi)(h):=\pi(\alpha_{h^{-1}}(a))\xi(h),~
(\lambda_g\xi)(h):=\xi(g^{-1}h)
\]
for $a\in A$, $g,h\in\Gamma$, and $\xi\in\ell^2(\Gamma,\mathcal H)$. The \emph{reduced crossed product} is 
\[
A\rtimes_{\alpha,\mathrm r}\Gamma
:=\overline{\operatorname{span}}\{\sigma(a)\lambda_g\mid a\in A,\ g\in\Gamma\}
\subseteq\mathcal B(\ell^2(\Gamma,\mathcal H)),
\]
where the closure is taken in the operator norm. This construction is independent of $\pi$, up to a canonical isomorphism; see \cite[Section~4.1]{BrownOzawa08}. Identifying $A$ with $\sigma(A)$, we have $\lambda_ga\lambda_g^*=\alpha_g(a)$, and the group unitaries generate a canonical copy of $\Csr(\Gamma)$. The canonical faithful conditional expectation $\mathbb E_A\colon A\rtimes_{\alpha,\mathrm r}\Gamma\to A$ is determined by $\mathbb E_A(a\lambda_g)=\delta_{g,e}a$. 

By a slight abuse of notation, for $\Lambda\leq\Gamma$, we also write $\mathbb E_\Lambda$ for the canonical conditional expectation onto $A\rtimes_{\alpha|_\Lambda,\mathrm r}\Lambda$, given on finite sums by $\mathbb E_\Lambda(\sum_g a_g\lambda_g)=\sum_{g\in\Lambda}a_g\lambda_g$.

Let us now turn to the von Neumann setup.
The \emph{group von Neumann algebra} of $\Gamma$ denoted by
$\vN(\Gamma)$ is the von Neumann algebra generated by $\{\lambda_g\mid g\in\Gamma\}
\subseteq\mathcal B(\ell^2(\Gamma))$.
The canonical trace extends to a faithful normal tracial state on
$\vN(\Gamma)$, still denoted by $\tau_\Gamma$, given by
$\tau_\Gamma(x)=\langle x\delta_e,\delta_e\rangle$.
For a subgroup $\Lambda\leq\Gamma$, we also write $\mathbb E_\Lambda$
for the canonical trace-preserving normal conditional expectation
from $\vN(\Gamma)$ onto $\vN(\Lambda)$.

We also briefly recall the tracial crossed product of von Neumann algebras. Let $\alpha\colon\Gamma\curvearrowright(M,\tau)$ be a trace-preserving
action on a finite von Neumann algebra with faithful normal tracial
state. Represent $M$ on $\mathcal H:=L^2(M,\tau)$ by left multiplication.
On the space $\ell^2(\Gamma,\mathcal H)$ of square-summable
$\mathcal H$-valued functions on $\Gamma$, define
\[
(\sigma(a)\xi)(h):=\alpha_{h^{-1}}(a)\xi(h),~
(\lambda_g\xi)(h):=\xi(g^{-1}h)
\]
for $a\in M$, $g,h\in\Gamma$, and $\xi\in\ell^2(\Gamma,\mathcal H)$.
The \emph{tracial crossed product} $M\rtimes_\alpha\Gamma$ is the
von Neumann algebra generated by $\sigma(M)$ and
$\{\lambda_g\mid g\in\Gamma\}$. Identifying $M$ with $\sigma(M)$,
one has $\lambda_ga\lambda_g^*=\alpha_g(a)$.

The canonical faithful normal conditional expectation
$\mathbb E_M\colon M\rtimes_\alpha\Gamma\to M$ is determined by
$\mathbb E_M(a\lambda_g)=\delta_{g,e}a$, and $\tau\circ\mathbb E_M$
is a faithful normal tracial state.
For $\Lambda\leq\Gamma$, we also denote the canonical trace-preserving
normal conditional expectation onto
$M\rtimes_{\alpha|_\Lambda}\Lambda$ by $\mathbb E_\Lambda$, and write
$\mathbb E_g:=\mathbb E_{\langle g\rangle}$ for $g\in\Gamma$.
\vspace{3mm}


\section{Inner \texorpdfstring{\Cs}{}-selfless groups}
\label{sec:innerselflessness}

Inspired by Robert's notion of selfless \Cs-algebras
\cite{Robert25}, and more directly by the notion of
existentially \Cs-residually-$\Gamma$ groups in
\cite[Definition~3.4]{AGKP25} and the ultrapower embeddings
constructed by Ozawa in \cite[Theorem~1]{Oza}, we introduce
the following definition.

\begin{definition} \label{InnerDefinition} A non-trivial discrete group $\Gamma$ is called \emph{inner \Cs-selfless} if there exist an ultrafilter $\mathcal{U}$ on a set $I$, a family $(g_{i})_{i\in I}\subseteq\Gamma$, and a unital embedding $\theta\colon \Csr(\Gamma\star\mathbb{Z})\hookrightarrow \Csr(\Gamma)^{\mathcal{U}}$ such that $\theta|_{\Csr(\Gamma)}=\Delta$ and $\theta(\lambda_{h})=(\lambda_{g_{i}})_{i\rightarrow\mathcal{U}}$, where $h$ denotes the canonical generator of $\mathbb{Z}$ and where $\Delta\colon \Csr(\Gamma)\hookrightarrow \Csr(\Gamma)^{\mathcal{U}}$ is the diagonal embedding. In this case, the family $(g_{i})_{i\in I}\subseteq\Gamma$ is said to \emph{witness inner \Cs-selflessness} of $\Gamma$. \end{definition}

\begin{remark} (i) By \cite[Theorem~1]{Oza}, every countably infinite group admitting a topologically free extreme boundary with more than two points is inner \Cs-selfless. Among countable groups, examples include groups acting minimally and of general type on a tree with trivial half-tree fixators, such as $\mathrm{BS}(m,n)$ with $|m|,|n|\geq2$ and $|m|\ne|n|$; see \cite[Lemma~8.6]{FMMS}. Combining Ozawa's theorem with \cite[Theorem~1.3]{Yang26} also covers all countable acylindrically hyperbolic groups with trivial finite radical. Further examples are $\operatorname{Aut}(\Gamma)$ and $\operatorname{Comm}(\Gamma)$ for non-elementary hyperbolic groups $\Gamma$ with trivial finite radical \cite{BasuFlores26}.

(ii) The embedding $\theta$ in Definition \ref{InnerDefinition} is automatically
trace-preserving in the sense that
$\tau_\Gamma^{\mathcal U}\circ\theta
=\tau_{\Gamma\star\mathbb Z}$.
Indeed, the left-hand side is a tracial state on
$\Csr(\Gamma\star\mathbb Z)$, whose tracial state is unique.\end{remark}

The following lemma demonstrates that inner \Cs-selflessness imposes severe restrictions on the normal subgroup structure of the respective group.

\begin{lemma} \label{NormalWitnessLemma} Let $\Gamma$ be an inner \Cs-selfless discrete group and let $N\triangleleft\Gamma$ be a non-trivial normal subgroup. Then there exists a family $(n_{i})_{i \in I}\subseteq N$ witnessing inner \Cs-selflessness of $\Gamma$. In particular, $N$ itself is inner \Cs-selfless. \end{lemma}

\begin{proof} Let $(g_{i})_{i \in I}\subseteq\Gamma$ be a family witnessing inner \Cs-selflessness of $\Gamma$, so that there exists an ultrafilter $\mathcal{U}$ on $I$ and a unital embedding $\theta\colon \Csr(\Gamma\star\mathbb{Z})\hookrightarrow \Csr(\Gamma)^{\mathcal{U}}$ with $\theta|_{\Csr(\Gamma)}=\Delta$ and $\theta(\lambda_{h})=(\lambda_{g_{i}})_{i\rightarrow\mathcal{U}}$. Here, $h$ denotes the canonical generator of $\mathbb{Z}$ and $\Delta\colon \Csr(\Gamma)\hookrightarrow \Csr(\Gamma)^{\mathcal{U}}$ is the diagonal embedding. For an arbitrary element $n\in N\setminus\{e\}$ and $i\in I$ define $n_{i}:=(g_{i}ng_{i}^{-1})(g_{i}^{2}ng_{i}^{-2})\in N$. Note that $\widetilde{h}:=(hnh^{-1})(h^{2}nh^{-2})$ is contained in the kernel of the canonical group homomorphism $\Gamma\star\mathbb{Z}\rightarrow\mathbb{Z}$, which canonically identifies with the free product $\star_{k\in\mathbb{Z}}(h^{k}\Gamma h^{-k})$. It follows that $\widetilde{h}$ has infinite order and is freely independent from $\Gamma\leq\Gamma\star\mathbb{Z}$. Thereby, the subgroup $\Lambda$ of $\Gamma\star\mathbb{Z}$ generated by $\Gamma$ and $\widetilde{h}$ is of the form $\Lambda\cong\Gamma\star\mathbb{Z}$ and restricting $\theta$ to the \Cs-subalgebra $\Csr(\Lambda)$ yields a unital embedding $\widetilde{\theta}\colon \Csr(\Lambda)\hookrightarrow \Csr(\Gamma)^{\mathcal{U}}$ with $\widetilde{\theta}|_{\Csr(\Gamma)}=\Delta$ and $\widetilde{\theta}(\lambda_{\widetilde{h}})=(\lambda_{n_{i}})_{i\rightarrow\mathcal{U}}$. \end{proof}

\begin{remark} Lemma \ref{NormalWitnessLemma} implies that for every inner \Cs-selfless discrete group $\Gamma$, and every non-trivial normal subgroup $N \triangleleft\Gamma$ the inclusion $\Csr(N)\subseteq(\Csr(\Gamma),\tau_{\Gamma})$ is selfless in the sense of \cite{HKEPR25}, i.e., there exists an ultrafilter $\mathcal{U}$, a non-trivial \Cs-probability space $(C,\kappa)$, and a unital $\ast$-embedding $\theta\colon(\Csr(\Gamma),\tau_{\Gamma})\star(C,\kappa)\hookrightarrow \Csr(\Gamma)^{\mathcal{U}}$ with $\theta|_{\Csr(\Gamma)}=\Delta$, $\tau_{\Gamma}^{\mathcal{U}}\circ\theta=(\tau_{\Gamma}\star\kappa)$ and $\theta(C)\subseteq \Csr(N)^{\mathcal U}$. Here $\Delta\colon \Csr(\Gamma)\hookrightarrow \Csr(\Gamma)^{\mathcal{U}}$ denotes the diagonal embedding of $\Csr(\Gamma)$ inside $\Csr(\Gamma)^{\mathcal{U}}$, and $\tau_{\Gamma}\star\kappa$ is the free product state of the reduced free product $(\Csr(\Gamma),\tau_{\Gamma})\star(C,\kappa)$.
\end{remark}

To collect some immediate group-theoretic implications of inner \Cs-selflessness, let us introduce a few well-known notions.

\begin{definition}
    A group $\Gamma$ is said to be:
    \begin{itemize}
    \item \emph{mixed identity-free} (MIF) if, for every
    integer $r\geq1$ and every non-trivial word $w\in\Gamma\mathbin{\star}\mathbb F_r$, there is
    a group homomorphism $\varphi\colon\Gamma\mathbin{\star} \mathbb F_r\to\Gamma$ such that
    $\varphi|_\Gamma=\id_\Gamma$ and $\varphi(w)\ne e$.
    For non-trivial $\Gamma$, it is equivalent to impose this condition
    only for $r=1$ (see \cite[Remark~5.1 and Definition~5.2]{HO}),
    \item \emph{directly indecomposable} if it cannot be expressed as a direct product of two non-trivial groups, i.e., if $\Gamma\cong H\times K$, then $H=\{e\}$ or $K=\{e\}$.
    \end{itemize}
\end{definition}

\begin{proposition} \label{InnerPropertiesProposition} Let $\Gamma$ be an inner \Cs-selfless discrete group. Then the following statements hold:
\begin{enumerate}[leftmargin=*,label=\textup{(\roman*)}]
\item The group $\Gamma$ is mixed identity-free.
\item The group $\Gamma$ is directly indecomposable.
\item The group $\Gamma$ contains no two non-trivial commuting normal subgroups.
\item Any two non-trivial normal subgroups of $\Gamma$ intersect non-trivially.
\item The centralizer of every non-trivial normal subgroup of $\Gamma$ is trivial.
\item There exists at most one minimal non-trivial normal subgroup of $\Gamma$.
\end{enumerate}
\end{proposition}

\begin{proof} Let $(g_{i})_{i\in I}\subseteq\Gamma$ be a family witnessing inner $\Cs$-selflessness of $\Gamma$, so that there exists an ultrafilter $\mathcal{U}$ on $I$ and a unital embedding $\theta\colon\Csr(\Gamma\star\mathbb{Z})\hookrightarrow\Csr(\Gamma)^{\mathcal{U}}$ with $\theta|_{\Csr(\Gamma)}=\Delta$ and $\theta(\lambda_{h})=(\lambda_{g_{i}})_{i\rightarrow\mathcal{U}}$. If $w\in\Gamma\star\mathbb{Z}$ is a mixed identity, then $\theta(\lambda_{w})=1$, so that the injectivity of $\theta$ gives $w=e$. Thus (i) follows from \cite[Remark 5.1]{HO}.

Using (i), the assertions (ii) and (iii) immediately follow from \cite[Proposition 5.4(b)]{HO} and \cite[Section~2.4]{LeBoudecMatteBon22}. Furthermore, (iii) implies the statement in (iv) since $[H,K]\subseteq H\cap K$ for all normal subgroups $H,K\triangleleft\Gamma$. Similarly, the centralizer of a normal subgroup is normal in $\Gamma$ and commutes with that subgroup, so (iii) also implies (v).

Finally, by (iv), two minimal non-trivial normal subgroups of $\Gamma$ have a non-trivial intersection, which must therefore be equal to both subgroups. This proves (vi). \end{proof}

We next show that inner \Cs-selflessness passes to finite index subgroups.

\begin{lemma} 
\label{FiniteIndexLemma}
Let $\Gamma$ be an inner \Cs-selfless discrete group and let $\Lambda\leq\Gamma$ be a finite index subgroup. Then $\Lambda$ is inner \Cs-selfless. \end{lemma}

\begin{proof} The normal core $N:=\bigcap_{g\in\Gamma}g\Lambda g^{-1}$ is a non-trivial finite index normal subgroup of $\Gamma$ contained in $\Lambda$. By Lemma \ref{NormalWitnessLemma}, we may choose a family $(n_{i})_{i\in I}\subseteq N$ witnessing inner $\Cs$-selflessness of $\Gamma$ with respect to an ultrafilter $\mathcal{U}$ on $I$. Since $N\subseteq\Lambda$, the corresponding $\ast$-embedding restricts to a unital embedding $\Csr(\Lambda\star\mathbb{Z})\hookrightarrow\Csr(\Lambda)^{\mathcal{U}}$ that is diagonal on $\Csr(\Lambda)$ and maps $\lambda_{h}$ to $(\lambda_{n_{i}})_{i\to\mathcal{U}}$. Thus $\Lambda$ is inner $\Cs$-selfless. \end{proof}

Beyond the algebraic implications of inner \Cs-selflessness recorded above, the notion leads to a number of analytic consequences, some of which we record in the following. By \cite[Theorem~3.1]{Robert25}, inner \Cs-selflessness of a discrete group $\Gamma$ implies that $\Csr(\Gamma)$ is simple, has $\tau_\Gamma$ as its unique tracial
state, and has stable rank one and strict comparison of positive elements. 

The references \cite{Haagerup17} and \cite{Kennedy20} establish an equivalence between \Cs-simplicity and Powers\textquoteright{} averaging property. As the following lemma illustrates, inner \Cs-selfless groups allow for a particularly efficient averaging process.

\begin{lemma} Let $\Gamma$ be an inner $\Cs$-selfless discrete group and let $N\triangleleft\Gamma$ be a non-trivial normal subgroup. Then, for every $x\in\Csr(\Gamma)$, 
\[
\inf_{n_{1},\ldots,n_{5}\in N}\Big\Vert \frac{1}{5}\sum_{j=1}^{5}\lambda_{n_{j}}^{*}x\lambda_{n_{j}}-\tau_{\Gamma}(x)1\Big\Vert \leq\frac{2}{\sqrt{5}}\|x\|.
\]
 In particular, the inclusion $\Csr(N)\subseteq\Csr(\Gamma)$ is $\Cs$-irreducible in the sense of \cite[Definition 3.1]{Rordam23}. \end{lemma}

\begin{proof} By Lemma \ref{NormalWitnessLemma} we may choose a family $(n_{i})_{i\in I}\subseteq N$ witnessing inner $\Cs$-selflessness of $\Gamma$. Let $\mathcal{U}$ be the associated ultrafilter on $I$ and denote the corresponding embedding $\Csr(\Gamma\star\mathbb{Z})\hookrightarrow\Csr(\Gamma)^{\mathcal{U}}$ by $\theta$. Following the projection argument in \cite[Lemma~1.2]{PaschkeSalinas79}, let $p$ be the orthogonal projection onto the closure of the span of all elements $\delta_g \in \ell^2(\Gamma \star \mathbb{Z})$, where $g$ is a reduced expression in $\Gamma \star \mathbb{Z}$ starting with a syllable in $\Gamma\setminus\{e\}$. Then, $p^{\perp}xp^{\perp}=\tau_{\Gamma}(x)p^{\perp}$ and thereby $x-\tau_{\Gamma}(x)1=xp+y$, where $y:=pxp^{\perp}-\tau_{\Gamma}(x)p$. Furthermore, we estimate $\|xp\|\leq\|x\|$ and 
\[
\begin{aligned}\|y\|^{2} & =\Vert pxp^{\perp}x^{*}p+|\tau_{\Gamma}(x)|^{2}p\Vert \\
 & =\|pxp^{\perp}\|^{2}+|\tau_{\Gamma}(x)|^{2}=\|xp^{\perp}\|^{2}\leq\|x\|^{2},
\end{aligned}
\]
where the third equality follows from $(xp^{\perp})^{*}(xp^{\perp})=p^{\perp}x^{*}pxp^{\perp}+|\tau_{\Gamma}(x)|^{2}p^{\perp}$.

For $1\leq j\leq5$ the projections $\lambda_{h^{j}}^{*}p\lambda_{h^{j}}\in\mathcal{B}(\ell^{2}(\Gamma\star\mathbb{Z}))$ are mutually orthogonal, where $h$ denotes the generator of $\mathbb{Z}$. The corresponding conjugates of $xp$ have orthogonal domain projections, and those of $y$ have orthogonal ranges. Consequently, the Cauchy--Schwarz inequality yields 
\[
\Big\Vert \frac{1}{5}\sum_{j=1}^{5}\lambda_{h^{j}}^{*}(x-\tau_{\Gamma}(x)1)\lambda_{h^{j}}\Big\Vert \leq\frac{\|xp\|+\|y\|}{\sqrt{5}}\leq\frac{2}{\sqrt{5}}\|x\|.
\]
 Applying $\theta$ implies the asserted inequality. \end{proof}

Following \cite[Definition 3.4]{AGKP25}, a discrete group $\Gamma$ containing a subgroup $\Lambda$ is said to be \emph{existentially \Cs-residually-$\Lambda$} if for any finite subset $F\subseteq\Gamma$ and $\varepsilon>0$, there exists a group homomorphism $\varphi\colon\Gamma\to\Lambda$ restricting to the identity on $\Lambda$, such that for every element $x:=\sum_{g\in F}x(g)\lambda_{g}\in\mathbb{C}[\Gamma]$ with $(x(g))_{g\in F}\subseteq\mathbb{C}$ and $\sum_{g\in F}|x(g)|=1$ the inequality 
\[
\Big\Vert \sum_{g\in F}x(g)\lambda_{\varphi(g)}\Big\Vert \leq\|x\|+\varepsilon
\]
holds. The following proposition illustrates that inner \Cs-selflessness can be expressed in terms of this terminology.

\begin{proposition} Let $\Gamma$ be a non-trivial discrete group, and $\mathbb Z=\langle h\rangle$. Then the following statements are equivalent:
\begin{enumerate}[label=\textup{(\roman*)},leftmargin=*]
\item $\Gamma$ is inner \Cs-selfless.
\item $\Gamma\star\mathbb{Z}$ is existentially \Cs-residually-$\Gamma$.
\item For all finite subsets $K_{1},K_{2}\subseteq\Gamma\star\mathbb{Z}$ with $e\notin K_{2}$ and $\varepsilon>0$ there exists an element $s\in\Gamma$ such that the unique group homomorphism $\varphi_{s}\colon\Gamma\star\mathbb{Z}\rightarrow\Gamma$ given by $\varphi_{s}|_{\Gamma}=\id_{\Gamma}$ and $\varphi_{s}(h):=s$ satisfies $\varphi_{s}(g)\neq e$ for all $g\in K_{2}$ and $\left\Vert \sum_{g\in K_{1}}x(g)\lambda_{\varphi_{s}(g)}\right\Vert \leq\|x\|+\varepsilon$ for all $x:=\sum_{g\in K_{1}}x(g)\lambda_{g}\in\mathbb{C}[\Gamma\star\mathbb Z]$ with $(x(g))_{g\in K_{1}}\subseteq\mathbb{C}$ and $\sum_{g\in K_1}|x(g)|=1$.
\end{enumerate}
\end{proposition}

\begin{proof} We first show that (ii) implies (i) and (iii). For this purpose, assume that $\Gamma\star\mathbb{Z}$ is existentially \Cs-residually-$\Gamma$ and define $I:=\{(F,m) \mid F\subseteq\Gamma\mathbin{\star}\mathbb Z\text{ finite},\ m \in \mathbb{N}_{\geq1}\}$, ordered by inclusion on the first component, and with the standard order on the second entry. By the assumption, for each $(F,m) \in I$ there exists an element $g_{F,m}\in\Gamma$ such that $\Vert\sum_{g\in F}x(g)\lambda_{\varphi_{g_{F,m}}(g)}\|\leq\|x\|+m^{-1}$ for all elements $x:=\sum_{g\in F}x(g)\lambda_{g}\in\mathbb{C}[\Gamma\star\mathbb{Z}]$ with $\sum_{g\in F}|x(g)|=1$, where $\varphi_{g_{F,m}}\colon\Gamma\star\mathbb{Z}\rightarrow\Gamma$ is the group homomorphism uniquely determined by $\varphi_{g_{F,m}}|_{\Gamma}=\id_{\Gamma}$ and $\varphi_{g_{F,m}}(h):=g_{F,m}$. Write $g_i:=g_{F,m}$ for $i=(F,m)\in I$. The implication ``$\text{(ii)}\Rightarrow\text{(i)}$'' now follows from the observation that for a cofinal ultrafilter $\mathcal{U}$ on $I$, the conditions above imply that the map $\theta\colon \Csr(\Gamma\star\mathbb{Z})\to \Csr(\Gamma)^{\mathcal{U}}$ given by $\theta(\lambda_{g}):=(\lambda_{\varphi_{g_{i}}(g)})_{i\rightarrow\mathcal{U}}$ for $g\in\Gamma\star\mathbb{Z}$ is a well-defined unital $\ast$-homomorphism with $\theta|_{\Csr(\Gamma)}=\Delta$ and $\theta(\lambda_{h})=(\lambda_{g_{i}})_{i\rightarrow\mathcal{U}}$. Since $\Csr(\Gamma\star\mathbb Z)$ is simple and has a unique tracial state, $\theta$ is injective and trace-preserving. To confirm the implication ``$\text{(ii)}\Rightarrow\text{(iii)}$'', fix finite subsets $K_{1}\subseteq\Gamma\star\mathbb{Z}$, $K_{2}\subseteq(\Gamma\star\mathbb{Z})\setminus\{e\}$, and $\varepsilon>0$. By the trace preservation of $\theta$,
\[
\lim_{i\to\mathcal{U}}\delta_{\varphi_{g_{i}}(g),e}=(\tau_{\Gamma}^{\mathcal{U}}\circ\theta)(\lambda_{g})=(\tau_{\Gamma}\star\tau_{\mathbb{Z}})(\lambda_{g})=0\quad\text{for every }g\in K_{2},
\]
so that the set $I_{0}:=\bigcap_{g\in K_{2}}\left\{ i\in I\mid\varphi_{g_{i}}(g)\neq e\right\}$ belongs to $\mathcal{U}$. Note that the set \[
I_{1}:=\left\{(F,m)\in I\mid K_{1}\subseteq F,\,m^{-1}<\varepsilon\right\}
\]
contains a final segment of $I$, and thus belongs to $\mathcal{U}$ as well. Then for any $i\in I_{0}\cap I_{1}$, the corresponding element $g_{i}\in\Gamma$ satisfies (iii).

The implication ``$\text{(iii)}\Rightarrow\text{(ii)}$'' is true by definition. For the implication ``$\text{(i)}\Rightarrow\text{(ii)}$'', suppose that $(g_{i})_{i\in I}\subseteq\Gamma$ is a family witnessing inner \Cs-selflessness of $\Gamma$ with respect to an ultrafilter $\mathcal U$ on $I$, and denote by $\theta\colon \Csr(\Gamma\star\mathbb{Z})\hookrightarrow \Csr(\Gamma)^{\mathcal{U}}$ the corresponding $\ast$-embedding with $\theta|_{\Csr(\Gamma)}=\Delta$ and $\theta(\lambda_{h})=(\lambda_{g_{i}})_{i\rightarrow\mathcal{U}}$. Fix an arbitrary finite subset $F\subseteq\Gamma\star\mathbb{Z}$, and $\varepsilon>0$. Since the unit sphere of $\ell^{1}(F)$ is compact, there exist finitely many elements $x_{1},\ldots,x_{k}\in\mathbb{C}[\Gamma\star\mathbb{Z}]$ of the form $x_{l}:=\sum_{g\in F}x_{l}(g)\lambda_{g}$ with $\sum_{g\in F}|x_{l}(g)|=1$ such that every other element $x:=\sum_{g\in F}x(g)\lambda_{g}\in\mathbb{C}[\Gamma\star\mathbb{Z}]$ with $\sum_{g\in F}|x(g)|=1$ admits a number $1\leq l\leq k$ with $\sum_{g\in F}|x(g)-x_{l}(g)|<\frac{\varepsilon}{3}$. By $\lim_{i\to\mathcal{U}}\Vert\sum_{g\in F}x_{l}(g)\lambda_{\varphi_{g_{i}}(g)}\Vert=\|x_{l}\|$ for every $1\leq l\leq k$, the set
\[
\bigl\{ i \in I \bigm| \bigl| \bigl\| \textstyle\sum_{g \in F} x_l(g) \lambda_{\varphi_{g_i}(g)} \bigr\| - \|x_l\| \bigr| < \varepsilon/3 \bigr\}
\]
belongs to $\mathcal{U}$, and thereby
\[
\bigcap_{l=1}^{k} \bigl\{ i \in I \bigm| \bigl| \bigl\| \textstyle\sum_{g \in F} x_l(g) \lambda_{\varphi_{g_i}(g)} \bigr\| - \|x_l\| \bigr| < \varepsilon/3 \bigr\} \in \mathcal{U}.
\]
Fix an element $i$ belonging to this intersection. For $x$ and $1\leq l\leq k$ as above we then obtain
\begin{eqnarray*}
\left\Vert \sum_{g\in F}x(g)\lambda_{\varphi_{g_{i}}(g)}\right\Vert &\leq& \left\Vert \sum_{g\in F}x_{l}(g)\lambda_{\varphi_{g_{i}}(g)}\right\Vert +\sum_{g\in F}\left|x(g)-x_{l}(g)\right| \\
&<& \Vert x\Vert+\varepsilon,
\end{eqnarray*}
hence finishing the proof. \end{proof}

Lastly, we observe that inner \Cs-selflessness passes to directed unions of groups.

\begin{proposition}\label{prop:directed-union}
    Let $(\Gamma_j)_{j\in J}$ be a directed family of inner \Cs-selfless groups. Then $\Gamma=\bigcup_{j\in J}\Gamma_j$ is inner \Cs-selfless.
\end{proposition}

\begin{proof}
    Let $I:=\big\{(F,n) \mid F\subseteq\mathbb C[\Gamma\star\mathbb Z] \text{ finite}, \, n\in\mathbb N\setminus\{0\} \big\}$ be ordered by inclusion on the first variable, and with the standard order on the second. 
    Moreover, let $h$ be the generator of $\mathbb Z$.
    Then, notice that for any $i=(F,n) \in I$, there exists an index $j\in J$ such that $F\subseteq \mathbb C[\Gamma_j\star\mathbb Z]$ because the family is directed. Since $\Gamma_j$ is inner \Cs-selfless, we may choose some element $s_i\in\Gamma_j\subseteq\Gamma$ such that the $\ast$-homomorphism $\theta_i\colon \mathbb C[\Gamma\star\mathbb Z] \to \mathbb C[\Gamma]$ induced by the identity map on $\Gamma$ and $h\mapsto s_i$ satisfies 
    \[
        \|\theta_i(x)\|_{\Cs_r(\Gamma)} \leq \|x\|_{\Cs_r(\Gamma\star\mathbb Z)}+\frac1n \qquad \text{for all }x\in F,
    \]
    where we used that both embeddings $\Cs_r(\Gamma_j\star\mathbb Z) \hookrightarrow\Cs_r(\Gamma\star\mathbb Z)$ and $\Cs_r(\Gamma_j) \hookrightarrow\Cs_r(\Gamma)$ are isometric.
    Now, choose any cofinal ultrafilter $\mathcal U$ on $I$, and define 
    \[
        \Theta\colon \mathbb C[\Gamma\star\mathbb Z] \to \Cs_r(\Gamma)^{\mathcal U},\qquad \Theta(x)=\lim_{i\to\mathcal U}\theta_i(x).
    \]
    Then $\Theta$ extends to a $\ast$-homomorphism $\Cs_r(\Gamma\star\mathbb Z) \to \Cs_r(\Gamma)^{\mathcal U}$ by the norm bounds of the $\theta_i$ maps. Moreover, it is clear that $\Theta$ restricts to the diagonal embedding on $\Cs_r(\Gamma)$, and by definition $\lambda_h$ is mapped to a family of group unitaires. This establishes that $\Gamma$ is inner \Cs-selfless.
\end{proof}

\vspace{3mm}


\section{Invariant subalgebras of crossed products by inner \texorpdfstring{\Cs}{}-selfless groups}
\label{sec:invC*crossed}
The aim of the present section is to prove the following theorem.

\begin{theorem} \label{InvariantMainTheorem} Let $\alpha\colon\Gamma\curvearrowright A$ be an action of an inner \Cs-selfless discrete group $\Gamma$ on a unital \Cs-algebra $A$, and assume that the ultrafilter $\mathcal{U}$ on the set $I$ and the family $(g_{i})_{i\in I}\subseteq\Gamma$ witnessing inner \Cs-selflessness satisfy
\begin{eqnarray} \label{Compatibility}
\lim_{i\rightarrow\mathcal{U}}\left\Vert \alpha_{g_{i}}(a)-a\right\Vert =0\quad\text{for every }a\in A.
\end{eqnarray}
Furthermore, let $N\triangleleft\Gamma$ be a non-trivial normal subgroup and assume that $B\subseteq A\rtimes_{\alpha,\mathrm r}\Gamma$ is a (possibly non-unital) \Cs-subalgebra that is invariant under conjugation by $N$. Then, $\mathbb{E}_{A}(b\lambda_{g}^{\ast})\lambda_{g}\in B$ for every $b\in B$ and $g\in\Gamma$. \end{theorem}

The proof of Theorem \ref{InvariantMainTheorem} requires some preparation and will be provided in Subsection \ref{subsec:ProofInvariantMainTheorem}. Before establishing the respective ingredients, we investigate its applications.

\begin{remark}
If $\alpha\colon\Gamma\curvearrowright A$ is an action of an inner \Cs-selfless discrete group $\Gamma$ on a unital \Cs-algebra $A$, for which the image $\alpha(\Gamma) \subseteq \mathrm{Aut}(A)$ has compact closure in the point-norm topology, the compatibility hypothesis in Theorem \ref{InvariantMainTheorem} is automatic, i.e., there exists an ultrafilter $\mathcal{U}$ and a family witnessing inner \Cs-selflessness of $\Gamma$  and satisfying \eqref{Compatibility}. Indeed, choose an ultrafilter $\mathcal{V}$ and a family $(h_i)_{i \in I} \subseteq \Gamma$ witnessing inner \Cs-selflessness of $\Gamma$. By the compactness of $\overline{\alpha(\Gamma)}$ there exists an element $\beta\in\mathrm{Aut}(A)$ with $\lim_{i\rightarrow \mathcal{V}}\alpha_{h_i}(a)=\beta(a)$ for every $a \in A$. Let $\mathcal U:=\mathcal V\otimes\mathcal V$ be the Fubini ultrafilter on $I \times I$, defined by
\[
S\in\mathcal U
\quad\Longleftrightarrow\quad
\bigl\{j\in I:
       \{i\in I:(i,j)\in S\}\in\mathcal V\bigr\}
\in\mathcal V,
\]
and put $g_{(i,j)}:=h_i h_j^{-1}$ for $i,j \in I$. Then,
\[
\begin{aligned}
\lim_{(i,j)\to\mathcal U}
 \|\alpha_{g_{(i,j)}}(a)-a\|
&=
\lim_{j\to\mathcal V}\lim_{i\to\mathcal V}
 \|\alpha_{h_{i}}(\alpha_{h_{j}^{-1}}(a))-a\|\\
&=
\lim_{j\to\mathcal V}
 \|\beta(\alpha_{h_{j}^{-1}}(a))-a\|
=0
\end{aligned}
\]
for every $a\in A$. One checks that the family $(g_{(i,j)})_{(i,j)\in I \times I}$ and the ultrafilter $\mathcal{U}$ also witness inner \Cs-selflessness of $\Gamma$. This implies the claim.

Similarly, if $\ker(\alpha)\ne\{e\}$, Lemma
\ref{NormalWitnessLemma} provides a witness contained in
$\ker(\alpha)$, for which the compatibility hypothesis in \eqref{Compatibility} holds exactly.
\end{remark}

We now apply the Fourier coefficient extraction theorem to invariant subalgebra rigidity.

\begin{definition}[{\cite[Definition 1.1]{AmrutamJiang26} and \cite[Definition 1.1]{amrutam2026relative}}] Let $\Gamma$ be a discrete group.
\begin{enumerate}[label=\textup{(\roman*)},leftmargin=*]
\item The group is said to have the \emph{\Cs-invariant subalgebra rigidity property} (\emph{\Cs-ISR property}) if every unital \Cs-subalgebra $A \subseteq \Csr (\Gamma)$ that is invariant under conjugation by $\Gamma$ is of the form $A=\Csr(N)$ for some normal subgroup $N\triangleleft\Gamma$.
\item The group is said to have the\emph{ relative \Cs-invariant subalgebra rigidity property} (\emph{relative \Cs-ISR property}) if, for every non-trivial normal subgroup $N\triangleleft\Gamma$, every unital \Cs-subalgebra $A\subseteq \Csr(\Gamma)$ that is invariant under conjugation by $N$ is of the form $A=\Csr(\Lambda)$ for some subgroup $\Lambda\leq\Gamma$.
\end{enumerate}
\end{definition}

\begin{corollary} \label{rCISR-Corollary}Let $\Gamma$ be an inner \Cs-selfless discrete group. Then $\Gamma$ has the relative \Cs-ISR property. \end{corollary}

\begin{proof}
Let $N\triangleleft\Gamma$ be a non-trivial normal subgroup and let
$A \subseteq \Csr(\Gamma)$ be a unital
$N$-invariant \Cs-subalgebra. Since $A$ is unital and closed under multiplication and
adjoints,  $\Lambda:=\{g\in\Gamma \mid \lambda_g\in A\}$ is a subgroup of $\Gamma$ with $\Csr(\Lambda)\subseteq A$. Furthermore, by applying Theorem \ref{InvariantMainTheorem} to the trivial action on $\mathbb C$, we have that $\tau_\Gamma(a\lambda_g^{\ast})\lambda_g\in A$ for all $a\in A$ and $g\in \Gamma$. It follows that for every element $a\in A$ the equality
$\tau_\Gamma(a\lambda_g^{\ast})=0$ holds whenever $g\notin\Lambda$.
Consequently, $a=\mathbb E_\Lambda(a)\in \Csr(\Lambda)$ and hence $A=\Csr(\Lambda)$.
\end{proof}

We next pass to finite products.

\begin{corollary} \label{C-ISR} Let $\Gamma_1,\ldots,\Gamma_n$ be inner \Cs-selfless discrete groups. Then the product $\Gamma_1 \times \dots\times \Gamma_n$ has the \Cs-ISR property. \end{corollary}

\begin{proof}
Let $\Gamma:=\Gamma_1\times\cdots\times\Gamma_n$, and let $A$ be a unital \Cs-subalgebra of $\Csr(\Gamma)$ that is invariant under conjugation by $\Gamma$. For each integer $1 \leq j \leq n$ there exists a canonical isomorphism
\[
\Csr(\Gamma)
\cong
\Csr\Bigl(\prod_{k\ne j}\Gamma_k\Bigr)
\rtimes_{\mathrm{id},\mathrm r}\Gamma_j,
\]
where $\id\colon \Gamma_j \curvearrowright \Csr\Bigl(\prod_{k\ne j}\Gamma_k\Bigr)$ denotes the trivial action. Since $A$ is in particular invariant under the conjugation by $\Gamma_j$, we may apply Theorem \ref{InvariantMainTheorem} with $N:=\Gamma_j$. Successively extracting the Fourier coefficient in each coordinate, we obtain that $\tau_\Gamma(a\lambda_g^{\ast})\lambda_g\in A$ for all $a\in A$ and $g\in\Gamma$. 

Since $A$ is unital and $\Gamma$-invariant, $\Lambda:=\{g\in\Gamma \mid \lambda_g\in A\}$ is a normal subgroup
of $\Gamma$ with $\Csr(\Lambda)\subseteq A$. By the above, $\tau_\Gamma(a\lambda_g^{\ast})=0$ holds whenever $a \in A$ and $g\notin\Lambda$. Consequently, $a=\mathbb E_\Lambda(a)\in \Csr(\Lambda)$ and hence $A=\Csr(\Lambda)$.
\end{proof}
To describe invariant subalgebras of crossed products, we first recall the approximation property.

For a finitely supported complex-valued function $\varphi$ on a discrete group $\Gamma$ we denote the corresponding completely bounded multiplier on $\Csr(\Gamma)$ given by $\lambda_{g}\mapsto\varphi(g)\lambda_{g}$ by $m_{\varphi}$. Recall that a discrete group $\Gamma$ is said to have the \emph{approximation property} if there exists a net $(\varphi_{i})_{i\in I}$ of finitely supported complex-valued functions on $\Gamma$ such that for any unital \Cs-algebra $A$ the net $(m_{\varphi_{i}}\otimes \id_{A})_{i\in I}$ of completely bounded maps on $\Csr(\Gamma)\otimes A$ converges to the identity in the pointwise norm topology; we refer the reader to \cite[Chapter 12]{BrownOzawa08} for more details.

\begin{corollary} Under the hypotheses of Theorem~\ref{InvariantMainTheorem}, the following statements hold.
\begin{enumerate}[label=\textup{(\roman*)},leftmargin=*]
\item If $\Gamma$ satisfies the approximation property, then $B$ is of the form
\[
B=\overline{\mathrm{span}}\left\{a\lambda_g \mid g\in\Gamma,a\in A\cap B\lambda_g^\ast \right\}.
\]
\item If $A$ is simple with $A\subseteq B$, then there exists a subgroup $\Lambda\leq\Gamma$ normalized by $N$ such that $B=A\rtimes_{\alpha|_{\Lambda},\mathrm r}\Lambda$.
\end{enumerate}
\end{corollary}

\begin{proof}
\emph{About (i)}: Choose a net $(\varphi_i)_{i \in I}$ witnessing the approximation property of $\Gamma$. By \cite[Proposition 4.1.7]{BrownOzawa08} there exists a $\ast$-embedding $\delta\colon A\rtimes_{\alpha,\mathrm r}\Gamma
\hookrightarrow
\Csr(\Gamma)\otimes_{\min}
(A\rtimes_{\alpha,\mathrm r}\Gamma)$ given by $\delta(a\lambda_g)=\lambda_g\otimes a\lambda_g$ for $a\in A$, $g \in \Gamma$. For $b\in B$ we have that
\begin{eqnarray*}
\left\Vert b-\sum_{g\in\text{supp}(\varphi_{i})}\varphi_{i}(g)\mathbb{E}_{A}(b\lambda_{g}^{\ast})\lambda_{g}\right\Vert &=& \left\Vert \delta(b)-\delta\left(\sum_{g\in\text{supp}(\varphi_{i})}\varphi_{i}(g)\mathbb{E}_{A}(b\lambda_{g}^{\ast})\lambda_{g}\right)\right\Vert \\
&=& \left\Vert \delta(b)-(m_{\varphi_{i}}\otimes \id)\bigl(\delta(b)\bigr)\right\Vert \rightarrow0
\end{eqnarray*}
By Theorem \ref{InvariantMainTheorem},
$\mathbb E_A(b\lambda_g^\ast)\in A\cap B\lambda_g^\ast$,
so $b$ belongs to the closed linear span of all elements $a\lambda_g$ with $g\in \Gamma$ and $a \in A\cap B\lambda_g^\ast$.
The reverse inclusion is obvious.\\

\emph{About (ii)}: By $A\subseteq B$, each subspace of the form $A\cap B\lambda_g^\ast \subseteq A$ with $g \in \Gamma$ is a closed two-sided ideal of $A$, so that either $A\cap B\lambda_g^\ast=\{0\}$ or $A\cap B\lambda_g^\ast=A$.
Our assumptions moreover imply that $\Lambda:=\{g\in\Gamma\mid\lambda_g\in B\}$ is a subgroup of $\Gamma$ normalized by $N$ and satisfying
$A\rtimes_{\alpha|_\Lambda,\mathrm r}\Lambda\subseteq B$. For $g\notin \Lambda$ the \Cs-algebra $A\cap B\lambda_g^\ast$ cannot contain the unit element of $A$, so that $A\cap B\lambda_g^\ast=\{0\}$. With Theorem \ref{InvariantMainTheorem} it therefore follows that $\mathbb E_A(b\lambda_g^\ast)=0$ for all $b\in B$ and $g \notin \Lambda$. Hence,
$b=\mathbb E_\Lambda(b)$ for every $b\in B$, proving that $B \subseteq A\rtimes_{\alpha |_\Lambda, \mathrm r}\Lambda$. The converse inclusion is clear. \end{proof}

The following definition is a natural generalisation of \cite[Definition 1.2]{amrutam2024crossed} to actions on
arbitrary unital \Cs-algebras.

\begin{definition}[{\cite[Definition 1.2]{amrutam2024crossed}}] Let $\alpha\colon\Gamma\curvearrowright A$ be an action of a discrete group on a unital \Cs-algebra $A$.
\begin{enumerate}[label=\textup{(\roman*)},leftmargin=*]
\item The crossed product $A\rtimes_{\alpha,\mathrm r}\Gamma$ is called \emph{reflecting} if every intermediate \Cs-algebra $\Csr(\Gamma)\subseteq B\subseteq A\rtimes_{\alpha,\mathrm r}\Gamma$ is of the form $B_{0}\rtimes_{\alpha,\mathrm r}\Gamma$ for a unital $\Gamma$-invariant \Cs-subalgebra $B_{0}\subseteq A$.
\item The crossed product $A\rtimes_{\alpha,\mathrm r}\Gamma$ is called \emph{almost reflecting} if $\mathbb{E}_{A}(B)\subseteq B$ for every intermediate \Cs-algebra $\Csr(\Gamma)\subseteq B\subseteq A\rtimes_{\alpha,\mathrm r}\Gamma$.
\end{enumerate}
\end{definition}

The following statement generalises \cite[Theorems 1.1 and 1.2]{amrutam2021intermediate}.

\begin{corollary} 
\label{CReflectingCorollary}
Let $\alpha\colon\Gamma\curvearrowright A$ be an action of an inner \Cs-selfless discrete group $\Gamma$ on a unital \Cs-algebra $A$, and assume that the ultrafilter $\mathcal{U}$ on $I$ and the family $(g_{i})_{i\in I}\subseteq\Gamma$ witnessing inner \Cs-selflessness satisfy $\lim_{i\rightarrow\mathcal{U}}\left\Vert \alpha_{g_{i}}(a)-a\right\Vert =0$ for every $a\in A$. Then the crossed product $A\rtimes_{\alpha,\mathrm r}\Gamma$ is almost reflecting. Furthermore, if $\Gamma$ satisfies the approximation property, the crossed product is reflecting.
\end{corollary}

\begin{proof}
Let
$\Csr(\Gamma)\subseteq B
\subseteq A\rtimes_{\alpha,\mathrm r}\Gamma$
be an intermediate \Cs-algebra. Since $\lambda_g\in B$
for every $g\in\Gamma$, the \Cs-algebra $B$ is invariant under conjugation by $\Gamma$. Applying Theorem
\ref{InvariantMainTheorem} gives that
$\mathbb E_A(B)\subseteq B$, proving that the crossed product is almost reflecting.

The unital \Cs-subalgebra $B_0:=A\cap B$ is $\Gamma$-invariant. If $\Gamma$ has the approximation property, then
$B\lambda_g^\ast=B$ for every $g\in\Gamma$, so that part (i) of the preceding corollary yields
\[
B=\overline{\mathrm{span}}
\{a\lambda_g\mid a\in B_0,\ g\in\Gamma\}
=B_0\rtimes_{\alpha,\mathrm r}\Gamma.
\]
Thus the crossed product is reflecting.
\end{proof}


\subsection{Selective averaging}
To prove Theorem~\ref{InvariantMainTheorem}, we first establish a selective averaging result.
The following proposition is a refinement of \cite[Proposition 3.2]{AmrutamJiang26}.

\begin{proposition} \label{AveragingProposition} Let $\Gamma$ be a discrete group acting continuously on a compact Hausdorff space $X$. Assume that $g\in\Gamma$ is an infinite order element for which there exist distinct points $z_{+},z_{-}\in X$ satisfying the following conditions:
\begin{itemize}
\item $g^{n}.x\rightarrow z_{+}$ for every $x\in X\setminus\{z_{-}\}$;
\item $g^{-n}.x\rightarrow z_{-}$ for every $x\in X\setminus\{z_{+}\}$;
\item $\{h.z_{+},h.z_{-}\}\cap\{z_{+},z_{-}\}=\emptyset$ for every $h\in\Gamma\setminus\langle g\rangle$.
\end{itemize}
Then, $\mathbb{E}_{g}(x)$ is for every $x\in \Csr(\Gamma)$ contained in the norm closure of the convex hull of the set $\{\lambda_{g^{n}}x\lambda_{g^{n}}^{\ast}\mid n\geq1\}$. \end{proposition}

\begin{proof} We view $\Csr(\Gamma)$ as a \Cs-subalgebra of the reduced crossed product $C(X)\rtimes_{\mathrm r}\Gamma$. Let $x\in \Csr(\Gamma)$ and assume that $\mathbb{E}_{g}(x)$ is not contained in the norm closure of the convex hull of the set $\{\lambda_{g^{n}}x\lambda_{g^{n}}^{\ast}\mid n\geq1\}$. Then the Hahn-Banach separation theorem provides a functional $\omega$ on $C(X)\rtimes_{\mathrm r}\Gamma$ with 
\begin{equation}
\operatorname{Re}\omega(\mathbb E_g(x))>
\sup_{n\geq1}\operatorname{Re}\omega(\lambda_{g^n}x\lambda_{g^n}^{\ast}).\label{eq:Separation}
\end{equation}
Write $\omega=\sum_{i=1}^4c_i\omega_i$, where each $\omega_i$ is a state, and put $\nu_i:=\omega_i|_{C(X)}$ and $\omega_{i,n}(y):=\omega_i(\lambda_{g^n}y\lambda_{g^n}^{\ast})$. By dominated convergence,
\[
\omega_{i,n}|_{C(X)}\longrightarrow
\nu_i(X\setminus\{z_+\})\delta_{z_-}+\nu_i(\{z_+\})\delta_{z_+}
\]
in the weak$^{\ast}$ topology. Passing to a common subnet, assume that $\omega_{i,n}\to\overline\omega_i$ weak$^{\ast}$ for every $i$. Each $\overline\omega_i$ is a state whose restriction to $C(X)$ is supported on $\{z_+,z_-\}$. By \cite[Lemma 3.1]{AmrutamJiang26}, $\overline\omega_i(\lambda_h)=0$ for $h\in\Gamma\setminus\langle g\rangle$. Thus $\overline\omega:=\sum_{i=1}^4c_i\overline\omega_i$ satisfies $\overline\omega|_{\Csr(\Gamma)}=\overline\omega\circ\mathbb E_g$.
Since $\mathbb E_g(x)$ commutes with $\lambda_g$, (\ref{eq:Separation}) gives
\begin{align*}
\operatorname{Re}\omega(\mathbb E_g(x))
&=\operatorname{Re}\overline\omega(x)\\
&\leq\sup_{n\geq1}\operatorname{Re}\omega(\lambda_{g^n}x\lambda_{g^n}^{\ast})
<\operatorname{Re}\omega(\mathbb E_g(x)),
\end{align*}
a contradiction. This proves the claim.
\end{proof}


\subsection{The free product case}

We first prove Fourier coefficient extraction for a free product in which the added generator acts trivially.

In the context of the present subsection, we fix a free product $\Gamma=\Lambda\star\mathbb{Z}$ of a non-trivial discrete group $\Lambda$ by the integers. Let $h\in\mathbb{Z}$ be a generator of $\mathbb{Z}$, and let $\alpha\colon\Gamma\curvearrowright A$ be an action of $\Gamma$ on a unital \Cs-algebra $A$ with $\alpha_{h}= \id_{A}$. Furthermore, assume that $B$ is a \Cs-subalgebra of $A\rtimes_{\alpha,\mathrm r}\Gamma$ that is invariant under conjugation by the group $L$ generated by all elements of the form $g^{-1}hg$ with $g\in\Lambda$. Observe that $L$ is the kernel of the group homomorphism $\Lambda\star\mathbb{Z}\twoheadrightarrow\Lambda$ restricting to the identity on $\Lambda$ and mapping $h$ to the neutral element; in particular, $L$ is a normal subgroup of $\Gamma$. It can furthermore be expressed as $L\cong\star_{g\in\Lambda}g^{-1}\langle h\rangle g\cong\star_{g\in\Lambda}\mathbb{Z}$ and is thereby inner \Cs-selfless.

For every $g\in\Lambda\setminus\{e\}$ define 
\[
\mathbf{w}(g):=g(hg^{-1}h^{-1})(h^{2}gh^{-2})(h^{3}g^{-1}h^{-3})\in L
\]
and denote the normal subgroup of $L$ generated by this element by $\langle\mathbf{w}(g)\rangle^{L}$. Note that by our assumptions the restriction of $\alpha$ to the subgroup $L$ is trivial, so that in particular $[a,\lambda_{\mathbf{w}(g)}]=0$ for every $a\in A$ and $g\in\Lambda\setminus\{e\}$.

\begin{lemma} \label{ElementPropertiesLemma} For every element $g\in\Lambda\setminus\{e\}$ the following statements hold:
\begin{enumerate}[label=\textup{(\roman*)},leftmargin=*]
\item $\mathbf{w}(g)$ is primitive in $\Gamma$ (i.e., it is not a proper power).
\item $\mathbf{w}(g)$ is not conjugate to $\mathbf{w}(g)^{-1}$.
\item For $f_{1},\ldots,f_{4}\in\Lambda\setminus\{e\}$ the element $f_{1}hf_{2}^{-1}hf_{3}hf_{4}^{-1}h^{-3}$ is contained in $\langle\mathbf{w}(g)\rangle$ if and only if $f_{1}=f_{2}=f_{3}=f_{4}=g$.
\end{enumerate}
 \end{lemma}

\begin{proof} The expression $\mathbf{w}(g)=ghg^{-1}hghg^{-1}h^{-3}$ is cyclically reduced, has eight syllables, and has exactly one $\langle h\rangle$-syllable equal to $h^{-3}$. Since the syllable sequence of a cyclically reduced proper power is periodic, this proves (i). Furthermore, by the conjugacy theorem for free products, $\mathbf{w}(g)$ cannot be conjugate to $\mathbf{w}(g)^{-1}$, implying (ii).

For $f_{1},\ldots,f_{4}\in\Lambda\setminus\{e\}$ the expression $f_{1}hf_{2}^{-1}hf_{3}hf_{4}^{-1}h^{-3}$ is cyclically reduced with eight syllables. If $f_{1}hf_{2}^{-1}hf_{3}hf_{4}^{-1}h^{-3}=\mathbf{w}(g)^{r}$ for some $r\in\mathbb{Z}$, the reduced syllable lengths imply that $|r|=1$. The case $r=(-1)$ is excluded by the $\langle h\rangle$-syllables, and comparison of reduced forms for $r=1$ gives $f_{1}=f_{2}=f_{3}=f_{4}=g$. The converse implication is clear by the definition of $\mathbf{w}(g)$.  \end{proof}

We also need averaging over non-trivial normal subgroups of $L$.

\begin{lemma} \label{PlumpSubgroupLemma} Every non-trivial normal subgroup $N$ of $L$ is \emph{plump} in $\Gamma$, meaning that for any $\varepsilon>0$ and any finite subset $F\subseteq\Gamma\setminus\{e\}$, there exist elements $h_{1},\ldots,h_{m}\in N$ such that
\[
\left\Vert \frac{1}{m}\sum_{i=1}^{m}\lambda_{h_{i}gh_{i}^{-1}}\right\Vert <\varepsilon \quad \text{ for all } g\in F.
\]
\end{lemma}

\begin{proof} Given a finite subset $F\subseteq\Gamma\setminus\{e\}$ and $n\in N\setminus\{e\}$, choose a non-trivial finitely generated subgroup $\Lambda_{0}\leq\Lambda$ containing all $\Lambda$-syllables occurring in the reduced expressions of elements of $F\cup\{n\}$, and put $\Gamma_{0}:=\Lambda_{0}\star\mathbb{Z}$. Then $F\subseteq\Gamma_{0}$ and $\{e\}\neq N\cap\Gamma_{0}\triangleleft L\cap\Gamma_{0}\triangleleft\Gamma_{0}$, where $L\cap\Gamma_{0}$ is the kernel of the canonical retraction onto $\Lambda_{0}$. Since $\Lambda_{0}$ is countable, we may thereby without loss of generality assume that $\Lambda$ itself is countable.

The action of $\Gamma$ on its Bass--Serre tree $T$ is faithful, minimal, and of general type in the sense of \cite{LeBoudec16}; see \cite[Proposition~18\textup{(i)--(ii)}]{deLaHarpePreaux11}. By \cite[Lemma~4.2]{LeBoudec16}, applied first to $L\triangleleft\Gamma$ and then to $N\triangleleft L$, the same holds for $N$. In particular, $N$ contains two loxodromic elements with disjoint pairs of axis endpoints.

The $N$-action on the compactification $VT\cup\partial T$ of $T$ (see Subsection \ref{subsec:free-products}) has no fixed point and is strongly proximal. Indeed, positive powers of the first loxodromic element send any probability measure in the weak$^{\ast}$-limit to a measure supported on its two endpoints; positive powers of the second then move this measure to a Dirac mass. Moreover, by the discussion in Subsection \ref{subsec:free-products}, every non-identity element of $\Gamma$ has at most two fixed points in $VT\cup\partial T$. Thus \cite[Lemma~3.5]{amrutam2021intermediate} implies that $N$ is plump in $\Gamma$. \end{proof}

For notational convenience, we frequently write $x(g):=\mathbb{E}_{A}(x\lambda_{g}^{\ast})$ for $g\in\Gamma$ and $x\in A\rtimes_{\alpha,\mathrm r}\Gamma$, where $\mathbb{E}_{A}\colon A\rtimes_{\alpha,\mathrm r}\Gamma\rightarrow A$ is the canonical conditional expectation.

First observe that $\mathbb E_A(B)\subseteq B$. Indeed, let $x\in B$ and $\varepsilon>0$, and choose $y=\sum_{f\in F}a_f\lambda_f$ with $\|x-y\|<\varepsilon$. By Lemma~\ref{PlumpSubgroupLemma}, choose an average $T=\frac1m\sum_{i=1}^m\operatorname{Ad}(\lambda_{h_i})$, with $h_i\in L$, such that $\|T(\lambda_f)\|<\varepsilon/(1+\sum_{f\in F}\|a_f\|)$ for $f\in F\setminus\{e\}$. Since $L$ acts trivially on $A$,
\[
\|T(x)-\mathbb E_A(x)\|
\leq2\|x-y\|+\sum_{f\in F\setminus\{e\}}\|a_f\|\,\|T(\lambda_f)\|
<3\varepsilon.
\]
Moreover, $T(x)\in B$ by $L$-invariance. Closedness gives $\mathbb E_A(x)\in B$.

We first use selective averaging to construct an element of $B$ supported on the word $\mathbf{w}(g)$.

\begin{proposition} \label{FirstExtractionStep} Let $x\in B\cap(A\rtimes_{\alpha|_{\Lambda},\mathrm r}\Lambda)$ and $g\in\Lambda\setminus\{e\}$. Then, 
\begin{equation}
\left(x(g)x(g)^{\ast}\right)^{2}\lambda_{\mathbf{w}(g)}\in B.\label{eq:FirstExtractionStep}
\end{equation}
\end{proposition}

\begin{proof} Since $\mathbb E_A(x)\in B$ by the preceding argument, replacing $x$ by $x-\mathbb E_A(x)$ preserves membership in $B$ and does not change $x(g)$. Thus we may assume that $\mathbb E_A(x)=0$. For every $y\in A\rtimes_{\alpha,\mathrm r}\Gamma$ define 
\[
\mathbf{P}(y):=y\left(\lambda_{h}y^{\ast}\lambda_{h}^{\ast}\right)\left(\lambda_{h^{2}}y\lambda_{h^{2}}^{\ast}\right)\left(\lambda_{h^{3}}y^{\ast}\lambda_{h^{3}}^{\ast}\right)\in A\rtimes_{\alpha,\mathrm r}\Gamma,
\]
and note that $\mathbf{P}(x)\in B$.

By Subsection \ref{subsec:free-products} the cyclically reduced word $\mathbf{w}(g)$ acts loxodromically on the Bass--Serre tree $T$. Its two poles $z_{+},z_{-}$ belong to the compact Hausdorff space $VT\cup\partial T$, on which $\Gamma$ acts by homeomorphisms, and its powers have the convergence required in Proposition \ref{AveragingProposition}. By Subsection \ref{subsec:free-products}, two conjugates of $\mathbf{w}(g)$ with a common pole agree up to inversion, since they have the same translation length. Thus an element taking a pole to a pole preserves the axis. By Lemma~\ref{ElementPropertiesLemma}, the orientation-preserving axis stabilizer is $\langle\mathbf w(g)\rangle$, and no element reverses the axis. Hence Proposition~\ref{AveragingProposition} applies.

The averaging conclusion also holds in the coefficient crossed product. For finitely many group elements outside $\langle\mathbf w(g)\rangle$, choose one averaging map for each element. These maps are convex combinations of positive powers of $\operatorname{Ad}(\lambda_{\mathbf w(g)})$ and therefore commute. Their composition makes all the corresponding group terms simultaneously small, by contractivity. They fix $A$, since $\alpha_{\mathbf w(g)}=\mathrm{id}_A$, and fix every term supported on $\langle\mathbf w(g)\rangle$. Applying the triangle inequality to finite Fourier sums, and then using contractivity and density, gives convergence to $\mathbb E_{\mathbf w(g)}$ on $A\rtimes_{\alpha,\mathrm r}\Gamma$. These averages preserve $B$ because $\mathbf w(g)\in L$. Thus $\mathbb E_{\mathbf w(g)}(B)\subseteq B$, and in particular $\mathbb E_{\mathbf w(g)}(\mathbf P(x))\in B$.

We next calculate $\mathbb{E}_{\mathbf{w}(g)}(\mathbf{P}(x))$. For this, choose for every $k\geq1$ an element $x_{k}\in C_{c}(\Lambda,A)$ with $\Vert x-x_{k}\Vert<\frac{1}{k}$. The element $x_{k}$ can be expressed as a sum of the form $x_{k}=\sum_{f\in F_{k}}x_{k}(f)\lambda_{f}$, where $F_{k}\subseteq\Lambda$ is finite. Without loss of generality, we can assume that $e\notin F_{k}$ and $g\in F_{k}$ with $x_{k}(g)=x(g)$. The element $\mathbf{P}(x_{k})$ is given by
\begin{eqnarray*}
\mathbf{P}(x_{k})&=& \sum_{f_{1},f_{2},f_{3},f_{4}\in F_{k}}\left[x_{k}(f_{1})\alpha_{f_{1}hf_{2}^{-1}}\left(x_{k}(f_{2})^{\ast}\right)\alpha_{f_{1}hf_{2}^{-1}h}\left(x_{k}(f_{3})\right)\alpha_{f_{1}hf_{2}^{-1}hf_{3}hf_{4}^{-1}}\left(x_{k}(f_{4})^{\ast}\right)\right] \\
& & \qquad \qquad \qquad \qquad \qquad \qquad \qquad \qquad \qquad \qquad \qquad \qquad\qquad  \times \lambda_{f_{1}hf_{2}^{-1}hf_{3}hf_{4}^{-1}h^{-3}},
\end{eqnarray*}
so that by Lemma \ref{ElementPropertiesLemma} the equality 
\begin{eqnarray*}
\mathbb{E}_{\mathbf{w}(g)}(\mathbf{P}(x_{k})) &=& \left[x_{k}(g)\alpha_{ghg^{-1}}\left(x_{k}(g)^{\ast}\right)\alpha_{ghg^{-1}h}\left(x_{k}(g)\right)\alpha_{ghg^{-1}hghg^{-1}}\left(x_{k}(g)^{\ast}\right)\right]\lambda_{\mathbf{w}(g)} \\
&=& \left(x(g)x(g)^{\ast}\right)^{2}\lambda_{\mathbf{w}(g)}
\end{eqnarray*}
 holds, and so
\begin{eqnarray*}
\left\Vert \mathbb{E}_{\mathbf{w}(g)}(\mathbf{P}(x))-\left(x(g)x(g)^{\ast}\right)^{2}\lambda_{\mathbf{w}(g)}\right\Vert \leq \left\Vert \mathbf{P}(x)-\mathbf{P}(x_{k})\right\Vert \rightarrow 0.
\end{eqnarray*}
This implies the claim. \end{proof}

In a next step, we apply a functional calculus argument to the elements extracted in Proposition \ref{FirstExtractionStep}.

\begin{lemma} \label{FunctionalCalculusLemma} Given $\varepsilon>0$, $x\in B\cap(A\rtimes_{\alpha|_{\Lambda},\mathrm r}\Lambda)$ and $g\in\Lambda\setminus\{e\}$ one has that 
\[
\left(x(g)x(g)^{\ast}\right)^{2}\left(\left(x(g)x(g)^{\ast}\right)^{2}+\varepsilon1\right)^{-1}\lambda_{k}\in B\quad\text{for all }k\in\langle\mathbf{w}(g)\rangle^{L}.
\]
\end{lemma}

\begin{proof} Since $\lambda_{\mathbf w(g)}$ commutes with $A$, it commutes with $b:=(x(g)x(g)^*)^2$. Thus Proposition~\ref{FirstExtractionStep} gives $b=|b\lambda_{\mathbf w(g)}|\in B$. Conjugation by $l\in L$ fixes $b$, so $b\lambda_{l\mathbf w(g)^{\pm1}l^{-1}}\in B$. Writing $k\in\langle\mathbf w(g)\rangle^L$ as a non-empty product of these conjugates gives $b^n\lambda_k\in B$ for some $n\geq1$.

For $\delta>0$, define $f_\delta\colon[0,\|b\|]\to\mathbb R_{\geq0}$ by $f_\delta(z):=z[(z+\varepsilon)(z^n+\delta)]^{-1}$. Since $f_\delta(0)=0$, continuous functional calculus gives $f_\delta(b)\in B$. Moreover,
\[
\big\|(b^n\lambda_k)f_\delta(b)-b(b+\varepsilon1)^{-1}\lambda_k\big\|
\leq\varepsilon^{-1}\delta^{1/n}\longrightarrow0.
\]
It follows that $b(b+\varepsilon1)^{-1}\lambda_k\in B$, finishing the proof.
\end{proof}

By combining the tools established above with each other, we obtain the following main result of the present subsection.

\begin{theorem} \label{FreeProductCaseTheorem} Let $\Gamma=\Lambda\star\mathbb{Z}$ be a free product of a non-trivial discrete group $\Lambda$ by the integers, let $h\in\mathbb{Z}$ be a generator of $\mathbb{Z}$, and let $\alpha\colon\Gamma\curvearrowright A$ be an action of $\Gamma$ on a unital \Cs-algebra $A$ with $\alpha_{h}=\id_{A}$. Then, if $B$ is a \Cs-subalgebra of $A\rtimes_{\alpha,\mathrm r}\Gamma$ invariant under conjugation by the subgroup $\left\langle g^{-1}hg\mid g\in\Lambda\right\rangle \leq\Gamma$, one has that $\mathbb{E}_{A}(x\lambda_{g}^{\ast})\lambda_{g}\in B$ for every $x\in B\cap(A\rtimes_{\alpha|_{\Lambda},\mathrm r}\Lambda)$ and $g\in\Lambda$. \end{theorem}

\begin{proof} As before, write $L:=\langle g^{-1}hg\mid g\in\Lambda\rangle $ and let $x\in B\cap(A\rtimes_{\alpha|_{\Lambda},\mathrm r}\Lambda)$ be an arbitrary element. For every $k\geq1$ choose $x_{k}\in C_{c}(\Lambda,A)$ with $\Vert x-x_{k}\Vert<\frac{1}{4k}$ and express the element as a sum of the form $x_{k}=\sum_{f\in F_{k}}x_{k}(f)\lambda_{f}$, where $F_{k}\subseteq\Lambda$ is finite and contains $e$. For every $f\in F_{k}$ the corresponding coefficient satisfies
\begin{equation}
\left\Vert x_{k}(f)\right\Vert =\left\Vert \mathbb{E}_{A}(x_{k}\lambda_{f}^{\ast})\lambda_{f}\right\Vert \leq\left\Vert \mathbb{E}_{A}(x\lambda_{f}^{\ast})\lambda_{f}\right\Vert +\frac{1}{4k}<\Vert x\Vert+\frac{1}{k}.\label{eq:CoefficientInequality}
\end{equation}

We proceed by distinguishing two cases.\\

\begin{itemize}
\item \emph{Case 1}: Assume that $g=e$. The assertion is the inclusion $\mathbb E_A(x)\in B$ established before Proposition~\ref{FirstExtractionStep}.
\item \emph{Case 2}: Now assume that $g\in\Lambda\setminus\{e\}$, and define for every $\varepsilon>0$ elements 
\begin{eqnarray*}
p_{\varepsilon} &:=& \left(x(g)x(g)^{\ast}\right)^{2}\left(\left(x(g)x(g)^{\ast}\right)^{2}+\varepsilon1\right)^{-1} \\
q_{\varepsilon} &:=& \alpha_{g^{-1}}\left(\left(x(g)^{\ast}x(g)\right)^{2}\left(\left(x(g)^{\ast}x(g)\right)^{2}+\varepsilon1\right)^{-1}\right),
\end{eqnarray*}
so that $p_{\varepsilon}x(g)=x(g)\alpha_{g}(q_{\varepsilon})$. By $L\triangleleft\Gamma$, the group $g^{-1}\langle\mathbf w(g)\rangle^{L}g$ is a non-trivial normal subgroup of $L$. Since $L\cong\star_{g\in\Lambda}\mathbb{Z}$ is inner \Cs-selfless, it follows with Proposition \ref{InnerPropertiesProposition} that the intersection
\[
N:=\langle\mathbf{w}(g^{-1})\rangle^{L}\cap\left(g^{-1}\langle\mathbf w(g)\rangle^{L}g\right)
\]
is also a non-trivial normal subgroup of $L$. Lemma \ref{PlumpSubgroupLemma} hence asserts that $N$ is a plump subgroup of $\Gamma$. Apply Lemma~\ref{FunctionalCalculusLemma} to $x$ at $g$ and to $x^*$ at $g^{-1}$. Since $x^*(g^{-1})=\alpha_{g^{-1}}(x(g)^*)$, we obtain $p_\varepsilon\lambda_{gng^{-1}}\in B$ and $q_\varepsilon\lambda_n\in B$ for every $n\in N$. \\
By the above, we find elements $n_{1},\ldots,n_{m_{k}}\in N$ with $\Vert\frac{1}{m_{k}}\sum_{i=1}^{m_{k}}\lambda_{n_{i}(g^{-1}f)n_{i}^{-1}}\Vert<\frac{1}{k(\#F_{k})}$ for every $f\in F_{k}\setminus\{g\}$. Thereby,
\begin{eqnarray*}
& & \left\Vert \frac{1}{m_{k}}\sum_{i=1}^{m_{k}}p_{\varepsilon}\lambda_{gn_{i}g^{-1}}x_{k}q_{\varepsilon}\lambda_{n_{i}^{-1}}-p_{\varepsilon}x_{k}(g)\alpha_{g}(q_{\varepsilon})\lambda_{g}\right\Vert \\
& & \qquad = \left\Vert \sum_{f\in F_{k}\setminus\{g\}}\left(p_{\varepsilon}x_{k}(f)\alpha_{f}(q_{\varepsilon})\right)\lambda_{g}\left(\frac{1}{m_{k}}\sum_{i=1}^{m_{k}}\lambda_{n_{i}(g^{-1}f)n_{i}^{-1}}\right)\right\Vert \\
& & \qquad \qquad \leq \frac{1}{k(\#F_{k})}\sum_{f\in F_{k}\setminus\{g\}}\left\Vert p_{\varepsilon}x_{k}(f)\alpha_{f}(q_{\varepsilon})\right\Vert \leq \sum_{f\in F_{k}\setminus\{g\}}\frac{\left\Vert x_{k}(f)\right\Vert }{k(\#F_{k})} \\
& & \qquad \qquad \qquad \leq \frac{1}{k}\left(\Vert x\Vert+\frac{1}{k}\right)
\end{eqnarray*}
and so 
\begin{eqnarray*}
& & \left\Vert \frac{1}{m_{k}}\sum_{i=1}^{m_{k}}p_{\varepsilon}\lambda_{gn_{i}g^{-1}}xq_{\varepsilon}\lambda_{n_{i}^{-1}}-p_{\varepsilon}x(g)\alpha_{g}(q_{\varepsilon})\lambda_{g}\right\Vert \\
&\leq& 2\left\Vert x-x_{k}\right\Vert +\left\Vert \frac{1}{m_{k}}\sum_{i=1}^{m_{k}}p_{\varepsilon}\lambda_{gn_{i}g^{-1}}x_{k}q_{\varepsilon}\lambda_{n_{i}^{-1}}-p_{\varepsilon}x_{k}(g)\alpha_{g}(q_{\varepsilon})\lambda_{g}\right\Vert \\
&<& \frac{1}{k}\left(\Vert x\Vert+\frac{1}{k}+1\right)\rightarrow0.
\end{eqnarray*}
For every $1\leq i\leq m_{k}$ the element $p_{\varepsilon}\lambda_{gn_{i}g^{-1}}xq_{\varepsilon}\lambda_{n_{i}^{-1}}$ is contained in $B$, so $p_{\varepsilon}x(g)\alpha_{g}(q_{\varepsilon})\lambda_{g}\in B$. The identity $p_{\varepsilon}x(g)=x(g)\alpha_{g}(q_{\varepsilon})$ finally gives 
\begin{eqnarray*}
& & \left\Vert p_{\varepsilon}x(g)\alpha_{g}(q_{\varepsilon})\lambda_{g}-x(g)\lambda_{g}\right\Vert ^{2} \leq \left\Vert x(g)\alpha_{g}(q_{\varepsilon}^{2})-x(g)\right\Vert ^{2} \\
& &  \qquad = \left\Vert \left(x(g)^{\ast}x(g)\right)\alpha_{g}(q_{\varepsilon}^{4})-2\left(x(g)^{\ast}x(g)\right)\alpha_{g}(q_{\varepsilon}^{2})+\left(x(g)^{\ast}x(g)\right)\right\Vert \\
& & \qquad \qquad \leq \sup_{z\in[0,\Vert x\Vert^2]}\left|\frac{z^{9}}{(z^{2}+\varepsilon)^{4}}-\frac{2z^{5}}{(z^{2}+\varepsilon)^{2}}+z\right|\\
& & \qquad \qquad \leq4\sqrt\varepsilon,
\end{eqnarray*}
Here the last bound follows, for $z\geq0$, from
\[
z\left(1-\left(\frac{z^2}{z^2+\varepsilon}\right)^2\right)^2
\leq\frac{4\varepsilon^2z}{(z^2+\varepsilon)^2}
\leq4\sqrt\varepsilon.
\]
Thus $x(g)\lambda_{g}=\lim_{\varepsilon\downarrow0}p_{\varepsilon}x(g)\alpha_{g}(q_{\varepsilon})\lambda_{g}\in B.$
\end{itemize}
This finishes the proof. \end{proof}
For the ultrapower argument below, we also need the following extension of Theorem~\ref{FreeProductCaseTheorem} to a quotient of a full amalgamated free product.
\begin{remark} \label{UniversalFreeProductRemark} In the setting of Theorem \ref{FreeProductCaseTheorem}, let $D$ be the quotient of the full amalgamated free product $(A\rtimes_{\alpha|_{\Lambda},\mathrm{r}}\Lambda)\star_{\Csr(\Lambda)}\Csr(\Gamma)$ by the two-sided closed ideal generated by all commutators of the form $[a,\lambda_{h}]$ with $a\in A$. The restriction of the canonical quotient map $q\colon D\twoheadrightarrow A\rtimes_{\alpha,\mathrm{r}}\Gamma$ to both $A\rtimes_{\alpha|_{\Lambda},\mathrm{r}}\Lambda$ and $\Csr(\Gamma)$ is faithful. Identifying $A$ with its canonical image in $D$, the map
$\widetilde{\mathbb E}_A:=\mathbb E_A\circ q$ is a conditional expectation from $D$ onto $A$. The conclusion of Theorem \ref{FreeProductCaseTheorem} then remains valid for $L$-invariant $\Cs$-subalgebras of $D$, i.e., if $B\subseteq D$ is a $\Cs$-subalgebra invariant under conjugation by $L:=\langle g^{-1}hg\mid g\in\Lambda\rangle$, then $\widetilde{\mathbb{E}}_{A}(x\lambda_{g}^{*})\lambda_{g}\in B$ for every $x\in B\cap(A\rtimes_{\alpha|_{\Lambda},\mathrm{r}}\Lambda)\subseteq D$ and $g\in\Lambda$.

Indeed, for $g\in\Lambda\setminus\{e\}$ the amenability of $\langle\mathbf{w}(g)\rangle$ and the triviality of $\alpha|_{\langle\mathbf{w}(g)\rangle}$ imply that $q$ restricts to an isomorphism 
\[
C^{*}(A\cup\{\lambda_{s}\mid s\in\langle\mathbf{w}(g)\rangle\})\cong A\otimes\Csr(\langle\mathbf{w}(g)\rangle).
\]
Thus the conditional expectation $\mathbb{E}_{\langle\mathbf{w}(g)\rangle}$ onto $A\rtimes_{\alpha|_{\langle\mathbf{w}(g)\rangle},\mathrm{r}}\langle\mathbf{w}(g)\rangle$ pulls back to a conditional expectation on $D$. The canonical unitaries in $D$ satisfy the covariance relations for $\alpha$; in particular, $L$ acts trivially on $A$, and finite Fourier sums are dense in $D$. All group-norm estimates used in the preceding proofs remain valid in $D$, since its canonical copy of $\Csr(\Gamma)$ is faithful. The initial and selective averaging arguments therefore converge to these expectations by the same finite-sum estimates and density. The polynomial computation, functional calculus, and final averaging argument are unchanged. Thus the proof of Theorem \ref{FreeProductCaseTheorem} applies to $D$ as well. \end{remark}


\subsection{The general case\label{subsec:ProofInvariantMainTheorem}}

The following statement is a strengthening of Lemma \ref{NormalWitnessLemma}.

\begin{lemma} \label{StrongNormalWitnessLemma} Let $\alpha\colon\Gamma\curvearrowright A$ be an action of an inner \Cs-selfless discrete group $\Gamma$ on a unital \Cs-algebra $A$, and assume that the ultrafilter $\mathcal{U}$ on $I$ and the family $(g_{i})_{i\in I}\subseteq\Gamma$ witnessing inner \Cs-selflessness satisfy $\lim_{i\rightarrow\mathcal{U}}\Vert\alpha_{g_{i}}(a)-a\Vert=0$ for every $a\in A$. Then, if $N\triangleleft\Gamma$ is a non-trivial normal subgroup, there exists a family $(n_{i})_{i\in I}\subseteq N$ witnessing inner \Cs-selflessness of $\Gamma$ with $\lim_{i\rightarrow\mathcal{U}}\Vert\alpha_{n_{i}}(a)-a\Vert=0$ for every $a\in A$. \end{lemma}

\begin{proof} Fix $n\in N\setminus\{e\}$. By following the argument in the proof of Lemma \ref{NormalWitnessLemma}, one obtains that the family $(n_{i})_{i\in I}$ with $n_{i}:=(g_{i}ng_{i}^{-1})(g_{i}^{2}n^{-1}g_{i}^{-2})\in N$ witnesses the inner \Cs-selflessness of $\Gamma$. Furthermore, for every $a\in A$,
\begin{eqnarray*}
\left\Vert \alpha_{n_{i}}(a)-a\right\Vert &=& \left\Vert \alpha_{g_{i}ng_{i}n^{-1}g_{i}^{-2}}(a)-a\right\Vert \\
&\leq& \left\Vert \alpha_{g_{i}^{-1}}(a)-a\right\Vert +\left\Vert \alpha_{g_{i}ng_{i}n^{-1}g_{i}^{-1}}(a)-a\right\Vert \\
&\leq& 2\left\Vert \alpha_{g_{i}^{-1}}(a)-a\right\Vert +\left\Vert \alpha_{g_{i}ng_{i}}\left(\alpha_{n^{-1}}(a)\right)-a\right\Vert \\
&\leq& 3\left\Vert \alpha_{g_{i}^{-1}}(a)-a\right\Vert +\left\Vert \alpha_{g_{i}}\left(\alpha_{n^{-1}}(a)\right)-\alpha_{n^{-1}}(a)\right\Vert 
\end{eqnarray*}
It follows that $\lim_{i\rightarrow\mathcal{U}}\Vert\alpha_{n_{i}}(a)-a\Vert=0$. \end{proof}

We have now assembled all the ingredients to prove Theorem \ref{InvariantMainTheorem}.

\begin{proof}[Proof of Theorem \ref{InvariantMainTheorem}] As in the statement of Theorem \ref{InvariantMainTheorem}, let $\alpha\colon\Gamma\curvearrowright A$ be an action of an inner $\Cs$-selfless discrete group $\Gamma$ on a unital $\Cs$-algebra $A$, and assume that the ultrafilter $\mathcal{U}$ on $I$ and the family $(g_{i})_{i\in I}\subseteq\Gamma$ witnessing inner $\Cs$-selflessness satisfy $\lim_{i\rightarrow\mathcal{U}}\Vert\alpha_{g_{i}}(a)-a\Vert=0$ for every $a\in A$. Write $\theta\colon\Csr(\Gamma\star\mathbb{Z})\hookrightarrow\Csr(\Gamma)^{\mathcal{U}}\subseteq(A\rtimes_{\alpha,\mathrm r}\Gamma)^{\mathcal{U}}$ for the associated map with $\theta|_{\Csr(\Gamma)}=\Delta$ and $\theta(\lambda_{h})=(\lambda_{g_{i}})_{i\rightarrow\mathcal{U}}$. Let $N\triangleleft\Gamma$ be a non-trivial normal subgroup and assume that $B\subseteq A\rtimes_{\alpha,\mathrm r}\Gamma$ is a $\Cs$-subalgebra that is invariant under conjugation by $N$. Then, by Lemma \ref{StrongNormalWitnessLemma}, we can assume that the family $(g_{i})_{i\in I}$ is already contained in $N$.

Denote the trivial extension of $\alpha$ to $\Gamma\star\mathbb{Z}$ by $\widetilde{\alpha}$, and let $D$ be the quotient from Remark \ref{UniversalFreeProductRemark}. The map $\theta$ and the diagonal embedding of $A\rtimes_{\alpha,\mathrm r}\Gamma$ agree on $\Csr(\Gamma)$, and the compatibility assumption gives $[\theta(\lambda_{h}),\Delta(a)]=0$ for every $a\in A$. By the universal property of $D$, the map $\theta$ therefore extends to a unital $\ast$-homomorphism $\widetilde{\theta}\colon D\to(A\rtimes_{\alpha,\mathrm r}\Gamma)^{\mathcal{U}}$ whose restriction to $A\rtimes_{\alpha,\mathrm r}\Gamma$ is the diagonal embedding and which satisfies $\widetilde{\theta}(\lambda_{h})=(\lambda_{g_{i}})_{i\rightarrow\mathcal{U}}$. Since $B$ is $N$-invariant, the preimage $C:=\widetilde{\theta}^{-1}(B^{\mathcal{U}})$ is a $\Cs$-subalgebra of $D$ that is invariant under conjugation by the subgroup $L:=\left\langle g^{-1}hg\mid g\in\Gamma\right\rangle \leq\Gamma\star\mathbb{Z}$. Indeed, for every element $c\in C$ there exists a bounded family $(b_{i})_{i\in I}\subseteq B$ with $\widetilde{\theta}(c)=(b_{i})_{i\rightarrow\mathcal{U}}$, so that 
\[
\widetilde{\theta}\left(\operatorname{Ad}(\lambda_{g^{-1}hg})(c)\right)=\widetilde{\theta}\left(\lambda_{g^{-1}hg}c\lambda_{g^{-1}hg}^{\ast}\right)=\left(\lambda_{g^{-1}g_{i}g}b_{i}\lambda_{g^{-1}g_{i}g}^{\ast}\right)_{i\rightarrow\mathcal{U}}\in B^{\mathcal{U}}
\]
for every $g\in\Gamma$. It follows that $\lambda_{g^{-1}hg}c\lambda_{g^{-1}hg}^{\ast}\in C$ for every $g\in\Gamma$. The same holds with $h^{-1}$ in place of $h$, and therefore $C$ is $L$-invariant.

By Remark \ref{UniversalFreeProductRemark}, and since its conditional expectation $\widetilde{\mathbb{E}}_{A}$ agrees with $\mathbb{E}_{A}$ on $A\rtimes_{\alpha,\mathrm r}\Gamma$, we obtain that $\mathbb{E}_{A}(x\lambda_{g}^{\ast})\lambda_{g}\in C$ for every $x\in C\cap(A\rtimes_{\alpha,\mathrm r}\Gamma)$ and $g\in\Gamma$. Moreover, $C\cap(A\rtimes_{\alpha,\mathrm r}\Gamma)=B$, since $\widetilde{\theta}$ is diagonal on $A\rtimes_{\alpha,\mathrm r}\Gamma$ and $\Delta(y)\in B^{\mathcal{U}}$ exactly when $y\in B$. This finishes the proof.
\end{proof}

\vspace{3mm}


\section{The cyclic Powers property}
\label{sec:cyclicpowers}
The main goal of this section is to provide examples of inner \Cs-selfless groups. With this aim in mind, we introduce a new group-theoretic notion, which we call the \emph{cyclic Powers property}, that implies inner \Cs-selflessness, Ozawa's property $\mathrm{P}_{\mathrm{PHP}}$, and de la Harpe's Powers property. Our new property admits a sufficient criterion in terms of bounded free independence, and seems to be an interesting notion on its own right.

\begin{definition}\label{def:cyclic-Powers}
    We say that a non-trivial discrete group $\Gamma$ has the \emph{cyclic Powers property} if for every finite subset $F\subseteq \Gamma$ and $n\in\mathbb N\setminus\{0\}$, there exist $A\subseteq \Gamma$ and $s\in \Gamma$ such that
    \begin{enumerate}[leftmargin=*,label=\textup{(\roman*)}]
        \item the sets $\{gA\}_{g\in F}$ are pairwise disjoint,
        \item $s^k(\Gamma\setminus A) \subseteq A$ for all $k=1,\dots,n$, and
        \item $A\cap F=\emptyset$.
    \end{enumerate} 
\end{definition}

For notational convenience we sometimes write $A^c=\Gamma\setminus A$ for a subset $A\subseteq\Gamma$.
Next, we compare the cyclic Powers property with Ozawa's property $\mathrm{P}_{\mathrm{PHP}}$ and the Powers property.

\begin{definition}[{see~\cite[Section 8]{Oza}}]
    A discrete group $\Gamma$ has \emph{property $\mathrm{P}_{\mathrm{PHP}}$} if, given a finite subset $F\subseteq\Gamma$ and $\varepsilon>0$, there exists an integer $N\geq1$ such that the following holds for every integer $n\geq N$.
    There are elements $t_1,\dots,t_n\in\Gamma$ and subsets $C_i\subseteq D_i\subseteq\Gamma$, for $1\leq i\leq n$, satisfying:
    \begin{enumerate}[label=\textup{(\roman*)},leftmargin=*]
        \item the sets in the family
        \[
            \{aC_i\mid a\in F,\ 1\leq i\leq n\}\cup\{bt_j^{-1}(\Gamma\setminus D_j)\mid b\in F,\ 1\leq j\leq n\}
        \]
        are pairwise disjoint,
        \item we have
        \[
            \sup_{x\in\Gamma}\left|\{i\in\{1,\dots,n\}\mid x\in D_i\}\cup\{j\in\{1,\dots,n\}\mid x\in t_j^{-1}(\Gamma\setminus C_j)\}\right|\leq\varepsilon\sqrt n.
        \]
    \end{enumerate}
\end{definition}

\begin{definition}[{see~\cite[Definition~9]{de2007simplicity}}]
    A discrete group $\Gamma$ is said to have the \emph{Powers property} if for every finite subset $F\subseteq \Gamma\setminus\{e\}$ and $n\in\mathbb N\setminus\{0\}$, there exist sets $C,D\subseteq\Gamma$ and elements $s_1 ,\ldots ,s_n \in \Gamma$ such that:
    \begin{enumerate}[leftmargin=*,label=\textup{(\roman*)}]
        \item $C\cap D=\emptyset$ and $\Gamma=C\cup D$,
        \item $gC\cap C=\emptyset$ for all $g\in F$,
        \item $s_i D\cap s_j D=\emptyset$ for all $i\neq j$.
    \end{enumerate} 
\end{definition}

\begin{lemma}\label{lem:cyclicPowers-to-Powers}
Let $\Gamma$ be a discrete group with the cyclic Powers property. Then the following statements hold:
\begin{enumerate}[leftmargin=*,label=\textup{(\roman*)}]
    \item $\Gamma$ has the property $\mathrm{P}_{\mathrm{PHP}}$,
    \item $\Gamma$ has the Powers property.
\end{enumerate}
\end{lemma}

\begin{proof}
    We start by proving (i).
    Fix a finite subset $F\subseteq\Gamma$ and $\varepsilon>0$. Without loss of generality, we may assume that $e\in F$ for if this is not already the case we can replace $F$ with $F\cup\{e\}$.
    Let $m$ be an arbitrary positive integer that is larger than $\varepsilon^{-2}$.
    Since $\Gamma$ is necessarily infinite, we may find elements $p_1,\dots,p_m$ and $q_1,\dots,q_m$ in $\Gamma$ such that $\{Fp_i\}\cup\{Fq_i\}$ are pairwise disjoint.
    Set $E=\bigcup_{i=1}^m (Fp_i\cup Fq_i)$, and apply the cyclic Powers property to the pair $(E,n=1)$, which yields a subset $A\subseteq \Gamma$ and an element $s\in\Gamma$ such that $\{gA\}_{g\in E}$ are pairwise disjoint, and $sA^c\subseteq A$.
    Now, set 
    \[
    C_i=D_i=p_iA,\quad t_i=p_is^{-1}q_i^{-1}
    \]
    for all $i=1,\dots,m$.
    Then one has that, for each $i$,
    \[
    t_i^{-1}(C_i^c)=t_i^{-1}(D_i^c)=q_isp_i^{-1}(p_iA^c)\subseteq q_iA.
    \]
    It follows that $aC_i\subseteq ap_iA$ and is disjoint from $bt_j^{-1}(D_j^c)\subseteq bq_jA$ for all $a,b\in F$ and $i,j=1,\dots,m$.
    Also, $D_i=p_iA$ and $t_j^{-1}(C_j^c)\subseteq q_jA$ are  disjoint for all $i,j=1,\dots,m$, which implies that 
    \[
    \sup_{x\in\Gamma}|\{i \mid x\in D_i\} \cup \{j \mid x\in t_j^{-1}(C_j^c)\}| \leq 1\leq\varepsilon m^{1/2}.
    \]
    Hence, $\Gamma$ has the property $\mathrm{P}_{\mathrm{PHP}}$.

    Let us now prove (ii). 
    Fix a finite subset $F\subseteq \Gamma\setminus\{e\}$ and $n\in\mathbb N\setminus\{0\}$.
    Define $\widetilde F=F\cup\{e\}$, and use the cyclic Powers property to find $s\in\Gamma$ and $A\subseteq\Gamma$ such that $\{gA\}_{g\in \widetilde F}$ are pairwise disjoint, and $s^kA^c\subseteq A$ and for all $k=1,\dots,n$. Set $C=A$, $D=A^c$, and $s_k=s^k$ for all $k=1,\dots,n$.
    We have that, by definition, $C\cap D=\emptyset$ and $C\cup D=\Gamma$.
    Also, $gC\cap C=\emptyset$ for all $g\in F$ because $e\in\widetilde F$.
    Moreover, for any pair of indices $i< j$ we have that $s_iD\cap s_jD=\emptyset$ because 
    \[
    s_iD\cap s_jD=s^iD\cap s^jD=s^i(D\cap s^{j-i}D)=\emptyset,
    \]
    and the case $i>j$ follows from the same argument.
\end{proof}

The following theorem, which is interesting in its own right, has two main applications.
First, it implies that groups with the cyclic Powers property are inner \Cs-selfless; see Corollary~\ref{cor:cP-to-innerselfless}. 
Second, groups with the cyclic Powers property give rise to completely selfless reduced crossed products, which are often tracial, and thus non-nuclear; see Corollary~\ref{cor:uniquely-ergodic-crossed-selfless} and the following remark.

\begin{theorem}\label{thm:cP-crossed-product}
    Let $\Gamma$ be a non-trivial discrete group with the cyclic Powers property.
    Let $\alpha\colon \Gamma\curvearrowright A$ be an action on a unital
    \Cs-algebra, $\rho$ be a state on $A$ such that $\rho \circ \mathbb E_A$ is GNS-faithful, and suppose that there exists a net of unitaries $(v_i)_{i \in I} \subseteq \Csr(\Gamma)$ such that
    \[
        \|\mathbb E_A(v_i^*xv_i)-\rho(\mathbb E_A(x))1_A\|\longrightarrow0
    \]
    for all $x\in A\rtimes_{\alpha,\mathrm r}\Gamma$.
    Then the canonical inclusion $\Csr(\Gamma)\subseteq(A\rtimes_{\alpha,\mathrm r}\Gamma,\rho\circ\mathbb E_A)$ is completely selfless.
\end{theorem}

\begin{proof}
    Let $(\pi_\rho,\mathcal H_\rho,\xi_\rho)$ be the GNS-representation of $(A,\rho)$. Consider the faithful regular representation of $A\rtimes_{\alpha,\mathrm r}\Gamma$ on $\mathcal H :=\mathcal H_\rho\otimes\ell^2(\Gamma)$ induced by
    \[
        \pi(a)(\eta\otimes\delta_h)=\pi_\rho(\alpha_{h^{-1}}(a))\eta\otimes\delta_h, \qquad \lambda_g(\eta\otimes\delta_h)=\eta\otimes\delta_{gh}
    \]
    for $a\in A$, $g,h\in \Gamma$, and $\eta\in\mathcal H_\rho$. Then $\rho\circ\mathbb E_A$ corresponds to the vector state associated with $\Omega=\xi_\rho\otimes\delta_e$, and $\Omega$ is cyclic. 
    Hence, by the uniqueness of the GNS-construction we have that this representation is unitarily equivalent to the GNS-representation of $\rho \circ \mathbb E_A$. Since the latter is assumed to be faithful, we have that $A \rtimes_{\alpha,\mathrm r} \Gamma$ is faithfully represented on $\mathcal H_\rho\otimes\ell^2(\Gamma)$. In what follows, we identify $A\rtimes_{\alpha,\mathrm r}\Gamma$ with its image in $\mathcal B(\mathcal H)$ under $(\pi,\lambda)$.
    
    For any subset $S\subseteq \Gamma$, let $P_S$ be the projection onto $\mathcal H_\rho\otimes\ell^2(S)\subseteq\mathcal H$. By definition, $A$ commutes with each $P_S$, and whenever $gS\cap S=\emptyset$, we have
    \begin{equation}
        P_S(a\lambda_g)P_S=0 \label{eq:annihilation}
    \end{equation}
    for all $a\in A$.
    Now, fix a finite subset $\mathcal F\subseteq A\rtimes_{\alpha,\mathrm r}\Gamma$ containing $1_A$, and an integer $k\geq1$. By assumption, there exists a unitary $v_{\mathcal F,k}\in \Csr(\Gamma)$ such that
    \[
        \max_{x\in\mathcal F}\|\mathbb E_A(v_{\mathcal F,k}^*xv_{\mathcal F,k})-\rho(\mathbb E_A(x))1_A\|<\frac1k.
    \]
    For every $x\in\mathcal F$, we find a finite subset $F_x\subseteq \Gamma$ and $(a_{x,g})_{g\in F_x}\subseteq A$ such that $y_x:=\sum_{g\in F_x}a_{x,g}\lambda_g$ satisfies 
    \[
       \|v_{\mathcal F,k}^*xv_{\mathcal F,k}-y_x\|<\frac1k.
    \]
    Next, put
    \[
        F_{\mathcal F,k}:= \bigcup_{x\in\mathcal F}(F_x\setminus \{e\}).
    \]
    Then we may find a finite subset $e\in K_{\mathcal F,k}\subseteq \Gamma$ such that
    \[
        \max_{x\in\mathcal F}\|(1-P_{K_{\mathcal F,k}})v_{\mathcal F,k}^*x\Omega\|<\frac1k.
    \]
    Note that such a subset $K_{\mathcal F,k}$ exists because the subspaces of the form $\mathcal H_\rho\otimes\ell^2(K)$, for finite $K\subseteq \Gamma$, have dense union in $\mathcal H$.
    Applying the cyclic Powers property of $\Gamma$ to the finite set $F_{\mathcal F,k}\cup K_{\mathcal F,k}$, and the positive integer $k$, we obtain $S_{\mathcal F,k}\subseteq \Gamma$ and $s_{\mathcal F,k}\in\Gamma$ such that, in particular, 
    \[
        gS_{\mathcal F,k}\cap S_{\mathcal F,k}=\emptyset,\quad S_{\mathcal F,k}\cap K_{\mathcal F,k}=\emptyset,\quad s_{\mathcal F,k}^jS_{\mathcal F,k}^c\subseteq S_{\mathcal F,k}
    \]
    for all $g\in F_{\mathcal F,k}$ and $j=1,\dots,k$.
    To each pair $(\mathcal F,k)$, we associate $(S_{\mathcal F,k},s_{\mathcal F,k})$ as above. To simplify the notation, in the following we stop retaining the label $(\mathcal F,k)$ as follows: $F:=F_{\mathcal F,k}$, $K:=K_{\mathcal F,k}$, $S:=S_{\mathcal F,k}$, $s:=s_{\mathcal F,k}$ and $v:=v_{\mathcal F,k}$. Let us define: 
    \[
        p_{\mathcal F,k}=vP_Sv^*\in \mathcal{B}(\mathcal H), \qquad z_{\mathcal F,k}=v\lambda_sv^*\in \Csr(\Gamma).
    \]
    Since we have that $gS\cap S=\emptyset$ for all $g\in F$, we may argue with \eqref{eq:annihilation} that, for any $x\in\mathcal F$,
    \begin{equation}
        P_Sy_xP_S=\mathbb E_A(y_x)P_S. \label{eq:expected}
    \end{equation}
    Since $\mathbb E_A$ is contractive, we also have
    \begin{align*}
        \max_{x\in\mathcal F}\|\mathbb E_A(y_x)-\rho(\mathbb E_A(x))1_A\|\leq & \max_{x\in\mathcal F}\|y_x-v^*xv\|\\
        &+\max_{x\in\mathcal F}\|\mathbb E_A(v^*xv)-\rho(\mathbb E_A(x))1_A\|<\frac2k.
    \end{align*}
    Consequently, combining the previous inequality with \eqref{eq:expected}, we have that  
    \begin{align*}
        \max_{x\in\mathcal F}\|p_{\mathcal F,k}xp_{\mathcal F,k}-\rho(\mathbb E_A(x))p_{\mathcal F,k}\|&=\max_{x\in\mathcal F}\|P_Sv^*xvP_S-\rho(\mathbb E_A(x))P_S\|\\
        &< \frac1k +\max_{x\in\mathcal F}\|P_Sy_xP_S-\rho(\mathbb E_A(x))P_S\|\\
        &=\frac1k +\max_{x\in\mathcal F}\|\mathbb E_A(y_x)P_S-\rho(\mathbb E_A(x))P_S\|\\
        &<\frac3k .
    \end{align*}
    Furthermore, $S\cap K=\emptyset$ gives $P_SP_K=0$, and hence
    \[
        \max_{x\in\mathcal F}\|p_{\mathcal F,k}x\Omega\|=\max_{x\in\mathcal F}\|P_Sv^*x\Omega\|\leq\max_{x\in\mathcal F}\|(1-P_K)v^*x\Omega\|<\frac1k.
    \]
    Note then that the inclusions $s^jS^c\subseteq S$ for $j=1,\dots,k$ yield
    \[
        (1-P_S)\lambda_s^j(1-P_S)=0 
    \]
    for all $j=1,\dots,k$.
    Conjugating by $v$, and taking adjoints, we obtain
    \[
        (1-p_{\mathcal F,k})z_{\mathcal F,k}^j(1-p_{\mathcal F,k})=0 
    \]
    for all $j$ with $1\leq|j|\leq k$.
    Since $e\in K$ and $S\cap K=\emptyset$, the complement of $S$ is non-empty, and it follows that $s^j\neq e$ for $1\leq j\leq k$. Using that the restriction of $\rho \circ \mathbb E_A$ to $\Csr(\Gamma)$ is the canonical tracial state $\tau_\Gamma$, we have 
    \[
        \rho(\mathbb E_A(z_{\mathcal F,k}^j))=\tau_\Gamma(v\lambda_{s^j}v^*)=\tau_\Gamma(\lambda_{s^j})=0 
    \]
    for all $1\leq|j|\leq k$.
    Notice also that $S\neq\emptyset$, by the partition condition for $j=1$, so $p_{\mathcal F,k}\neq0$.

    Consider the index set $I=\{(\mathcal F,k) \mid 1_A\in \mathcal F\subseteq A\rtimes_{\alpha,\mathrm r}\Gamma \text{ finite},\, k\geq 1\}$ ordered by inclusion on the first component, and with the standard order on the second entry. Then choose an ultrafilter $\mathcal U$ over $I$ containing all final segments, and put $z:=(z_{\mathcal F,k})_{(\mathcal F,k)\to\mathcal U}\in \Csr(\Gamma)^{\mathcal U}$.
    The preceding identities show that $(\rho \circ \mathbb E_A)^{\mathcal U}(z^j)=0$ 
    for all $j\in\mathbb Z\setminus\{0\}$.
    {Let $\gamma$ denote the restriction of $(\rho\circ\mathbb E_A)^{\mathcal U}$ to $\Cs(z)$. Since $z$ is a Haar unitary, $\gamma$ is faithful.} Let $\theta\colon \Cs(z)\rightarrow \Csr(\Gamma)^{\mathcal U}$ denote the inclusion map.
    Next, fix a unit vector $\eta\in\ell^2$, and put
    \[
        p:=(p_{\mathcal F,k}\otimes1_{\mathcal B(\ell^2)})_{(\mathcal F,k)\to\mathcal U}\in\mathcal B((\mathcal H\otimes\ell^2)^{\mathcal U}).
    \]
    Let $\Delta\colon A\rtimes_{\alpha,\mathrm r}\Gamma\to (A\rtimes_{\alpha,\mathrm r}\Gamma)^{\mathcal U}$ denote the diagonal embedding, and identify $(A\rtimes_{\alpha,\mathrm r}\Gamma)^{\mathcal U}$ with its image inside $\mathcal B((\mathcal H\otimes\ell^2)^{\mathcal U})$ via
    \[
        (x_{\mathcal F,k})_{(\mathcal F,k)\to\mathcal U} \mapsto (x_{\mathcal F,k}\otimes1_{\mathcal B(\ell^2)})_{(\mathcal F,k)\to\mathcal U}
    \]
    for all elements $(x_{\mathcal F,k})_{(\mathcal F,k)\to\mathcal U} \in (A\rtimes_{\alpha,\mathrm r}\Gamma)^{\mathcal U}$.
    For every $d\in\mathcal B(\ell^2)$, each projection $p_{\mathcal F,k}\otimes1_{\mathcal B(\ell^2)}$ commutes with $1_{\mathcal B(\mathcal H)}\otimes d$. Hence $p$ commutes with the diagonal copy of $\mathcal B(\ell^2)$ in $\mathcal B((\mathcal H\otimes\ell^2)^{\mathcal U})$.
    Since tensoring an operator with $1_{\mathcal B(\ell^2)}$ preserves its norm, the compression estimates from before yield
    \[
        p\,\Delta(x)\,p=\rho(\mathbb E_A(x))p, \quad p^\perp\,\theta(y)\,p^\perp={\gamma(y)p^\perp},\quad p\,\Delta(x)(\Omega\otimes\eta)^{\mathcal U}=0,
    \]
    for all $x\in A\rtimes_{\alpha,\mathrm r}\Gamma$, and $y\in\Cs(z)$, where $(\Omega\otimes\eta)^{\mathcal U}$ denotes the vector in $(\mathcal H\otimes\ell^2)^{\mathcal U}$ induced by the family constantly equal to $\Omega\otimes\eta$. {The second identity follows first for powers of $z$ and then by norm density.}
    Thus, we may apply \cite[Theorem C]{KP26}, and conclude that the inclusion $\Csr(\Gamma)\subseteq(A\rtimes_{\alpha,\mathrm r}\Gamma,\rho\circ \mathbb E_A)$ is completely selfless.
\end{proof}

The following is the first main application of the preceding theorem: groups with the cyclic Powers property are inner \Cs-selfless.

\begin{corollary}\label{cor:cP-to-innerselfless}
    Let $\Gamma$ be a non-trivial discrete group with the cyclic Powers property.
    Then $\Gamma$ is completely \Cs-selfless and inner \Cs-selfless.
\end{corollary}

\begin{proof}
    {Consider the trivial action $\Gamma\curvearrowright\mathbb C$. Then the GNS-faithfulness condition and the unitaries assumption of Theorem~\ref{thm:cP-crossed-product} are automatically satisfied with $v_i=1$. Hence $\Csr(\Gamma)$ is completely selfless. Choosing $v_{\mathcal F,k}=1$ throughout the construction gives $z_{\mathcal F,k}=\lambda_{s_{\mathcal F,k}}$, so the additional Haar unitary is represented by group unitaries. Thus $\Gamma$ is also inner \Cs-selfless.}
\end{proof}

To show the second main application of Theorem~\ref{thm:cP-crossed-product}, we need a few intermediate steps.

\begin{lemma}\label{lem:uniquely-ergodic-trick}
    Let $\alpha\colon \Gamma\curvearrowright A$ be an action of a discrete group
    on a unital \Cs-algebra, and let $t\in \Gamma$ be an element of infinite order.
    Assume that $\rho$ is a state on $A$ such that
    \[
        \lim_{N\to\infty}\left\|\frac1N\sum_{k=0}^{N-1}\alpha_{t}^{-k}(a)
        -\rho(a)1_A\right\|=0
    \]
    for all $a\in A$.
    Then there exists a sequence of unitaries $(v_n)_{n \in \mathbb{N}}\subseteq \Cs(\lambda_t)\subseteq \Csr(\Gamma)$ such that
    \[
        \lim_{n\to\infty}\|\mathbb E_A(v_n^*xv_n)-\rho(\mathbb E_A(x))1_A\|=0
    \]
    for all $x\in A\rtimes_{\alpha,\mathrm r}\Gamma$.
\end{lemma}

\begin{proof}
    For each $n\geq1$, put $N:=2n$, and define a continuous function
    $f_n\colon\mathbb T\to\mathbb T$ by $ f_n(e^{2\pi is}):=e^{\pi iNs^2}$ for $s\in[0,1]$.
    Note that the values at the endpoints agree because $N$ is even.
    Using functional calculus, we may define $v_n=f_n(\lambda_t)$, and denote the $k$-th Fourier coefficients of $f_n$ by
    \[
        c_{N,k}:=\int_0^1e^{\pi iNs^2-2\pi iks}\,ds
    \]
    for $k\in\mathbb Z$.
    We first claim that
    \[
        \lim_{N\to\infty}\sum_{k\in\mathbb Z}|c_{N,k}-d_{N,k}|^2=0,
        \quad \text{where}\quad 
        d_{N,k}:=\frac{e^{i\pi/4}}{\sqrt N}
        e^{-\pi ik^2/N}\,1_{\{0,\ldots,N-1\}}(k).
    \]
    Let $c_N:=(c_{N,k})_{k\in\mathbb Z}$ and $d_N:=(d_{N,k})_{k\in\mathbb Z}$, and note that Parseval's identity gives that the norm in $\ell^2(\mathbb Z)$ of $c_N$ is $\|c_N\|_2=1$, while an easy calculation shows that $\|d_N\|_2=1$ as well.
    Now, note that
    \[
        c_{N,k}=e^{-\pi ik^2/N}\int_0^1e^{\pi iN(s-k/N)^2}\,ds=e^{-\pi ik^2/N}
        \int_{-k/N}^{1-k/N}e^{\pi iNs^2}\,ds,
    \]
    and that, as  
    $\int_{\mathbb R}e^{\pi iNs^2}\,ds=e^{i\pi/4}/\sqrt N$, we have $\left| \int_b^\infty e^{\pi iNs^2}\,ds\right| \leq\frac1{\pi Nb}$ for any $b>0$.
    Now, fix $\varepsilon\in(0,\frac12)$. For an index satisfying $\varepsilon N\leq k\leq(1-\varepsilon)N$, both \(k/N\) and \(1-k/N\) are at least \(\varepsilon\). Hence, 
    \[
        \begin{aligned}
        |c_{N,k}-d_{N,k}|
        &\leq
        \left|\int_{-\infty}^{-k/N}e^{\pi iNs^2}\,ds\right|
        +\left|\int_{1-k/N}^{\infty}e^{\pi iNs^2}\,ds\right|\\
        &\leq \frac{1}{\pi N(k/N)}
        +\frac{1}{\pi N(1-k/N)}
        \leq\frac{2}{\pi\varepsilon N}.
        \end{aligned}
    \]
    Split the support of \(d_N\) into
    \[
        I_N:=\{k\in\{0,\ldots,N-1\}\mid \varepsilon N\leq k\leq(1-\varepsilon)N\},
        \qquad
        J_N:=\{0,\ldots,N-1\}\setminus I_N.
    \]
    Their sizes satisfy
    \[
        \lim_{N\to\infty}\frac{|I_N|}{N}=1-2\varepsilon,
        \qquad
        \lim_{N\to\infty}\frac{|J_N|}{N}=2\varepsilon.
    \]
    By writing \(c_{N,k}=d_{N,k}+(c_{N,k}-d_{N,k})\) on \(I_N\), and using that \(|d_{N,k}|=N^{-1/2}\), and the upper bound from before, 
    \[
        \begin{aligned}
        \operatorname{Re}\bigg(\sum_{k\in I_N}c_{N,k}\overline{d_{N,k}}\bigg)
        &\geq\sum_{k\in I_N}|d_{N,k}|^2
        -\sum_{k\in I_N}|c_{N,k}-d_{N,k}|\,|d_{N,k}|\\
        &\geq\frac{|I_N|}{N}
        -\frac{2|I_N|}{\pi\varepsilon N^{3/2}} \geq\frac{|I_N|}{N}-\frac{2}{\pi\varepsilon\sqrt N}.
        \end{aligned}
    \]
    Thus, as $N\to\infty$, the indices in $I_N$ contribute at least \(1-2\varepsilon\).  
    For \(J_N\), we have by the Cauchy--Schwarz inequality that
    \[
        \begin{aligned}
        \left|\sum_{k\in J_N}c_{N,k}\overline{d_{N,k}}\right|
        &\leq
        \left(\sum_{k\in J_N}|c_{N,k}|^2\right)^{1/2}
        \left(\sum_{k\in J_N}|d_{N,k}|^2\right)^{1/2}\\
        &\leq \|c_N\|_2\sqrt{\frac{|J_N|}{N}}
        =\sqrt{\frac{|J_N|}{N}}.
        \end{aligned}
    \]
    The contribution of the real part is therefore bounded below by \(-\sqrt{|J_N|/N}\). Indices outside \(\{0,\ldots,N-1\}\) contribute zero because \(d_{N,k}=0\). 
    Combining these estimates gives
    \[
        \operatorname{Re}\bigg(\sum_{k\in\mathbb Z}c_{N,k}\overline{d_{N,k}}\bigg)
        \geq
        \frac{|I_N|}{N}
        -\frac{2}{\pi\varepsilon\sqrt N}
        -\sqrt{\frac{|J_N|}{N}}.
    \]
    Taking \(N\to\infty\), with \(\varepsilon\) fixed, yields 
    \[
        \liminf_{N\to\infty}\operatorname{Re}\bigg(\sum_{k\in\mathbb Z}c_{N,k}\overline{d_{N,k}}\bigg)
        \geq1-2\varepsilon-\sqrt{2\varepsilon}.
    \]
    Since $\varepsilon\in(0,\frac12)$ was arbitrary, and using Cauchy--Schwarz again, we obtain
    \[
        1\leq \liminf_{N\to\infty} \, \operatorname{Re}\bigg(\sum_{k\in\mathbb Z}c_{N,k}\overline{d_{N,k}}\bigg)
        \leq\liminf_{N\to\infty} \, \|c_N\|_2\|d_N\|_2=1,
    \]
    so that the real part tends to \(1\). Consequently,
    \[
        \sum_{k\in\mathbb Z}|c_{N,k}-d_{N,k}|^2
        =\|c_N\|_2^2+\|d_N\|_2^2
        -2\operatorname{Re}\bigg(\sum_{k\in\mathbb Z}c_{N,k}\overline{d_{N,k}}\bigg)
        \xrightarrow{N\to\infty} 0,
    \]
    as claimed.

    Next, note that for any $a\in A$, ${n\geq1}$, and $l\in\mathbb Z$, 
    \[
        \mathbb E_A(v_n^*a\lambda_t^lv_n)
        =\sum_{k\in\mathbb Z}
        \overline{c_{N,k}}c_{N,k-l}\alpha_{t^{-k}}(a).
    \]
    The identity follows,
    for instance, by first observing that the series on the right is absolutely convergent in norm, and then using the Fej\'er polynomials of $f_n$ and finally passing to the limit.  Next, by the Cauchy--Schwarz inequality,
    \[
        \sum_k\left|
        \overline{c_{N,k}}c_{N,k-l}
        -\overline{d_{N,k}}d_{N,k-l}\right|
        \leq2\|c_N-d_N\|_2\xrightarrow{N\to\infty}0.
    \]
    Therefore,
    \[
        \bigg\|\mathbb E_A(v_n^*a\lambda _t^lv_n)
        -\frac{e^{-\pi il^2/N}}N
        \sum_{k=0}^{N-1}e^{2\pi ikl/N}\alpha_{t^{-k}}(a)
        \bigg\|\xrightarrow{N\to\infty}0.
    \]
    Here the replacement of the overlap of the two supports
    by $\{0,\ldots,N-1\}$ changes the sum by at most
    $|l|\|a\|/N$.

    Write $S_m=\sum_{k=0}^{m-1}(\alpha_{t^{-k}}(a)-\rho(a)1_A)$, and use the assumption to infer that 
    \[
        \left\|\frac1N\sum_{k=0}^{N-1}e^{2\pi ikl/N}(\alpha_{t^{-k}}(a)-\rho(a)1_A)\right\|
        \leq\frac{\|S_N\|}{N}
        +\frac{2\pi|l|}{N^2}\sum_{k=1}^{N-1}\|S_k\|
        \xrightarrow{N\to\infty}0.
    \]
    Since we have that 
    \[
    \frac1N \sum_{k=0}^{N-1}e^{\frac{2\pi ikl}{N}}=\frac1{2n} \sum_{k=0}^{2n-1}e^{\frac{2\pi ikl}{2n}}=
    \begin{cases}
        1 & \text{if }l=0 \\
        \frac{1}{2n}\frac{1-e^{2\pi i l}}{1-e^{\frac{2\pi i l}{2n}}}=0 & \text{if }0< |l|<N=2n
    \end{cases},
    \]
the previous estimates yield
    \begin{align*}
        \|\mathbb E_A(v_n^*a\lambda _t^lv_n) -\rho(a)\delta_{l,0}1_A\|  &\leq   \bigg\|\mathbb E_A(v_n^*a\lambda_t^lv_n)
        -\rho(a)e^{-\frac{\pi il^2}{2n}}
        \left(\frac1{2n} \sum_{k=0}^{2n-1}e^{\frac{2\pi ikl}{2n}}\right)1_A\bigg\| \\
        & \quad +\bigg\|\rho(a)e^{-\frac{\pi il^2}{2n}}
        \left(\frac1{2n} \sum_{k=0}^{2n-1}e^{\frac{2\pi ikl}{2n}}\right)1_A - \rho(a)\delta_{l,0}1_A\bigg\| \\
        &\qquad \xrightarrow{n\to\infty} 0.
    \end{align*}
    If $g\notin\langle t\rangle$, the identity
    $\mathbb E_A(v_n^*a\lambda_gv_n)=0$ holds for every $n$ because polynomial approximations of \(v_n\) produce group terms \(t^{-k}gt^j\), none of which can equal \(e\).
    The conclusion now follows from a density argument.
\end{proof}

\begin{definition}
    An automorphism $\alpha\in \mathrm{Aut}(A)$ of a unital \Cs-algebra is said to be \emph{uniquely ergodic} if there exists exactly one state $\rho$ on $A$ such that $\rho\circ \alpha=\rho$.
\end{definition}

\begin{remark}\label{rem:unique-ergodicity}
    Let $\alpha\in\mathrm{Aut}(A)$ be an automorphism of a unital \Cs-algebra.
    By applying \cite[Theorem 3.2]{MT07} (for the Abelian setting see \cite[Theorem 5.1]{Oxtoby52}, and for a more general version see \cite[Theorem 3.2]{AD09}), and by observing that both $\alpha$ and $\alpha^{-1}$ have the same invariant states, we may conclude that $\alpha$ is uniquely ergodic if and only if there exists a state $\rho$ such that  
    \[
        \lim_{N\to\infty}\bigg\|\frac1N \sum_{k=0}^{N-1}\alpha^{-k}(a)-\rho(a)1_A\bigg\|=0\qquad \text{for all }a\in A.
    \]
\end{remark}

We are now ready to show that, as an application of Theorem~\ref{thm:cP-crossed-product}, groups with the cyclic Powers property give rise to completely selfless crossed product inclusions under an ergodicity assumption.

\begin{corollary}\label{cor:uniquely-ergodic-crossed-selfless}
    Let $\Gamma$ be a non-trivial group with the cyclic Powers property, and let $\alpha\colon \Gamma\curvearrowright A$ be an action
    on a unital \Cs-algebra. Suppose that there exists an 
    element $t\in \Gamma$ of infinite-order such that $\alpha_t$ is uniquely ergodic with unique invariant state $\rho$, and that $\rho \circ \mathbb E_A$ is GNS-faithful. Then the inclusion
    \[
    \Csr(\Gamma)\subseteq
       \bigl(A\rtimes_{\alpha,\mathrm r}\Gamma,\rho\circ \mathbb E_A\bigr)
    \]
    is completely selfless.
\end{corollary}

\begin{proof}
    As recalled in Remark~\ref{rem:unique-ergodicity}, unique ergodicity of $\alpha_t$ is equivalent to
    \[
        \lim_{N\to\infty}\left\|\frac1N\sum_{k=0}^{N-1}\alpha_{t}^{-k}(a)
            -\rho(a)1_A\right\|=0
    \]
    for all $a\in A$. Lemma~\ref{lem:uniquely-ergodic-trick} now gives unitaries
    $v_n\in\Cs(\lambda_t)\subseteq \Csr(\Gamma)$ such that
    \[
        \lim_{n\to\infty}\|\mathbb E_A(v_n^*xv_n)-\rho(\mathbb E_A(x))1_A\|=0
    \]
    for all $x\in A\rtimes_{\alpha,\mathrm r}\Gamma$.
    To conclude, apply Theorem~\ref{thm:cP-crossed-product}.
\end{proof}

\begin{remark}
    We remark that the crossed products covered by Corollary~\ref{cor:uniquely-ergodic-crossed-selfless} are necessarily tracial when $\rho$ is a $\Gamma$-invariant trace, and thus they can only be completely selfless with respect to the unique trace induced by $\rho$. In particular, traciality and complete selflessness imply that these crossed products are not nuclear.
    
    Furthermore, note that for an Abelian unital \Cs-algebra $A=C(X)$, the assumptions of Corollary~\ref{cor:uniquely-ergodic-crossed-selfless} can be conveniently rephrased as follows. The homeomorphism $T\in\mathrm{Homeo}(X)$ corresponding to $\alpha_t$ has exactly one invariant probability measure $\mu$, and 
    \[
        \overline{\bigcup_{g\in\Gamma}g\cdot \big(\mathrm{supp} (\mu)\big)} = X.
    \]
    In particular, when the only probability measure $\mu\in\mathrm{Prob}(X)$ invariant under $T$ is of the form $\mu=\delta_x$ for $x\in X$, which is necessarily the unique fixed point of $T$, then the remaining assumption is that the orbit of $x$ is dense in $X$.
\end{remark}


\subsection{Examples of groups with the cyclic Powers property}

We now proceed to provide (classes of) groups that satisfy the cyclic Powers property. Thanks to Corollary~\ref{cor:cP-to-innerselfless}, these groups are also inner \Cs-selfless.
We start with groups that admit a topologically free \emph{extreme boundary action}.

\begin{definition}
    An action $\alpha\colon\Gamma\curvearrowright X$ of a discrete group $\Gamma$ on a compact Hausdorff space $X$ is said to be an \emph{extreme boundary action} if, for every pair of non-empty open subsets $U,V\subseteq X$, there exists $g\in\Gamma$ such that $\alpha_g(X\setminus U)\subseteq V$; see \cite[p.~328]{Gla74}.
    
    Moreover, the action is \emph{topologically free} if the points with trivial stabilizer form a dense subset of $X$.
\end{definition}

\begin{proposition}\label{prop:extremelyproximal-cP}
    Let $\Gamma$ be a discrete group admitting a topologically free extreme boundary action on a compact Hausdorff space $X$ with $|X|>2$.
    Then $\Gamma$ has the cyclic Powers property.
\end{proposition}
\begin{proof}
    The space $X$ has no isolated points, otherwise extreme proximality would map a proper two-point closed subset of $X$ into an open singleton, which is impossible.
    In particular, $X$ is infinite.

    Fix a finite set $F\subseteq\Gamma$ and an integer $n\geq1$.
    Enlarging $F$ if necessary, we may assume that $e\in F$.
    Choose any point $x_0\in X$.
    Since $X$ is infinite, $X\setminus (Fx_0)$ is a non-empty open subset of $X$.
    For distinct $g,h\in F$, the set $\{x\in X\mid gx=hx\}$ is closed with empty interior by topological freeness.
    Since there are only finitely many such pairs, we may choose $x_1\in X\setminus (Fx_0)$ such that the points $gx_1$, for $g\in F$, are pairwise distinct.
    Since $F$ is finite and $X$ Hausdorff, there exists an open neighbourhood $W$ of $x_1$ such that
    \[
        W\cap (Fx_0)=\emptyset,\qquad (gW)\cap (hW)=\emptyset
    \]
    for all $g,h\in F$ with $g\neq h$.
    Since $X$ has no isolated points, we may choose non-empty disjoint open subsets $U_-,U_+\subseteq W$.
    By extremality, there exists $s\in\Gamma$ such that
    \[
        s(X\setminus U_-)\subseteq U_+, \qquad s^{-1}(X\setminus U_+)\subseteq U_-.
    \]
    Since $U_+$ and $U_-$ are disjoint, these inclusions imply
    \[
        sU_+\subseteq U_+,\qquad s^{-1}U_-\subseteq U_-.
    \]
    Consequently, for every integer $k\geq1$,
    \[
        s^k(X\setminus(U_-\cup U_+))\subseteq U_+,\qquad s^{-k}(X\setminus(U_-\cup U_+))\subseteq U_-.
    \]
    Define $ A:=\{g\in\Gamma\mid gx_0\in U_-\cup U_+\}$. The sets $\{gA\}_{g\in F}$ are pairwise disjoint, since the sets $\{g(U_-\cup U_+)\}_{g\in F}$ are pairwise disjoint.
    Moreover, $A\cap F=\emptyset$, since $(U_-\cup U_+)\cap (Fx_0)=\emptyset$.
    If $h\in A^c$, then $hx_0\in X\setminus(U_-\cup U_+)$, so the preceding inclusions give
    \[
        s^khx_0\in U_-\cup U_+
    \]
    for all $k\in\mathbb Z\setminus\{0\}$, and thus also
    \[
        s^kA^c\subseteq A
    \]
    for all $k\in\mathbb Z\setminus\{0\}$. In particular, this holds for $1\leq k\leq n$, and proves the cyclic Powers property.
\end{proof}


The following is a refinement of free independence introduced in \cite[Definition 2.2(1)]{Vigdorovich25}.

\begin{definition}
    Let $\Gamma$ be a discrete group, $Q\subseteq\Gamma\setminus\{e\}$ a finite subset, and $n\in\mathbb N\setminus\{0\}$. An element $s\in\Gamma$ is said to be \emph{$n$-freely independent from $Q$} if 
    \[
    q_0s^{k_1}q_1s^{k_2}q_2\dots s^{k_\ell}q_{\ell}\neq e
    \]
    whenever 
    \[
    q_0t^{k_1}q_1t^{k_2}q_2\dots t^{k_\ell}q_{\ell}
    \]
    is a non-trivial reduced word in $\Gamma\star\langle t\rangle=\Gamma\star\mathbb{Z}$, where $\ell\in\mathbb{N}$, $q_{0},q_{\ell}\in Q\cup\{e\}$, $q_{i}\in Q$ for $1\leq i<\ell$, and $0<|k_{i}|\leq n$.

    Following \cite[Definition 2.2(1)]{Vigdorovich25}, if $s$ is $n$-freely independent from $Q$ for all $n>0$, then we say that $s$ is \emph{freely independent from $Q$}.
\end{definition}

The following lemma allows us to verify the cyclic Powers property using $n$-free independence.

\begin{lemma}\label{lem:free-independence-CPproperty}
    Let $\Gamma$ be a non-trivial discrete group, and assume that for every finite subset $Q\subseteq\Gamma\setminus\{e\}$ and $n\in\mathbb N\setminus\{0\}$, there exists $s\in\Gamma$ that is $n$-freely independent from $Q$. Then, $\Gamma$ has the cyclic Powers property.
\end{lemma}

\begin{proof}
    The group $\Gamma$ is infinite. Indeed, if it were finite, then for $q\neq e$ the identity $qs^{|\Gamma|!}q^{-1}=e$ would hold for every $s\in\Gamma$, contradicting $|\Gamma|!$-free independence from $\{q,q^{-1}\}$.
    Pick a finite subset $F\subseteq\Gamma$, and $n\in\mathbb N\setminus\{0\}$. We may assume that $e\in F$ and $|F|\geq2$ without loss of generality.
    For $Q:=F^{-1}F\setminus\{e\}$, we may find an element $s\in\Gamma$ that is $2n$-freely independent from $Q$. Set
    \[
    P:=\{s^j \mid 0<|j|\leq n\}
    \]
    and let $D_0$ denote the set obtained by taking (finite) alternating products of elements in $QP$, i.e.,
    \[
    D_0:=\{e\} \cup \bigcup_{\ell\geq1} (QP)^\ell
    \]
    and put $A_0:=PD_0$. Fix $q\in Q$. Every alternating expression beginning and ending with non-zero powers of $s$, with coefficients in $Q$ and exponents of absolute value at most $2n$, is non-trivial: conjugating it by $q$ gives a word covered by $2n$-free independence. Thus $e\notin A_0$.

    Suppose that a non-empty product $d=q_1s^{j_1}\cdots q_ms^{j_m}$ in $D_0$ equals $e$. Then $s^{-j_m}ds^{j_m}=e$ is an alternating expression of the preceding form, with first exponent $-j_m$ and last exponent $2j_m$, a contradiction. Hence no non-empty product in $\bigcup_{m\geq1}(QP)^m$ equals $e$.
    Moreover, since $Q^{-1}=Q$ and $P^{-1}=P$, we have that 
    \[
    D_0^{-1}A_0=A^{-1}_0D_0=A_0.
    \]
    Note also that 
    \[
    QA_0 = D_0\setminus\{e\},\quad A_0Q=D_0^{-1}\setminus\{e\}.
    \]

    Say that a subset $R\subseteq \Gamma$ is admissible if it satisfies 
    \[
    (A_0R) \cap (D_0R)=\emptyset.
    \]
    We note straightaway that $F$ is admissible. Indeed, assume towards a contradiction that $ag=dh$ for $a\in A_0, d\in D_0$ and $g,h\in F$. Then $h^{-1}d^{-1}ag=e$. Since $d^{-1}a\in D_0^{-1}A_0=A_0$, we can write 
    \[
    d^{-1}a=s^{j_0}q_1s^{j_1}\dots q_ms^{j_m},
    \]
    which is a reduced word with $0<|j_i|\leq n,\, q_i\in Q$. Then,
    \begin{equation}\label{eq:contradiction-identity}
    e=h^{-1}d^{-1}ag=h^{-1}s^{j_0}q_1s^{j_1}\dots q_ms^{j_m}g.
    \end{equation}
    Now, if both $g,h$ are non-trivial, then $h^{-1},g\in Q$, and thus \eqref{eq:contradiction-identity} contradicts $2n$-free independence.
    If both $h=g=e$, then \eqref{eq:contradiction-identity} contradicts $e\notin A_0$.
    If $g\neq e$ and $h=e$, then $d^{-1}ag \in A_0 Q =D_0^{-1}\setminus\{e\}$, a contradiction, and analogously for $g=e$, $h\neq e$.

    We invoke Zorn's Lemma to find a maximal admissible subset $R$ such that $F\subseteq R\subseteq\Gamma$.
    Next, define 
    \[
    A:=A_0R,\quad D:=D_0R.
    \]
    Since the set $R$ is admissible, we have that $A\cap D=\emptyset$.
    We claim that $A\cup D=\Gamma$.
    Suppose, towards a contradiction, that there exists $x\in\Gamma\setminus (A\cup D)$. Maximality of $R$ implies that $R\cup\{x\}$ cannot be admissible. 
    Note that 
    \[
    (A_0x)\cap(D_0x)=\emptyset 
    \]
    because $A_0\cap D_0=\emptyset$: otherwise $a\in A_0\cap D_0$ would give $e=a^{-1}a\in D_0^{-1}A_0=A_0$, a contradiction.
    Hence, we have that one of the following must be true:
    \[
    (A_0x)\cap(D_0R)\neq\emptyset \, \text{ or } \, (D_0x)\cap(A_0R)\neq\emptyset.
    \]
    Assuming that the first holds true, we can write $ax=dr$ for some $a\in A_0$, $d\in D_0$ and $r\in R$. This implies that 
    \[
    x=a^{-1}dr \in A_0^{-1}D_0R = A_0R=A,
    \]
    a contradiction.
    Assume then that the second condition holds true. Then, $dx=ar$ for some $a\in A_0$, $d\in D_0$ and $r\in R$. But this implies that 
    \[
    x=d^{-1}ar \in D_0^{-1}A_0R=A_0R=A,
    \]
    which is again a contradiction.
    This proves the claim that $\Gamma=A\cup D$.

    Now, since we chose $e\in D_0$, we have that $F\subseteq R\subseteq D$, which gives 
    \[
    A\cap F=\emptyset.
    \]
    Next, observe that $PD=PD_0R=A_0R=A$. Therefore, for any $1\leq k\leq n$, 
    \[
    s^k(\Gamma\setminus A)=s^kD\subseteq A.
    \]
    Moreover, since we have $QA=QA_0R\subseteq D_0R=D$, it follows that 
    \[
    A\cap qA=\emptyset
    \]
    for all $q\in Q$.
    When $g,h\in F$ are distinct, we have that $g^{-1}h\in Q$, and thus 
    \[
    gA \cap hA=\emptyset.
    \]
    This implies that $\{gA\}_{g\in F}$ are pairwise disjoint.
    This proves that the pair $(s,A)$ satisfies the cyclic Powers property with respect to $(F,n)$, and thus $\Gamma$ has the cyclic Powers property.
\end{proof}

We now apply this criterion to Zariski-dense linear groups.

For what follows, recall that a \emph{local field} $\mathbb K$ is a field equipped with a \emph{non-trivial} absolute value $|\cdot|\colon\mathbb K \to[0,+\infty)$ that induces a locally compact topology on $\mathbb K$. Non-triviality of $|\cdot |$ here means that there exists $a\in\mathbb K^\times$ with $|a|\neq1$.

\begin{definition}[{see~\cite[Definition 3.7]{JTT25} or \cite{Benoist97}}]
    Let \(d\geq2\), fix a local field $\mathbb K$, and let \(a\in\operatorname{PSL}(d,\mathbb K)\). Choose a lift \(\widetilde a\in\operatorname{SL}(d,\mathbb K)\). List its eigenvalues $(\lambda_i)_{i=1}^d$ in an algebraic closure of \(\mathbb K\), counted with algebraic multiplicity, so that
    \[
        |\lambda_1|\geq|\lambda_2|\geq\cdots\geq|\lambda_d|,
    \]
    using the extension of the absolute value of \(\mathbb K\). Then \(a\) is said to be \emph{biproximal} if both \(\widetilde a\) and \(\widetilde a^{-1}\) are \emph{proximal}, namely 
    \[
        |\lambda_1|>|\lambda_2|
        \qquad\text{and}\qquad
        |\lambda_{d-1}|>|\lambda_d|.
    \]
    This condition is independent of the lift, since multiplying a matrix by a non-zero scalar preserves these inequalities.
\end{definition}

\begin{theorem}\label{thm:Zariskidense-cyclicPowers}
    Let $d\geq2$, and $\mathbb K$ be a local field.
    Then, every Zariski-dense discrete subgroup $\Gamma$ of $\mathrm{PSL}(d,\mathbb K)$ that contains a biproximal element has the cyclic Powers property.
\end{theorem}

\begin{proof}
    Let $a\in\Gamma$ be biproximal, and fix a lift $\widetilde a\in\mathrm{SL}({d},\mathbb K)$.
    Enumerate its eigenvalues in an algebraic closure of $\mathbb K$, counted with algebraic multiplicity, so that
    \[
        |\lambda_1|\geq|\lambda_2|\geq\dots\geq|\lambda_{d}|>0.
    \]
    By biproximality, we have
    \[
        |\lambda_1|>|\lambda_2|,\qquad |\lambda_{{d}-1}|>|\lambda_{d}|.
    \]
    These inequalities imply that $\lambda_1$ and $\lambda_{d}$ have algebraic multiplicity one.
    Moreover, since every irreducible polynomial over a local field has all its roots of the same absolute value, the minimal polynomials of $\lambda_1$ and $\lambda_{d}$ over $\mathbb K$ must have degree one, and thus $\lambda_1,\lambda_{d}\in\mathbb K$.
    Choose non-zero eigenvectors $v_+,v_-\in\mathbb K^{d}$ corresponding to $\lambda_1$ and $\lambda_{d}$, respectively, and set
    \[
        V_+:=\operatorname{im}(\widetilde a-\lambda_{d} I),\qquad V_-:=\operatorname{im}(\widetilde a-\lambda_1 I).
    \]
    These are $\widetilde a$-invariant hyperplanes satisfying
    \[
        \mathbb K^{d}=\mathbb K v_+\oplus V_-,\qquad \mathbb K^{d}=\mathbb K v_-\oplus V_+.
    \]
    We introduce the following notation:
    \[
        p_+:=[v_+],\qquad p_-:=[v_-],\qquad H_+:=\mathbb P(V_+),\qquad H_-:=\mathbb P(V_-).
    \]
    The point $p_{+}$ (respectively, $p_-$) is called \emph{attracting point} of $a$ (respectively, of $a^{-1}$), while $H_-$ (respectively, $H_+$) is called \emph{repelling hyperplane} of $a$ (respectively, $a^{-1}$).
    Notice also that $p_+\in H_+$ and $p_-\in H_-$, whereas $p_+\notin H_-$ and $p_-\notin H_+$.
    
    Next, fix a finite subset $Q\subseteq\Gamma\setminus\{e\}$, which we may assume is non-empty by enlarging it if necessary.
    We claim that there exists $g\in\Gamma$ such that
    \[
        hgp_i\notin gH_j
    \]
    for all $h\in Q$ and $i,j\in\{+,-\}$.
    For each triple $(h,i,j)$, consider
    \[
        Z_{h,i,j}:=\{g\in\mathrm{PSL}({d},\mathbb K)\mid hgp_i\in gH_j\}.
    \]
    The incidence variety $\{(p,H)\mid p\in H\}$ is Zariski-closed in $\mathbb P^{d-1}(\mathbb K)\times\mathbb P((\mathbb K^d)^*)$, and the map $g\mapsto(hgp_i,gH_j)$ is algebraic. Consequently, its preimage $Z_{h,i,j}$ is Zariski-closed.
    We show that each $Z_{h,i,j}$ is also proper.
    Fix $h\in Q$ and $i,j\in\{+,-\}$, and let $\widetilde h\in\mathrm{SL}({d},\mathbb K)$ be a lift of $h$.
    Since $h\neq e$, the matrix $\widetilde h$ is not scalar.
    Then, there exists a non-zero vector $v\in\mathbb K^{d}$ such that $v$ and $\widetilde hv$ are linearly independent, otherwise $\widetilde h$ would preserve every line and hence would be scalar.
    Now, choose a hyperplane $W\subseteq\mathbb K^{d}$ as follows, 
    \begin{enumerate}[leftmargin=*,label=\textup{(\roman*)}]
        \item if $p_i\in H_j$, choose $W\subseteq\mathbb K^{d}$ such that $v\in W$ but $\widetilde hv\notin W$,
        \item if $p_i\notin H_j$, start by choosing any non-zero vector $v\in\mathbb K^{d}$. Since $v$ and $\widetilde hv$ are non-zero, the sets
        \[
            \{\varphi\in(\mathbb K^{d})^*\mid\varphi(v)=0\},\qquad \{\varphi\in(\mathbb K^{d})^*\mid\varphi(\widetilde hv)=0\}
        \]
        are proper linear subspaces of $(\mathbb K^{d})^*$.
        Their union cannot cover $(\mathbb K^{d})^*$, so we may choose a linear functional $\varphi$ satisfying $\varphi(v)\neq0$ and $\varphi(\widetilde hv)\neq0$.
        Then $W=\ker(\varphi)$ contains neither $v$ nor $\widetilde hv$.
    \end{enumerate}
    In either case, we have $\mathbb K v_i\subseteq V_j$ if and only if $\mathbb K v\subseteq W$.
    Choosing bases adapted to these pairs, we obtain $B\in\mathrm{GL}({d},\mathbb K)$ satisfying
    \[
        B(\mathbb K v_i)=\mathbb K v,\qquad BV_j=W.
    \]
    To arrange determinant one, choose a basis $w_1,\dots,w_{d}$ adapted to $(\mathbb K v,W)$ with $w_1=v$.
    Let $D$ multiply $w_1$ by $(\det B)^{-1}$ and fix the remaining basis vectors.
    Then $D$ preserves both $\mathbb K v$ and $W$, and $\det D=(\det B)^{-1}$.
    Consequently, $A=DB\in\mathrm{SL}({d},\mathbb K)$ still satisfies
    \[
        A(\mathbb K v_i)=\mathbb K v,\qquad AV_j=W.
    \]
    Denote by $g\in\mathrm{PSL}({d},\mathbb K)$ the image of $A$ under the quotient map. Observe that the equality \(A(\mathbb K v_i)=\mathbb K v\) means that \(Av_i=cv\) for some non-zero scalar \(c\in\mathbb K\). Consequently,
    \[
    hgp_i=[\widetilde hAv_i]=[c\widetilde hv]=[\widetilde hv].
    \]
    Note also that since \(H_j=\mathbb P(V_j)\), $gH_j=\mathbb P(AV_j)=\mathbb P(W)$.
    We conclude that $Z_{h,i,j}$ is proper because the $g$ we constructed satisfies 
    \[
        hgp_i = [\widetilde hv] \not\in \mathbb P(W)=gH_j.
    \]
    
    The space $\mathrm{PSL}({d},\mathbb K)$ is Zariski irreducible, being the image of the Zariski irreducible space $\mathrm{SL}({d},\mathbb K)$ under the quotient map.
    Thus the finite union of the proper Zariski-closed sets $Z_{h,i,j}$ cannot cover it.
    Consequently,
    \[
        Y=\mathrm{PSL}({d},\mathbb K)\setminus\bigcup_{\substack{h\in Q\\i,j\in\{+,-\}}}Z_{h,i,j}
    \]
    is non-empty and Zariski-open.
    By Zariski density, we may choose $g\in\Gamma\cap Y$, i.e., $g\in\Gamma$ satisfies $hgp_i\notin gH_j$ for all $h\in Q$ and $i,j\in\{+,-\}$, thus proving the claim.
    
    Now, set $b=gag^{-1}$, and introduce the notation
    \[
        p_b^\pm=gp_\pm,\qquad H_b^\pm=gH_\pm.
    \]
    By the claim, $hp_b^i\notin H_b^j$ for all $h\in Q$ and $i,j\in\{+,-\}$.
    Then $\{p_b^+,p_b^-\}\subseteq H_b^+\cup H_b^-$, whereas
    \[
        \left(\bigcup_{h\in Q}\{hp_b^+,hp_b^-\}\right)\cap(H_b^+\cup H_b^-)=\emptyset.
    \]
    Since $\mathbb P^{{d}-1}(\mathbb K)$ is compact Hausdorff and $\bigcup_{h\in Q}\{hp_b^+,hp_b^-\}$ is finite, we may choose an open set $V\subseteq\mathbb P^{{d}-1}(\mathbb K)$ such that
    \[
        \bigcup_{h\in Q}\{hp_b^+,hp_b^-\}\subseteq V,\qquad \overline V\cap(H_b^+\cup H_b^-)=\emptyset.
    \]
    Define
    \[
        U=(\mathbb P^{{d}-1}(\mathbb K)\setminus\overline V)\cap\bigcap_{h\in Q}h^{-1}V,
    \]
    and observe that $U$ is open, contains both $p_b^+$ and $p_b^-$, and satisfies
    \[
        U\cap V=\emptyset,\qquad \bigcup_{h\in Q}hU\subseteq V.
    \]
    
    Since $b$ is conjugate to the biproximal element $a$, both $b$ and $b^{-1}$ are proximal.
    Their attracting points are $p_b^+$ and $p_b^-$, and their repelling hyperplanes are $H_b^-$ and $H_b^+$, respectively.
    The set $\overline V$ is compact and disjoint from both repelling hyperplanes, while $U$ is an open neighborhood of both attracting points.
    Applying \cite[Lemma~3.8(i)]{JTT25} to $b$ and $b^{-1}$, we obtain an integer $N\geq1$ such that
    \[
        b^m\overline V\subseteq U,\qquad b^{-m}\overline V\subseteq U
    \]
    for all $m\geq N$.
    Set now $s=b^N\in\Gamma$, and observe that, for every $k\in\mathbb Z\setminus\{0\}$, we have $N|k|\geq N$, and hence
    \[
        s^kV\subseteq U,\qquad hU\subseteq V
    \]
    for all $k\in\mathbb Z\setminus\{0\}$ and $h\in Q$, with $U$ and $V$ non-empty, open, and disjoint.
    By a standard ping-pong argument (see for instance \cite[Lemma 2.4]{Vigdorovich25}),  $s$ is freely independent from $Q$, and thus, in particular, $n$-freely independent for any $n\geq1$.
    Now we invoke Lemma~\ref{lem:free-independence-CPproperty} to infer that $\Gamma$ has the cyclic Powers property.
\end{proof}

In the following corollary, we recapitulate the classes of groups with the cyclic Powers property (and thus inner \Cs-selfless groups) found in this section.

\begin{corollary}\label{cor:examples}
    The following classes of groups have the cyclic Powers property, and therefore are inner \Cs-selfless:
    \begin{enumerate}[leftmargin=*,label=\textup{(\roman*)}]
        \item Discrete groups that admit a  topologically free \emph{extreme boundary action} on a compact Hausdorff space $X$ with $|X|>2$, and 
        \item Zariski-dense discrete subgroups of $\mathrm{PSL}(d,\mathbb K)$ for a local field $\mathbb K$ and $d\geq2$ that contain a biproximal element. In particular, this class includes all Zariski-dense discrete subgroups of $\mathrm{PSL}(n,\mathbb R)$.
    \end{enumerate}
\end{corollary}

\begin{proof}
    Statement (i) follows from Proposition~\ref{prop:extremelyproximal-cP} and statement (ii) follows from Theorem~\ref{thm:Zariskidense-cyclicPowers}. The fact that any Zariski-dense subgroup of $\mathrm{PSL}(n,\mathbb R)$ contains a biproximal element follows from \cite[Théorème~A.1]{BL1993}.
\end{proof}

Furthermore, we answer a (far more general version of a) question of Bridson and de la Harpe posed more than twenty years ago in \cite[p.202]{BridsondelaHarpe2004} for $\mathrm{PSL}(n,\mathbb Z)$.
This follows from Corollary~\ref{cor:examples} and Lemma~\ref{lem:cyclicPowers-to-Powers}.

\begin{corollary}
\label{cor:linear-Powers}
    Let $\Gamma$ be a Zariski-dense discrete subgroup of $\mathrm{PSL}(d,\mathbb K)$ for a local field $\mathbb K$ and $d\geq2$ that contains a biproximal element.
    Then $\Gamma$ has the Powers property.
    
    In particular, any Zariski-dense discrete subgroup of $\mathrm{PSL}(n,\mathbb R)$ has the Powers property.
\end{corollary}

We end the section with an observation about the relation between inner \Cs-selfless groups and tensorial decomposition of their reduced group \Cs-algebras that arises from Corollary~\ref{cor:examples}.

\begin{remark}
    Recall that, by Proposition~\ref{InnerPropertiesProposition}(ii), inner \Cs-selfless groups are directly indecomposable. 
    One might be tempted to think that such an indecomposability holds also, more generally, at the \Cs-algebraic level. However, we present an example exhibiting that this can fail quite spectacularly.
    Let us say that a \Cs-algebra $A$ is \emph{tensorially indecomposable} if whenever $A \cong A_1\otimes A_2$ for some \Cs-algebras $A_i$, $i=1,2$, then $A_1\cong\mathbb C$ or $A_2\cong\mathbb C$. Moreover, a related notion says that $A$ is \emph{tensorially prime} if whenever  $A \cong A_1\otimes A_2$, then $A_1$ or $A_2$ must be finite-dimensional; this notion already appeared for instance in \cite{Rordam2003,KirchbergClassification,Ror02,MilhojRordam2023}. 

    By \cite[Proposition~7.1(b)]{Oml26b}, the reduced group \Cs-algebra of $\mathrm{SL}_{\infty}(\mathbb Z):=\bigcup_{n\geq1}\mathrm{SL}(2n+1,\mathbb Z)$ is $\mathcal Z$-stable, i.e.,
$\Csr(\mathrm{SL}_{\infty}(\mathbb Z))
\cong\Csr(\mathrm{SL}_{\infty}(\mathbb Z))\otimes\mathcal Z$. Also, as $\mathrm{SL}(2n+1,\mathbb Z)=\mathrm{PSL}(2n+1,\mathbb Z)$ is inner \Cs-selfless by combining Corollaries~\ref{cor:examples} and \ref{cor:cP-to-innerselfless}, it follows from Proposition~\ref{prop:directed-union} that $\mathrm{SL}_{\infty}(\mathbb Z)$ is inner \Cs-selfless. Hence, this provides an explicit example of an inner
\Cs-selfless group whose reduced group \Cs-algebra is
neither tensorially indecomposable nor tensorially prime.
\end{remark}

\vspace{3mm}


\section{MIF groups and the ISR Property}
\label{sec:MIFandISR}
We use tracial ultrapowers to study invariant von Neumann
subalgebras of crossed products by non-trivial MIF groups.
This gives a tracial analogue of the extraction results
in Section~\ref{sec:invC*crossed}.

Recall that a group $\Gamma$ has the \emph{invariant subalgebra rigidity property}
(ISR) if every unital von Neumann subalgebra $P\subseteq\vN(\Gamma)$
satisfying $\lambda_gP\lambda_g^*=P$ for every $g\in\Gamma$ is of the
form $P=\vN(\Lambda)$ for some normal subgroup
$\Lambda\triangleleft\Gamma$ (see~\cite{AmrutamJiangISR}).

Similarly, the group $\Gamma$ has the \emph{relative invariant subalgebra rigidity property} (relative ISR) if, for every non-trivial normal subgroup $N\triangleleft\Gamma$, every unital
von Neumann subalgebra $P\subseteq\vN(\Gamma)$ invariant under
conjugation by $N$ is of the form $P=\vN(\Lambda)$ for some subgroup
$\Lambda\leq\Gamma$ (see \cite{amrutam2026relative}). Such a subgroup $\Lambda$ is necessarily normalized by $N$.

We first express mixed identity-freeness in terms of an ultrapower
embedding.

\begin{proposition}\label{prop:characterization}
For a non-trivial discrete group $\Gamma$, the following are equivalent.
\begin{enumerate}[label=\textup{(\roman*)},leftmargin=*]
\item $\Gamma$ is mixed identity-free.
\item There exist a set $I$, an ultrafilter $\mathcal U$ on $I$, a family
$(g_i)_{i\in I}\subseteq\Gamma$, and a normal unital trace-preserving
embedding
$\theta:\vN(\Gamma\mathbin{\star}\langle t\rangle)
\hookrightarrow\vN(\Gamma)_{\mathcal U}$
such that $\langle t\rangle\cong\Z$,
$\theta|_{\vN(\Gamma)}=\Delta$, and
$\theta(\lambda_t)=(\lambda_{g_i})_{i\to\mathcal U}$.
\end{enumerate}
\end{proposition}

\begin{proof}
Suppose first that $\Gamma$ is MIF, and let
$I:=\{F\subseteq(\Gamma\mathbin{\star}\langle t\rangle)
\setminus\{e\}\mid F\text{ finite}\}$, directed by inclusion.
By the finite simultaneous-avoidance argument in the proof of
\cite[Proposition~5.3]{HO}, choose a family
$(g_F)_{F\in I}\subseteq\Gamma$ such that $h(g_F)\ne e$ for every
$F\in I$ and $h\in F$.
Let $\varphi_F:\Gamma\mathbin{\star}\langle t\rangle\to\Gamma$
be the homomorphism fixing $\Gamma$ pointwise and sending $t$ to $g_F$.
For any cofinal ultrafilter $\mathcal U$ on $I$, define
$\pi:\Gamma\mathbin{\star}\langle t\rangle
\to\vN(\Gamma)_{\mathcal U}$ by
$\pi(h)=(\lambda_{\varphi_F(h)})_{F\to\mathcal U}$.
We claim that $\pi$ defines a unitary representation and
$\tr_{\mathcal U}(\pi(h))=\delta_{h,e}$ for every
$h\in\Gamma\mathbin{\star}\langle t\rangle$.

Indeed, each $\varphi_F$ is a homomorphism, so
$\pi(hk)=\pi(h)\pi(k)$ and $\pi(e)=1$, and each $\pi(h)$ is unitary.
If $h\ne e$, then $\varphi_F(h)\ne e$ for every $F\supseteq\{h\}$.
Consequently, by cofinality,
$\tr_{\mathcal U}(\pi(h))
=\lim_{F\to\mathcal U}\tr_\Gamma(\lambda_{\varphi_F(h)})=0$.
Since $\tr_{\mathcal U}(\pi(e))=1$, this proves the claim.

Set $N=W^*(\pi(\Gamma\mathbin{\star}\langle t\rangle))$
with its restricted trace.
The trace identity shows that $\delta_h\mapsto\pi(h)$ extends to a
unitary
$\ell^2(\Gamma\mathbin{\star}\langle t\rangle)
\to L^2(N,\tr_{\mathcal U})$.
Its range is dense because the linear span of
$\pi(\Gamma\mathbin{\star}\langle t\rangle)$ is a weakly dense unital
$*$-subalgebra of $N$.
This unitary intertwines the left regular representation of
$\Gamma\mathbin{\star}\langle t\rangle$ with left multiplication by
$\pi(\Gamma\mathbin{\star}\langle t\rangle)$.
It therefore induces a normal trace-preserving isomorphism
$\vN(\Gamma\mathbin{\star}\langle t\rangle)\to N$.
Composing with $N\subseteq\vN(\Gamma)_{\mathcal U}$ gives $\theta$.
It agrees with the diagonal embedding on all canonical unitaries
from $\Gamma$, hence on $\vN(\Gamma)$ by normality, and
$\theta(\lambda_t)=(\lambda_{g_F})_{F\to\mathcal U}$.

Conversely, suppose that \textup{(ii)} holds, with an ultrafilter
$\mathcal U$ on $I$ and a family $(g_i)_{i\in I}$.
For any $e\ne h\in\Gamma\mathbin{\star}\langle t\rangle$,
trace preservation yields
\[
0=\tr_{\mathcal U}(\theta(\lambda_h))
=\lim_{i\to\mathcal U}\tr_\Gamma(\lambda_{h(g_i)})
=\lim_{i\to\mathcal U}\mathbf 1_{\{h(g_i)=e\}}.
\]
Thus $h(g_i)\ne e$ for some $i\in I$, so $\Gamma$ has no non-trivial
one-variable mixed identity.
By \cite[Remark~5.1]{HO}, $\Gamma$ is MIF.
\end{proof}

Throughout, $(M,\tau)$ denotes a finite von Neumann algebra with faithful
normal tracial state, $\|a\|_2:=\tau(a^*a)^{1/2}$, and $\mathbb E_M$
denotes the canonical conditional expectation onto $M$.
\begin{definition}\label{def:L2-compatible}
Let $\Gamma$ be a non-trivial MIF discrete group, and let
$\alpha\colon\Gamma\curvearrowright(M,\tau)$ be a trace-preserving action.
We say that $\alpha$ is \emph{$L^2$-compatible} (with respect to $M$) if the following hold:
\begin{enumerate}
    \item there exists an ultrafilter $\mathcal U$ on a set $I$, a family
$(g_i)_{i\in I}\subseteq\Gamma$, and a normal unital trace-preserving
embedding
$\theta:\vN(\Gamma\mathbin{\star}\langle t\rangle)
\hookrightarrow\vN(\Gamma)_{\mathcal U}$
such that $\langle t\rangle\cong\Z$,
$\theta|_{\vN(\Gamma)}=\Delta$ and $\theta(\lambda_t)=(\lambda_{g_i})_{i\to\mathcal U}$.
\item $\lim_{i\to\mathcal U}\|\alpha_{g_i}(a)-a\|_2=0$
for every $a\in M$.
\end{enumerate}
\end{definition}


\subsection{Fourier coefficient extraction}

We first prove the extraction result for a free product in which the
added generator acts trivially.

\begin{theorem}\label{TracialFreeProductCaseTheorem}
Let $\Gamma=\Lambda\star\mathbb Z$ be the free product of a discrete
group $\Lambda$ by the integers, and let $h$ generate $\mathbb Z$.
Let $\alpha\colon\Gamma\curvearrowright(M,\tau)$ be a trace-preserving
action with $\alpha_h=\mathrm{id}_M$.
Suppose that $B\subseteq M\rtimes_\alpha\Gamma$ is a possibly non-unital
von Neumann subalgebra invariant under conjugation by
$L:=\langle g^{-1}hg\mid g\in\Lambda\rangle$.
Then $\mathbb E_M(x\lambda_g^*)\lambda_g\in B$ for every
$x\in B\cap(M\rtimes_{\alpha|_\Lambda}\Lambda)$ and $g\in\Lambda$.
\end{theorem}

\begin{proof}
Equip $M\rtimes_\alpha\Gamma$ with the trace $\tau\circ\mathbb E_M$.
If $\Lambda=\{e\}$, the assertion is immediate.
Thus assume that $\Lambda\neq\{e\}$ and write
$L:=\langle f^{-1}hf\mid f\in\Lambda\rangle$.
Since $\alpha_h=\mathrm{id}_M$, the restriction of $\alpha$ to $L$
is trivial.

Fix $x\in B\cap(M\rtimes_{\alpha|_\Lambda}\Lambda)$ and $g\in\Lambda$,
and write $x(f):=\mathbb E_M(x\lambda_f^*)$ for $f\in\Lambda$.
By Kaplansky density, choose finite Fourier sums
$x_k=\sum_{f\in F_k}x_k(f)\lambda_f$ such that
$\|x-x_k\|_2\to0$ and $\sup_k\|x_k\|\leq C$ for some $C>0$.
Then $\|x_k(f)\|\leq C$ and
$\|x_k(f)-x(f)\|_2\leq\|x_k-x\|_2$.
We may enlarge each $F_k$ to contain $e$ and $g$.
We use that uniformly operator-bounded $L^2$-limits of elements of
$B$ belong to $B$.

We proceed by distinguishing two cases.\\

\begin{itemize}
\item \emph{Case 1}: Assume that $g=e$. By Lemma~\ref{PlumpSubgroupLemma}, we can choose elements
$h_{k,1},\ldots,h_{k,m_k}\in L$ such that
$\|\frac1{m_k}\sum_{i=1}^{m_k}
\lambda_{h_{k,i}fh_{k,i}^{-1}}\|<1/(k|F_k|)$
for every $f\in F_k\setminus\{e\}$.
Since $\alpha_{h_{k,i}}=\mathrm{id}_M$, we obtain
\[
\begin{aligned}
&\left\|
\frac1{m_k}\sum_{i=1}^{m_k}
\lambda_{h_{k,i}}x\lambda_{h_{k,i}}^*
-\mathbb E_M(x)
\right\|_2\\
&\quad\leq
2\|x-x_k\|_2+
\left\|
\sum_{f\in F_k\setminus\{e\}}x_k(f)
\left(\frac1{m_k}\sum_{i=1}^{m_k}
\lambda_{h_{k,i}fh_{k,i}^{-1}}\right)
\right\|\\
&\quad\leq2\|x-x_k\|_2+\frac{C}{k}\longrightarrow0.
\end{aligned}
\]
The averages belong to $B$ and have operator norm at most $\|x\|$.
Thus $\mathbb E_M(x)\in B$.
\item \emph{Case 2: $g\neq e$.}
By the first case, we may replace $x$ by $x-\mathbb E_M(x)$.
Replacing $x_k$ by $x_k-\mathbb E_M(x_k)$ preserves uniform boundedness
and $L^2$-convergence, so we may also assume that
$F_k\subseteq\Lambda\setminus\{e\}$.

We first verify the first extraction step in this setting.
As before, put
\[
\mathbf w(g):=g(hg^{-1}h^{-1})(h^2gh^{-2})(h^3g^{-1}h^{-3})
\]
and
\[
\mathbf P(y):=
y(\lambda_hy^*\lambda_h^*)
(\lambda_{h^2}y\lambda_{h^2}^*)
(\lambda_{h^3}y^*\lambda_{h^3}^*).
\]
Since $B$ is $L$-invariant, $\mathbf P(x)\in B$.
Let $\mathbb E_{\mathbf w(g)}$ denote the conditional expectation
onto $M\rtimes\langle\mathbf w(g)\rangle$.
By Lemma~\ref{ElementPropertiesLemma}, $\mathbf w(g)$ is primitive
and loxodromic.
The trivial edge stabilizers of the Bass--Serre action therefore give
$C_\Gamma(\mathbf w(g)^r)=\langle\mathbf w(g)\rangle$
for every $r\neq0$.
Consequently, every element outside $\langle\mathbf w(g)\rangle$
has an infinite orbit under conjugation by $\mathbf w(g)$.
Since $\alpha_{\mathbf w(g)}=\mathrm{id}_M$, the mean ergodic theorem
gives
\[
\frac1m\sum_{j=1}^m
\lambda_{\mathbf w(g)^j}y\lambda_{\mathbf w(g)^j}^*
\longrightarrow\mathbb E_{\mathbf w(g)}(y)
\quad\text{in }\|\cdot\|_2
\]
for every $y\in M\rtimes_\alpha\Gamma$.
The averages are uniformly operator bounded, so
$\mathbb E_{\mathbf w(g)}(B)\subseteq B$.

The Fourier calculation in Proposition~\ref{FirstExtractionStep}
and Lemma~\ref{ElementPropertiesLemma} give
\[
\mathbb E_{\mathbf w(g)}(\mathbf P(x_k))
=(x_k(g)x_k(g)^*)^2\lambda_{\mathbf w(g)}.
\]
After increasing $C$ so that $\|x\|,\|x_k\|\leq C$, telescoping and
$\|ayb\|_2\leq\|a\|\|y\|_2\|b\|$ yield
$\|\mathbf P(x_k)-\mathbf P(x)\|_2
\leq4C^3\|x_k-x\|_2\to0$.
The same estimate applies to the fourfold products of the Fourier
coefficients.
Since $\mathbb E_{\mathbf w(g)}$ is $L^2$-contractive, it follows that
$\mathbb E_{\mathbf w(g)}(\mathbf P(x))
=(x(g)x(g)^*)^2\lambda_{\mathbf w(g)}\in B$.

For $\varepsilon>0$, put
\[
\begin{aligned}
p_\varepsilon
&:=(x(g)x(g)^*)^2
\bigl((x(g)x(g)^*)^2+\varepsilon1\bigr)^{-1},\\
q_\varepsilon
&:=\alpha_{g^{-1}}\!\left(
(x(g)^*x(g))^2
\bigl((x(g)^*x(g))^2+\varepsilon1\bigr)^{-1}
\right).
\end{aligned}
\]
These are positive contractions and
$p_\varepsilon x(g)=x(g)\alpha_g(q_\varepsilon)$.
Apply the functional-calculus argument of
Lemma~\ref{FunctionalCalculusLemma} to the extracted element,
and also to $x^*$ at $g^{-1}$.
Since $x^*(g^{-1})=\alpha_{g^{-1}}(x(g)^*)$, we obtain
$p_\varepsilon\lambda_s\in B$ for
$s\in\langle\mathbf w(g)\rangle^L$ and
$q_\varepsilon\lambda_t\in B$ for
$t\in\langle\mathbf w(g^{-1})\rangle^L$.

By Proposition~\ref{InnerPropertiesProposition}, the subgroup
$N:=\langle\mathbf w(g^{-1})\rangle^L
\cap g^{-1}\langle\mathbf w(g)\rangle^Lg$
is a non-trivial normal subgroup of $L$.
Thus $p_\varepsilon\lambda_{gng^{-1}},
q_\varepsilon\lambda_n\in B$ for every $n\in N$.
By Lemma~\ref{PlumpSubgroupLemma}, choose
$n_1,\ldots,n_{m_k}\in N$ such that
$\|\frac1{m_k}\sum_{i=1}^{m_k}
\lambda_{n_i(g^{-1}f)n_i^{-1}}\|<1/(k|F_k|)$
for every $f\in F_k\setminus\{g\}$.
Since $\alpha|_L=\mathrm{id}_M$, we have
\[
\begin{aligned}
&\left\|
\frac1{m_k}\sum_{i=1}^{m_k}
p_\varepsilon\lambda_{gn_ig^{-1}}
xq_\varepsilon\lambda_{n_i^{-1}}
-p_\varepsilon x(g)\alpha_g(q_\varepsilon)\lambda_g
\right\|_2\\
&\quad\leq2\|x-x_k\|_2\\
&\qquad+
\left\|
\sum_{f\in F_k\setminus\{g\}}
p_\varepsilon x_k(f)\alpha_f(q_\varepsilon)\lambda_g
\left(\frac1{m_k}\sum_{i=1}^{m_k}
\lambda_{n_i(g^{-1}f)n_i^{-1}}\right)
\right\|\\
&\quad\leq2\|x-x_k\|_2+\frac{C}{k}\longrightarrow0.
\end{aligned}
\]
Each average belongs to $B$ and has operator norm at most $\|x\|$.
Hence $p_\varepsilon x(g)\alpha_g(q_\varepsilon)\lambda_g\in B$.

Finally, the identity
$p_\varepsilon x(g)=x(g)\alpha_g(q_\varepsilon)$
and functional calculus give
\[
\begin{aligned}
\|p_\varepsilon x(g)\alpha_g(q_\varepsilon)-x(g)\|
&\leq2\sup_{t\geq0}
\frac{\varepsilon\sqrt t}{t^2+\varepsilon}\\
&\leq2\varepsilon^{1/4}\longrightarrow0.
\end{aligned}
\]
Since $B$ is norm closed, $x(g)\lambda_g\in B$.
\end{itemize}

This finishes the proof. \end{proof}

We now prove Fourier coefficient extraction for $L^2$-compatible
actions of MIF discrete groups.

\begin{theorem}\label{thm:ISRMIF}\label{TracialInvariantMainTheorem}
Let $\Gamma$ be a non-trivial mixed identity-free discrete group, and let
$\alpha\colon\Gamma\curvearrowright(M,\tau)$ be a trace-preserving action.
Also, assume that $\alpha$ is $L^2$-compatible.

Let $N\triangleleft\Gamma$ be a non-trivial normal subgroup, and let
$B\subseteq M\rtimes_\alpha\Gamma$ be a possibly non-unital von Neumann
subalgebra invariant under conjugation by $N$.
Then $\mathbb E_M(x\lambda_g^*)\lambda_g\in B$ for every $x\in B$
and $g\in\Gamma$.
\end{theorem}

\begin{proof}
Choose $\mathcal U$ and $(g_i)_{i\in I}$ as in
Definition~\ref{def:L2-compatible} given by the $L^2$-compatibility.
Fix $n\in N\setminus\{e\}$ and put
$n_i:=(g_i n g_i^{-1})(g_i^2n^{-1}g_i^{-2})$.
The normal-form argument in the proof of Lemma~\ref{NormalWitnessLemma},
applied to the tracial embedding in
Proposition~\ref{prop:characterization}, shows that $(n_i)_{i\in I}$
again satisfies condition~\textup{(ii)} of that proposition.
The estimate in the proof of Lemma~\ref{StrongNormalWitnessLemma}
applies in $\|\cdot\|_2$, since $\alpha$ preserves $\tau$.
Thus the same $L^2$-convergence holds for $(n_i)_{i\in I}$, and we may
assume that $g_i\in N$ for every $i\in I$.

Let $P:=M\rtimes_\alpha\Gamma$, and let $\widetilde\alpha$ be the
extension of $\alpha$ to $\Gamma\star\langle t\rangle$ with
$\widetilde\alpha_t=\id_M$.
Compatibility implies that $(\lambda_{g_i})_{i\to\mathcal U}$
commutes with $\Delta(M)$ in $P_{\mathcal U}$.
Consequently, the assignments
\[
a\longmapsto\Delta(a),~
\lambda_g\longmapsto\Delta(\lambda_g),~
\lambda_t\longmapsto(\lambda_{g_i})_{i\to\mathcal U}
\]
for $a\in M$ and $g\in\Gamma$,
define a unital $\ast$-homomorphism on the algebraic crossed product.
For every non-identity word $w\in\Gamma\star\langle t\rangle$,
trace preservation of the given group embedding implies
$w(g_i)\ne e$ on a set belonging to $\mathcal U$.
Hence the image of $a\lambda_w$ has trace $\delta_{w,e}\tau(a)$.
The GNS-construction therefore gives a normal trace-preserving embedding
\[
\widetilde\theta\colon
M\rtimes_{\widetilde\alpha}(\Gamma\star\langle t\rangle)
\hookrightarrow P_{\mathcal U}
\]
whose restriction to $P$ is $\Delta$.

Set $C:=\widetilde\theta^{-1}(B_{\mathcal U})$ and
$L:=\langle g^{-1}tg\mid g\in\Gamma\rangle$.
For every $g\in\Gamma$, the unitary
$\widetilde\theta(\lambda_{g^{-1}tg})
=(\lambda_{g^{-1}g_i g})_{i\to\mathcal U}$
normalizes $B_{\mathcal U}$, since $g^{-1}g_i g\in N$.
Thus $C$ is an $L$-invariant von Neumann subalgebra.
Moreover, $\Delta(P)\cap B_{\mathcal U}=\Delta(B)$:
if $\Delta(y)=(b_i)_{i\to\mathcal U}$ for a uniformly bounded family
$(b_i)_{i\in I}\subseteq B$, then
$\lim_{i\to\mathcal U}\|y-b_i\|_2=0$, so $y\in B$ by ultraweak
closedness.
Consequently, $C\cap P=B$.

Apply Theorem~\ref{TracialFreeProductCaseTheorem} to $C$.
For every $x\in B$ and $g\in\Gamma$, it gives
$\mathbb E_M(x\lambda_g^*)\lambda_g\in C\cap P=B$.
\end{proof}
As in the $\Cs$-algebraic case, the trivial action gives the relative
ISR property.

\begin{corollary}\label{TracialRelativeISRCorollary}
Every non-trivial mixed identity-free discrete group $\Gamma$ has the relative ISR property.
More precisely, if $N\triangleleft\Gamma$ is non-trivial and
$B\subseteq\vN(\Gamma)$ is a unital von Neumann subalgebra invariant
under conjugation by $N$, then $B=\vN(\Lambda)$ for a subgroup
$\Lambda\leq\Gamma$ normalized by $N$.
In particular, $\Gamma$ has the ISR property.
\end{corollary}

\begin{proof}
Apply Theorem~\ref{TracialInvariantMainTheorem} to the trivial action
on $\mathbb C$.
Then $\tau_\Gamma(x\lambda_g^*)\lambda_g\in B$ for every $x\in B$
and $g\in\Gamma$.
Since $B$ is unital and closed under multiplication and adjoints,
$\Lambda:=\{g\in\Gamma\mid\lambda_g\in B\}$ is a subgroup with
$\vN(\Lambda)\subseteq B$.
Its normalization by $N$ follows from the invariance of $B$.

For $x\in B$ and $g\notin\Lambda$, the coefficient
$\tau_\Gamma(x\lambda_g^*)$ vanishes.
Thus $x=\mathbb E_\Lambda(x)\in\vN(\Lambda)$, proving
$B=\vN(\Lambda)$.
Taking $N=\Gamma$ gives the ISR property.
\end{proof}

We next obtain ISR for finite products.

\begin{corollary}\label{TracialProductISRCorollary}
Let $\Gamma_1,\ldots,\Gamma_n$ be non-trivial mixed identity-free discrete groups.
Then $\Gamma_1\times\cdots\times\Gamma_n$ has the ISR property.
\end{corollary}

\begin{proof}
Write $\Gamma:=\Gamma_1\times\cdots\times\Gamma_n$, and let
$B\subseteq\vN(\Gamma)$ be a unital $\Gamma$-invariant von Neumann
subalgebra.
For each $j$, view $\vN(\Gamma)$ as the crossed product of
$\vN(\prod_{k\ne j}\Gamma_k)$ by the trivial action of $\Gamma_j$.
Apply Theorem~\ref{TracialInvariantMainTheorem} with $N=\Gamma_j$.
Successively extracting the Fourier coefficient in each coordinate
gives $\tau_\Gamma(x\lambda_g^*)\lambda_g\in B$ for every $x\in B$
and $g\in\Gamma$.

Consequently, $\Lambda:=\{g\in\Gamma\mid\lambda_g\in B\}$ is a normal
subgroup, and every $x\in B$ satisfies $x=\mathbb E_\Lambda(x)$.
Thus $B=\vN(\Lambda)$.
\end{proof}

Fourier coefficient extraction also describes the invariant subalgebras
of these crossed products.

\begin{corollary}\label{TracialFourierStructureCorollary}
Under the hypotheses of Theorem~\ref{TracialInvariantMainTheorem},
the following statements hold.
\begin{enumerate}[label=\textup{(\roman*)},leftmargin=*]
\item One has
\[
B=\overline{\mathrm{span}}^{\,\sigma\text{-weak}}
\{a\lambda_g\mid g\in\Gamma,\ a\in M,\ a\lambda_g\in B\}.
\]
\item If $M$ is a factor and $M\subseteq B$, then
$B=M\rtimes_{\alpha|_\Lambda}\Lambda$ for a subgroup
$\Lambda\leq\Gamma$ normalized by $N$.
\end{enumerate}
\end{corollary}

\begin{proof}
For \textup{(i)}, denote the algebra on the right by $C$.
The homogeneous elements of $B$ span a $*$-algebra, so $C$ is a
von Neumann subalgebra contained in $B$.
By Theorem~\ref{TracialInvariantMainTheorem}, all finite Fourier sums
of $x\in B$ belong to $C$.
These sums converge to $x$ in $\|\cdot\|_2$.
Since $L^2(C)\cap(M\rtimes_\alpha\Gamma)=C$, we obtain $x\in C$.

For \textup{(ii)}, put
$I_g:=\{a\in M\mid a\lambda_g\in B\}$ for $g\in\Gamma$.
Since $M\subseteq B$, each $I_g$ is an ultraweakly closed two-sided
ideal of $M$.
As $M$ is a factor, $I_g$ is either $\{0\}$ or $M$.
Thus the non-zero homogeneous components of $B$ are precisely
$M\lambda_g$ for
$g\in\Lambda:=\{g\in\Gamma\mid\lambda_g\in B\}$.
This is a subgroup normalized by $N$, and \textup{(i)} gives
$B=M\rtimes_{\alpha|_\Lambda}\Lambda$.
\end{proof}

We now apply this description to intermediate von Neumann algebras.

\begin{corollary}\label{TracialReflectingCorollary}
Let $\alpha\colon\Gamma\curvearrowright(M,\tau)$ satisfy the action
and compatibility hypotheses of
Theorem~\ref{TracialInvariantMainTheorem}.
Then every intermediate von Neumann algebra
$\vN(\Gamma)\subseteq B\subseteq M\rtimes_\alpha\Gamma$
is of the form
$B=(B\cap M)\rtimes_{\alpha|_{B\cap M}}\Gamma$.
\end{corollary}

\begin{proof}
Since $\vN(\Gamma)\subseteq B$, the algebra $B$ is invariant under
conjugation by $\Gamma$, and $B\cap M$ is $\alpha$-invariant.
Apply Theorem~\ref{TracialInvariantMainTheorem} with $N=\Gamma$.
For $x\in B$ and $g\in\Gamma$, we obtain
$\mathbb E_M(x\lambda_g^*)\lambda_g\in B$.
Multiplying by $\lambda_g^*\in B$ gives
$\mathbb E_M(x\lambda_g^*)\in B\cap M$.
The conclusion follows from
Corollary~\ref{TracialFourierStructureCorollary}\textup{(i)}.
\end{proof}

\begin{remark}
The Fourier decomposition and reflecting conclusions require no
approximation property assumption on $\Gamma$.
\end{remark}

By Remark~3.2\textup{(i)} and Proposition~\ref{InnerPropertiesProposition}, every countable group admitting a topologically free extreme boundary action on a compact Hausdorff space with more than two points is mixed identity-free, and hence has the relative ISR property by Theorem~\ref{thm:ISRMIF}.\\
For the Baumslag--Solitar groups $\mathrm{BS}(m,n)$ with $|m|,|n|\geq2$ and $|m|\ne|n|$, mixed identity-freeness was already established in \cite[Theorem~B and Proposition~8.8]{FMMS}. As a consequence, the following corollary  answers the Baumslag--Solitar case of \cite[Question~6.1]{DJ}.

\begin{corollary}\label{cor:BS}
Let $m,n\in\Z$ satisfy $|m|,|n|\geq2$ and $|m|\ne|n|$.
Then the Baumslag--Solitar group
$\mathrm{BS}(m,n)=\langle a,t\mid ta^mt^{-1}=a^n\rangle$
has the relative ISR property.
\end{corollary}

\begin{proof}
By \cite[Theorem~B and Proposition~8.8]{FMMS}, the group
$\mathrm{BS}(m,n)$ is mixed identity-free.
Apply Corollary~\ref{TracialRelativeISRCorollary}.
\end{proof}
We also obtain a class of amenable groups with ISR-property.
\begin{corollary}\label{cor:jacobson}
\begin{enumerate}[label=\textup{(\roman*)},leftmargin=*]
\item There exists a two-generated infinite elementary amenable discrete
group with the relative ISR property.
\item For each prime $p$, there exists an infinite locally finite discrete $p$-group with the relative ISR property.
\item There exist continuum many pairwise non-isomorphic finitely
generated non-amenable discrete groups with the relative ISR property which contain no non-Abelian free subgroup and have non-trivial amenable radical.
\end{enumerate}
\end{corollary}

\begin{proof}
For \textup{(i)} and \textup{(ii)}, apply
\cite[Theorems~1.2 and~1.3]{J} along with
Corollary~\ref{TracialRelativeISRCorollary}.
For \textup{(iii)}, apply \cite[Corollary~1.3]{HL} and
Corollary~\ref{TracialRelativeISRCorollary}.
\end{proof}

The following embedding result gives further amenable
examples with the relative ISR property.

\begin{corollary}\label{cor:embedding}
Every countable amenable discrete group embeds in a finitely generated amenable group with the relative ISR property. If the given group is elementary amenable, the overgroup can be chosen elementary amenable.
\end{corollary}

\begin{proof}
By \cite[Corollary~1.2\textup{(1)}]{HL}, every countable
(elementary) amenable group embeds into a finitely generated
(elementary) amenable MIF discrete group.
Apply Corollary~\ref{TracialRelativeISRCorollary}.
\end{proof}

\begin{remark}
Amenable discrete groups with ISR were already known.
Jiang--Zhou establish ISR for the group of finitely supported
permutations of $\N$ \cite{JZ}.
\end{remark}

\vspace{3mm}


\bibliography{references}
\bibliographystyle{mybstnum}

\end{document}